\documentclass[francais,a4paper,11pt]{article}

\usepackage[utf8]{inputenc} 
\usepackage[T1]{fontenc}          
\usepackage[francais]{babel}
\usepackage{textcomp}
\usepackage{amsmath,amssymb,amsthm}
\usepackage[all]{xy}
\usepackage{graphics}
\usepackage{graphicx}
\usepackage{color}
\usepackage{enumerate}
\usepackage{mathtools}
\usepackage{dsfont}
\newcommand{\dd}{\displaystyle}

\def\OO{{\mathcal{O}}}
\def\AA{{\mathcal{A}}}
\def\RR{{\mathbb R}}
\def\NN{{\mathbb N}}

\def\CC{{\mathbb C}}

\def\TCC{{\underline{\mathbb C}}}
\def\CCC{{\mathcal{C}}}
\def\CCCC{{\mathcal{C}}}%changer
\def\ZZ{{\mathbb Z}}

\def\MM{{\mathcal{M}}}

\def\PPP{{\mathfrak p}}

\def\DB{{\bar{\partial}}}
\def\DDJ{{\bar{\partial}_J}}

\def\DDJTL{{\bar{\partial}_{J,f_l}^t}}
\def\DDJT{{\bar{\partial}_{J,f}^t}}
\def\DDJTG{{\bar{\partial}_{J,g_l}^t}}

\def\d{{\mathrm d}}

\def\NF{{\mathcal{N}}}

\newcommand{\riso}[1]{\RR Aut(#1)}
\newcommand{\raut}[1]{\RR GL(#1)}
\newcommand{\rsaut}[1]{\RR SL (#1)}

\newcommand{\pinpm}{Pin^{\pm}}
\newcommand{\hczdeux}[2][]{H^1_{#1}(#2,\ZZ/2\ZZ)}
\newcommand{\hhzdeux}[1]{H_1(#1,\ZZ/2\ZZ)}
\newcommand{\hhz}[1]{H_1(#1,\ZZ)}
\newcommand{\hcz}[1]{H^1(#1,\ZZ)}
\newcommand{\rspin}{\RR Spin}
\newcommand{\surg}{\Sigma_g}
\newcommand{\sur}{\Sigma}
\newcommand{\modsig}{\left[(\Sigma_g,c_{\Sigma}),(\CC^*,conj)\right]}
\newcommand{\DBS}[1]{\dd\frac{\DDJ(#1)}{#1}}
\newcommand{\rop}[1][N]{\RR \CCCC (#1)}
\newcommand{\ropj}[1][N]{\RR\CCCC_J(#1)}
\newcommand{\Det}[1][N]{\ddet(#1)}

\newcommand{\cp}[1][1]{\CC P^{#1}}
\newcommand{\rp}[1][1]{\RR P^{#1}}

\renewcommand\Re{\operatorname{\mathfrak{Re}}}

\DeclareMathOperator{\ddet}{Det}
\DeclareMathOperator{\coker}{coker}

\DeclareMathOperator{\ind}{ind}
\DeclareMathOperator{\im}{im}
\DeclareMathOperator{\card}{card}
\DeclareMathOperator{\id}{id}
\DeclareMathOperator{\rang}{rg}
\DeclareMathOperator{\dv}{div}
\DeclareMathOperator{\pd}{PD}
\DeclareMathOperator{\Fix}{Fix}
\DeclareMathOperator{\arf}{Arf}
\DeclareMathOperator{\e}{e}
\DeclareMathOperator{\Hom}{Hom}
\DeclareMathOperator{\Pic}{Pic}

\DeclareMathOperator{\rest}{rest}
\DeclareMathOperator{\res}{rest_2}

\newtheorem{Theoreme}{Théorème}[section]
\newtheorem{Theo}{Théorème}

\newtheorem{Corollaire}{Corollaire}[section]
\newtheorem{Proposition}{Proposition}[section]
\newtheorem{Lemme}{Lemme}[section]
\newtheorem{Definition}{Définition}[section]

\newcounter{Notcount}
\newenvironment{Notation}{\refstepcounter{Notcount}\noindent\ignorespaces {\bf Notation \arabic{Notcount}}. --- }{\bigskip}

\newcounter{Remcount}[section]
\renewcommand{\theRemcount}{\thesection.\arabic{Remcount}}
\newenvironment{Rem}{\refstepcounter{Remcount}\noindent\ignorespaces {\bf Remarque \theRemcount}. --- }{\bigskip}
\newcounter{Explcount}[section]
\renewcommand{\theExplcount}{\thesection.\arabic{Explcount}}
\newenvironment{Exple}{\refstepcounter{Explcount}\noindent\ignorespaces {\bf Exemple \theExplcount}. --- }{\bigskip}

\title{Automorphismes réels d'un fibré et opérateurs de Cauchy-Riemann \\ {\large Real automorphisms of a bundle and Cauchy-Riemann operators}}
\author{Rémi CR\'ETOIS}

\begin{document}

\maketitle

\centerline{\textbf{Résumé}}

Soit $(N,c_N)$ un fibré vectoriel complexe muni d'une structure réelle au-dessus d'une courbe réelle $(\surg,c_\sur)$ de genre $g\in\NN$. Nous étudions le signe de l'action des automorphismes de $(N,c_N)$ relevant l'identité de $\surg$ sur les orientations du fibré déterminant au-dessus de l'espace des opérateurs de Cauchy-Riemann réels de $(N,c_N)$. Ce signe s'obtient comme le produit de deux termes. Le premier calcule la signature des permutations induite par les automorphismes sur les structures $Pin^\pm$ de la partie réelle de $(N,c_N)$. Le second provient de l'action des automorphismes du fibré sur les classes de bordisme de structures $Spin$ réelles de $(\surg,c_\sur)$.

\centerline{\textbf{Abstract}}

Let $(N,c_N)$ be a complex vector bundle equipped with a real structure over a real curve $(\surg,c_\sur)$ of genus $g\in\NN$. We compute the sign of the action of the automorphisms of $(N,c_N)$ lifting the identity of $\surg$ on the orientations of the determinant line bundle over the space of real Cauchy-Riemann operators on $(N,c_N)$. This sign can be obtained as the product of two terms. The first one computes the signature of the permutations induced by the automorphisms acting on the $Pin^\pm$ structures of the real part of $(N,c_N)$. The second one comes from the action of the automorphisms of $(N,c_N)$ on the bordism classes of real $Spin$ structures on $(\surg,c_\sur)$.

\textsc{Classification AMS 2010}: 14H60, 53D45.

\textsc{Mots Clés}: fibrés vectoriels, courbes réelles, opérateurs de Cauchy-Riemann, espaces de modules, invariants de Gromov-Witten.

\tableofcontents

\section*{Introduction}

Soit $\surg$ une surface compacte, connexe, sans bord, orientée et de genre $g\in\NN$. Fixons une structure réelle sur $\surg$, c'est-à-dire une involution $c_\sur$ de $\surg$ qui renverse les orientations. Nous dirons alors que le couple $(\surg,c_\sur)$ est une courbe réelle. Considérons un fibré vectoriel complexe $N$ sur $\surg$. Fixons aussi une structure réelle sur $N$, c'est-à-dire une involution $c_N$ de $N$ induite par un isomorphisme entre $N$ et $\overline{c_\sur^*N}$. L'ensemble $\RR \surg$ des points fixes de $c_\sur$ est une sous-variété de dimension $1$ de $\surg$, éventuellement vide, au-dessus de laquelle l'ensemble $\RR N$ des points fixes de $c_N$ forme un fibré vectoriel réel de rang $\rang (N)$. Prenons un opérateur de Cauchy-Riemann $\DB$ sur $N$ qui est équivariant pour les actions de $c_N$ sur les sections de $N$ et sur les $(0,1)$-formes à valeurs dans $N$ (voir Définition \ref{opcr}). Un tel opérateur induit une unique structure holomorphe sur $N$ pour laquelle $c_N$ est anti-holomorphe (voir \cite{koba}). Le noyau de cet opérateur, noté $H^0_\DB(\surg,N)$, est alors formé des sections holomorphes de $N$, et son conoyau, noté $H^1_\DB(\surg,N)$ est canoniquement isomorphe au premier groupe de cohomologie du faisceau des sections holomorphes de $N$. L'involution $c_N$ induit deux applications $\CC$-antilinéaires sur ces deux espaces; nous noterons avec un indice $+1$ les espaces propres associés à la valeur propre $1$ de ces applications. Notons $\ddet(\DB)$ la droite déterminant du complexe trivial $\{(H^*_\DB(\surg,N)_{+1},\DB)\}$. L'ensemble $\rop$ de tous ces opérateurs de Cauchy-Riemann sur $(N,c_N)$ est un espace contractile, et l'union de toutes les droites déterminant forme un fibré en droites réelles $\Det$ au-dessus de $\rop$ (voir \cite{MDS}). Remarquons que lorsque $N$ est de rang $1$ et de degré congru à $g-1$ modulo $2$ le fibré $\Det$ induit un fibré en droites réelles qui n'est a priori pas orientable sur la composante du groupe de Picard réel de $(\surg,c_\sur)$ contenant les fibrés en droites holomorphes réels de degré $\deg(N)$ et dont la partie réelle a même première classe de Stiefel-Whitney que $\RR N$ (voir \S \ref{picard}). D'autre part, le groupe $\riso{N}$ des automorphismes de $(N,c_N)$ agit naturellement sur $\Det$. Le but de cet article est d'étudier cette action.

Celle-ci apparaît lorsque l'on s'intéresse à l'orientabilité de l'espace de modules $\RR \MM_g^d(X,J)$ des courbes réelles $J$-holomorphes dans une variété symplectique $(X,\omega)$ munie d'une involution $c_X$ anti-symplectique, qui sont de genre $g$ et de degré $d\in H_2(X,\ZZ)$, et où $J$ est une structure presque complexe générique sur $X$ compatible avec $\omega$ rendant $c_X$ anti-holomorphe. En effet, étant donnée une telle courbe immergée $u:(\surg,c_\sur)\rightarrow (X,c_X)$, on a un morphisme de monodromie $\mu:\pi_1(\RR\MM_g^d(X,J),u) \rightarrow \pi_0(\riso{N_u})$, où $N_u$ est le fibré normal à $u$. Alors pout tout $\gamma\in\pi_1(\RR\MM_g^d(X,J),u)$, la première classe de Stiefel-Whitney de $\RR \MM_g^d(X,J)$ calculée le long de $\gamma$ est nulle si et seulement si $\mu(\gamma)$ préserve les orientations de $\Det[N_u]$ (voir \cite{article2}).

Dans le présent article, nous étudions l'action du sous-groupe $\raut{N}$ de $\riso{N}$ formé des automorphismes de $(N,c_N)$ qui valent l'identité sur $\surg$. D'une part, un élément de $\raut{N}$ induit une permutation de l'ensemble à $2b_0(\RR\surg)$ éléments formé des structures $Pin^\pm$ sur $\RR N$. Nous notons $\varepsilon_{\PPP^\pm} : \raut{N} \rightarrow \ZZ/2\ZZ$ le morphisme calculant la signature de ces permutations. D'autre part, nous montrons qu'un élément de $\raut{N}$ agit à travers son déterminant sur les classes de bordisme de structures $Spin$ réelles sur $(\surg,c_\sur)$ (voir \S \ref{actmodspin}). \'Etant donnée une structure complexe sur $(\surg,c_\sur)$, une structure $Spin$ réelle est la donnée d'un fibré en droites holomorphe $L$ sur $(\surg,c_\sur)$, admettant une structure réelle et dont le carré est isomorphe au fibré canonique de $\surg$. Une telle structure admet une première classe de Stiefel-Whitney donnée par $w_1(\RR N)$. Nous dirons alors que deux structures $Spin$ réelles sont bordantes si elles ont même première classe de Stiefel-Whitney et si elles ont même invariant d'Arf. Nous notons alors $\AA^{w}:\raut{N}\rightarrow \{0,1\}$ le morphisme valant $0$ seulement pour les automorphismes préservant les classes de bordisme de structures $Spin$ réelles sur $(\surg,c_\sur)$ de première classe de Stiefel-Whitney $w$. La première partie du résultat principal de cet article est le Théorème suivant.

\begin{Theo}\label{degg}
 Soit $(N,c_N)$ un fibré vectoriel complexe muni d'une structure réelle sur $(\surg,c_\sur)$, de partie réelle non vide. Si $\deg(N) = g+1\mod2$, alors le signe de l'action d'un automorphisme $f\in\raut{N}$ sur les orientations de $\Det$ est donné par le produit $\varepsilon_{\PPP^\pm}(f)(-1)^{\AA^{w_1(\RR N)}(f)}$.
\end{Theo}

Lorsque $\deg(N) = g \mod 2$ et $\rang(N) = 1$, l'automorphisme $-\id_N$ préserve les structures $Spin$ sur $\surg$ mais renverse les orientations de $\Det$ d'après le Théorème de Riemann-Roch, ce qui montre que le Théorème \ref{degg} est faux dans ce cas. Nous définissons alors un nouveau morphisme $s_N : \raut{N}\rightarrow \{0,1\}$ en prenant en compte l'action des automorphismes sur les orientations des composantes orientables de $\RR N$ (voir \S \ref{subenon}). Nous obtenons alors,

\begin{Theo}\label{toutdeg}
  Soit $(N,c_N)$ un fibré vectoriel complexe muni d'une structure réelle sur $(\surg,c_\sur)$, de partie réelle non vide. Le signe de l'action d'un automorphisme $f\in\raut{N}$ sur les orientations de $\Det$ est donné par le produit $\varepsilon_{\PPP^\pm}(f)(-1)^{s_N(f)}$.
\end{Theo}

Ainsi, nous obtenons une interprétation en termes topologiques de l'action de $\raut{N}$ sur les orientations de $\Det$.

 En guise d'application, nous calculons la première classe de Stiefel-Whitney du fibré en droites réelles induit par le fibré déterminant sur le groupe de Picard réel d'une surface de Riemann réelle $(\surg,c_\sur)$ (voir les Théorèmes \ref{picp} et \ref{pic}).

Cet article est organisé de la façon suivante. Nous commençons par exposer quelques préliminaires concernant les éléments de $\raut{N}$. Nous décrivons notamment la classe d'homotopie d'un tel automorphisme (voir \S \ref{stru}). Les parties \ref{partie1} et \ref{automid} sont consacrées à la démonstration des Théorèmes \ref{degg} et \ref{toutdeg}. Nous traitons tout d'abord le cas d'un automorphisme de déterminant $1$ (voir Théorème \ref{action1}). Ceux-ci agissent trivialement sur les structures $Spin$ réelles et leur action sur les orientations de $\Det$ est entièrement décrite par le morphisme $\varepsilon_{\PPP^\pm}$. Nous passons ensuite au cas général. Nous montrons qu'il nous suffit de traiter le cas où $N$ est de rang $1$. Nous décrivons alors l'action d'un élément de $\raut{N}$ sur les structures $Spin$ réelles de $\surg$ puis nous relions celle-ci à l'action de $\raut{N}$ sur les orientations de $\Det$. Pour cela, nous traitons le cas du fibré trivial $(\TCC,conj)$ en décrivant le conoyau d'un opérateur de Cauchy-Riemann réel sur ce fibré puis nous nous ramenons à ce cas en utilisant des transformations élémentaires (voir \S \ref{trans}).

\section{Le groupe des automorphismes d'un fibré réel}\label{stru}

Soit $(\surg,c_\sur)$ une courbe réelle. Les points fixes de $c_\sur$ forment une
sous-variété de dimension $1$
de $\surg$ notée $\RR\surg$, qui n'est pas nécessairement connexe ou non vide (voir \cite{nat}). La lettre $k$ désignera toujours le
nombre de composantes connexes $(\RR\surg)_1,\ldots,(\RR \surg)_k$ de
$\RR \surg$, et nous supposerons à partir de maintenant que $k$ est non nul.

Considérons un fibré vectoriel complexe $(N,c_N)$ muni d'une structure réelle sur $(\surg,c_\sur)$. Nous noterons généralement $(\RR N)_i$ la restriction du fibré $\RR N$ au-dessus de $(\RR \surg)_i$.

Nous appellerons morphisme entre deux fibrés vectoriels complexes
munis de structures réelles $\pi_N: (N,c_N)\rightarrow
(\surg,c_{\sur})$ et $\pi_{N'}: (N',c_{N'})\rightarrow
(\surg',c_{\sur'})$ une paire $(\Phi,\phi)$ composée
\begin{itemize}
\item d'un difféomorphisme $\phi : (\surg,c_{\sur})\rightarrow
  (\surg',c_{\sur'})$ $\ZZ/2\ZZ$-équivariant et préservant les
  orientations
\item d'une application $\Phi : (N,c_N) \rightarrow (N',c_{N'})$
  $\ZZ/2\ZZ$-équivariante, $\CC$-linéaire dans les fibres et faisant
  commuter le diagramme
  \[
\xymatrix{
   (N,c_N)  \ar[d]_{\pi_N} \ar[r]^\Phi & (N',c_{N'}) \ar[d]_{\pi_{N'}} \\
    (\surg,c_{\sur}) \ar[r]_\phi& (\surg',c_{\sur'}). 
}
  \]
\end{itemize}
Notons $\raut{N}$ l'ensemble des automorphismes de $(N,c_N)$ au-dessus de l'identité. Nous ne considérerons dans le présent article que de tels automorphismes, le cas général étant traité dans \cite{}.

Nous rappelons tout d'abord un résultat de classification de ces
fibrés à isomorphisme près qui nous sera utile par la suite (voir par exemple les Propositions 4.1 et 4.2 de \cite{bisw}).

\begin{Lemme}\label{classification}
  Deux fibrés vectoriels complexes munis de structures réelles
  $(N,c_N)$ et $(N',c_{N'})$ sont isomorphes au-dessus de l'identité sur
  $(\surg,c_\sur)$ si et seulement si ils ont même rang, même degré et si leurs
  parties réelles ont même première classe de Stiefel-Whitney.\qed
\end{Lemme}

 Notons $\RR J(\surg)$ l'ensemble des structures complexes de $\surg$ compatibles avec l'orientation fixée et pour lesquelles $c_\sur$ est anti-holomorphe. Cet espace est non vide et contractile (voir
\cite{nat}, \cite{viterbo} et \cite{wel1}).

\begin{Definition}\label{opcr}
 Fixons une structure complexe $J\in \RR J(\surg)$, un entier $k \geq 1$ et un réel $p>1$ tels que $kp>2$. Un opérateur de Cauchy-Riemann réel sur $(N,c_N)$ est un opérateur $\CC$-linéaire
\[
\DB : L^{k,p}(\surg,N) \rightarrow L^{k-1,p}(\surg,\Lambda^{0,1}\surg \otimes N)
\]
équivariant sous l'action de $c_N$ et vérifiant la règle de Leibniz:
\[
\forall f\in\CCC^\infty(\surg,\CC),\ \forall v\in L^{k,p}(\surg,N),\ \DB(fv) = \DB_J(f)\otimes v + f\DB(v),
\]
où $\DB_J = \dd\frac{1}{2}(\d + i\circ\d\circ J)$.
\end{Definition}

Pour tout $J\in\RR J(\surg)$, notons $\ropj$ l'espace des opérateurs de Cauchy-Riemann
sur $(N,c_N)$ au-dessus de $(\surg,c_\sur)$ munie de la structure $J$. Cet ensemble forme un espace affine de dimension infinie
(voir ~\cite{MDS}). Nous noterons d'autre part $\rop$ l'ensemble des paires $(\DB,J)$ où $\DB$ est un élément de $\ropj$.

La structure réelle $c_N$ induit une involution $\CC$-antilinéaire sur les espaces de Banach $L^{k,p}(\surg,N)$ et $L^{k-1,p}(\surg,\Lambda^{0,1}\surg\otimes N)$. Nous notons avec un indice $+1$ (resp. $-1$) les sous-espaces propres associés à la valeur propre $+1$ (resp. $-1$). Un opérateur de Cauchy-Riemann réel $\DB$ induit par restriction un opérateur Fredholm de $L^{k,p}(\surg,N)_{+1}$ dans $L^{k-1,p}(\surg,\Lambda^{0,1}\surg\otimes N)_{+1}$. Nous notons respectivement $H_{\DB}^0(\surg,N)_{+1}$ et $H_{\DB}^{1}(\surg,N)_{+1}$ le noyau et le conoyau de l'opérateur obtenu. Par régularité elliptique, ces deux espaces vectoriels ne dépendent ni de $k$ ni de $p$. Il existe sur l'ensemble $\rop$ un fibré en droites réelles $\Det$ dont la fibre au-dessus de l'opérateur $\DB$ est son déterminant $\ddet(\DB) = \Lambda^{\max}_{\RR} H_{\DB}^0(\surg,N)_{+1} \otimes \left(\Lambda^{\max}_{\RR} H_{\DB}^1(\surg,N)_{+1}\right)^*$ (voir \ref{MDS}). Comme $\rop$ est un espace contractile, le fibré $\Det$ est orientable.

Nous étudions dans les \S\S \ref{partie1} et \ref{automid} de cet article l'action des
automorphismes réels de $(N,c_N)$ relevant l'identité sur le fibré $\Det$ au-dessus de $\rop$. Cette action est définie de la façon suivante. Tout d'abord, les
automorphismes réels de $(N,c_N)$ agissent sur les espaces de sections
de $N$ et de $(0,1)$-formes à valeurs dans $N$ d'une part et d'autre
part sur les structures complexes de $(\surg,c_{\sur})$
respectivement par
\[
\begin{array}{c}
  (\Phi,\phi)_*s = \Phi(s)\circ\phi^{-1}\\
  (\Phi,\phi)^*\alpha = \Phi^{-1}\circ\alpha\circ\d\phi\text{ et}\\
  (\Phi,\phi)^*J = \d\phi^{-1}\circ J \circ \d\phi
\end{array}
\]
où $s$ est une section de $N$, $\alpha$ est une $(0,1)$-forme à valeurs dans
$N$ et $J$ une structure complexe sur la surface (voir
\cite{wel1}). De plus, cette action est $\ZZ/2\ZZ$-équivariante.  Ceci
induit une action sur l'espace $\rop$ 
\[
(\Phi,\phi)^*(\DB,J) =
\left((\Phi,\phi)^*\left(\DB\left((\Phi,\phi).\right)\right),(\Phi,\phi)^*J\right)
\]
et donc sur le fibré $\Det$.

Rien ne garantit qu'en agissant sur ce fibré les automorphismes réels
de $(N,c_N)$ préservent ses orientations. Nous allons donc étudier
plus précisément l'action des éléments de $\raut{N}$ sur les
orientations du fibré $\Det$.

Remarquons tout d'abord que cette action ne dépend que de la classe d'homotopie de l'automorphisme réel de
$(N,c_N)$. Ainsi, nous commençons par décrire la structure des automorphismes réels d'un
fibré vectoriel complexe muni d'une structure réelle $(N,c_N)$ sur
$(\surg,c_\sur)$ pris à homotopie près.

Introduisons pour cela le sous-groupe
\[
\rsaut{N} := \left\{\Phi \in \raut{N}\ |\ \forall x\in\Sigma_g,\
  \det(\Phi_x) = 1\right\}
\]
des éléments de $\raut{N}$ de déterminant $1$.

\begin{Lemme}\label{decomposition}
  Les suites suivantes sont exactes :
  \begin{equation}\label{suite2}
    0 \rightarrow \rsaut{N} \longrightarrow \raut{N} \xrightarrow{\det} \RR\CCC^{\infty}(\surg,\CC^*) \rightarrow 0
  \end{equation}
  \begin{equation}\label{suite3}
    0 \rightarrow \ker(\rest) \longrightarrow \rsaut{N} \xrightarrow{\rest} \pi_0(SL(\RR N)) \rightarrow 0
  \end{equation}
  où :
  \[
  \begin{array}{c c c c}
    \rest : &\rsaut{N} & \rightarrow & \pi_0(SL(\RR N)) \\
    &\Phi      & \rightarrow & [\Phi_{|\RR N}].
  \end{array}
  \]
  De plus, $\ker(\rest)$ est connexe.
\end{Lemme}

 Notons avant de passer à la démonstration que l'application $\det$ est définie comme la
    composée de
    \[
    \begin{array}{c c c}
      \raut{N} & \rightarrow & \raut{{\det(N)}} \\
      \Phi     & \mapsto     & \det(\Phi)
    \end{array}
    \]
    et de l'isomorphisme canonique $\raut{{\det(N)}} =
    \RR\CCC^{\infty}(\surg,\CC^*)$.

\begin{proof}
  \begin{itemize}
  \item[(\ref{suite2})] D'après
    le Lemme \ref{classification} nous avons un isomorphisme
    \begin{equation}\label{decdet}
      (N,c_N) \cong (\det(N),c_{\det(N)})\oplus (\TCC^{\rang(N)-1},conj)
      \tag{$*$}
    \end{equation}
    où $(\TCC^{\rang(N)-1},conj)$ est le fibré trivial de rang
    $\rang(N)-1$ sur $(\surg,c_{\sur})$, et $\det(N) = \Lambda^{\max} N$. Si $f$ est un élément de
    $\RR\CCC^{\infty}(\surg,\CC^*)$, on obtient, grâce à
    l'isomorphisme (\ref{decdet}), un automorphisme de $(N,c_N)$
    $\Phi_f = f\oplus \id_{\TCC^{\rang(N)-1}}$ vérifiant $\det(\Phi_f)
    = f$.
  \item[(\ref{suite3})] \'Etudions tout d'abord la surjectivité de
    $\rest$. Fixons pour cela une classe $[\Psi]$ dans
    $\pi_0(SL(\RR N))$. Choisissons aussi deux ouverts
    emboîtés $\RR V_i\subset\overline{\RR V_i}\subset\RR U_i$
    homéomorphes à des intervalles sur chaque composante
    $(\RR\surg)_i$, et $k$ paires d'ouverts emboîtés disjointes
    $V_1\subset\overline{V_1}\subset U_1,\ldots,V_k\subset
    \bar{V_k}\subset U_k$ de $\surg$ invariants sous l'action de
    $c_\sur$ et homéomorphes à des disques, tels que $U_i\cap \RR\surg
    = \RR U_i$ et $V_i\cap\RR\surg = \RR V_i$. Prenons un représentant
    $\Psi$ de la classe $[\Psi]$ tel que $\Psi$ soit l'identité en
    dehors des $\RR V_i$. En trivialisant $(N,c_N)$ au dessus des
    $U_i$, $\Psi_{|\RR\overline{V_i}}$ nous fournit $k$ chemins dans
    $SL_{\rang(N)}(\RR)$ dont les extrémités se recollent en
    l'identité. Ces lacets peuvent être homotopés dans
    $SL_{\rang(N)}(\CC)$ (qui est simplement connexe) sur le lacet
    constant égal à l'identité. Ainsi, pour chaque $i$ on peut
    prolonger $\Psi$ sur une composante de $U_i - \RR U_i$ puis sur
    tout $U_i$ grâce à la structure réelle de $(N,c_N)$ de sorte que
    ce prolongement soit l'identité sur $U_i - V_i$. On obtient un
    automorphisme réel de $(N,c_N)$ relevant la classe $[\Psi]$ en
    prolongeant en dehors des $U_i$ par l'identité.

    Considérons enfin la connexité du noyau de $\rest$. Fixons
    une triangulation $T$ de $(\surg,c_\sur)$ invariante sous l'action
    de $c_\sur$. Prenons $\Phi$ un élément de $\ker(\rest)$. On peut supposer que la restriction de $\Phi$ à $\RR\surg$ est l'identité. On
    homotope $\Phi$ sur $\id$ progressivement:
    \begin{enumerate}
    \item tout d'abord au-dessus des sommets de $T$,
    \item puis au-dessus des arêtes de $T$, ce que l'on peut faire puisque $\pi_1(SL_{\rang(N)}(\CC))$ est trivial,
    \item et enfin au-dessus des faces de $T$, ce qui est possible car
      $\pi_2(SL_{\rang(N)}(\CC))$ est trivial.
    \end{enumerate}
  \end{itemize}
\end{proof}

\begin{Rem}\label{rem1}
  \begin{itemize}
  \item Si $\rang (N) \geq 3$, $\pi_0(SL(\RR N))\cong
    \left(\ZZ/2\ZZ\right)^k$, où $k = b_0(\RR \Sigma_g)$.
  \item Le cas $\rang (N) = 2$ est particulier. Le Lemme \ref{tauid}
    en donne une description.
  \item Lorsque $\rang(N) = 1$, l'application $\det:\raut{N} \rightarrow \RR\CCC^{\infty}(\surg,\CC^*)$ dans la suite exacte (\ref{suite2}) est en fait un isomorphisme, et le groupe $\rsaut{N}$ est trivial.
  \end{itemize}
\end{Rem}

\begin{Lemme}\label{tauid}
  Soit $(N,c_N)$ un fibré vectoriel complexe de rang $2$ muni d'une
  structure réelle sur $(\surg,c_\sur)$. Alors, le groupe $\pi_0(SL(\RR N))$ est isomorphe à
  $\ZZ^{k_+}\times\left(\ZZ/2\ZZ\right)^{k_-}$,  
  où 
\[
\begin{array}{c}
k_+ = \card\left\{1\leq i\leq k\ |\ w_1(\RR N)([\RR\surg]_i) =
    0\right\}\\
 \text{et } k_- = \card\left\{1\leq i\leq k\ |\ w_1(\RR
    N)([\RR\surg]_i) \neq 0\right\}.
\end{array}
\]
 De plus, avec les notations du
  Lemme \ref{decomposition} et sous cet isomorphisme, l'image de
  $\rest(-\id)\in\pi_0(SL(\RR N))\cong\ZZ^{k_+}\times\left(\ZZ/2\ZZ\right)^{k_-}$ par l'application de réduction modulo $2$ $\ZZ^{k_+}\times\left(\ZZ/2\ZZ\right)^{k_-}\rightarrow (\ZZ/2\ZZ)^k$ est égale à $(w_1(\RR N)([\RR \surg]_1),\ldots,w_1(\RR
  N)([\RR\surg]_k))$.
\end{Lemme}

\begin{proof}
  Fixons $1\leq i\leq k$. Nous distinguons deux cas, suivant que $(\RR
  N)_i\rightarrow (\RR \surg)_i$ est orientable ou non.
  \begin{itemize}
  \item[Si $w_1(\RR N)(\text{[}\RR\surg\text{]}_i) = 0$:] Nous pouvons
    dans ce cas trivialiser le fibré $(\RR N)_i$. Dans cette
    trivialisation, un automorphisme du fibré à homotopie près est
    donné par une classe dans le groupe fondamental de
    $SL_2(\RR)$. Or, ce groupe fondamental est isomorphe à $\ZZ$. Dans
    le cas particulier de $-\id$, nous obtenons le lacet constant égal
    à $\left(\begin{array}{c c} -1 & 0 \\ 0 & -1 \end{array}\right)$,
    qui est homotope au lacet constant égal à l'identité.
  \item[Si $w_1(\RR N)(\text{[}\RR\surg\text{]}_i) = 1$:] Le raisonnement
    employé ici s'inspire de \cite{Solomon}. Le fibré $(\RR N)_i$ est
    isomorphe au fibré $[0,1]\times \RR^2/(0,v)\sim(1,r(v))$ où
    $r(x_1,x_2) = (x_1,-x_2)$. Sous cet isomorphisme, un élément de
    $SL((\RR N)_i)$ correspond à une application $f :
    [0,1] \rightarrow SL_2(\RR)$ vérifiant $f(0) = rf(1)r$. Cette dernière condition se réécrit grâce à la décomposition polaire, produit d'une rotation par une triangulaire supérieure à coefficients diagonaux strictement positifs,
\[
f(t) = r_{\theta(t)}
\left(
  \begin{array}{c c}
    \lambda(t) & a(t) \\
    0          & \frac{1}{\lambda(t)}
  \end{array}
\right)
\]
en $\theta(0) = -\theta(1) \mod 2\pi$ et $a(0) = -a(1)$. D'autre part, cette décomposition fournit aussi une rétraction de $SL_2(\RR)$ sur $SO_2(\RR)$
\[
\begin{array}{c c c c}
h : & [0,1]\times SL_2(\RR) & \rightarrow & SL_2(\RR)\\
    &\left(s,r
\left(\begin{array}{c c}
  \lambda & a \\
  0       & \frac{1}{\lambda}
\end{array}\right)\right)
                           & \mapsto  & r 
\left(\begin{array}{c c}
  (1-s)\lambda + s & (1-s)a \\
  0       & \frac{1}{(1-s)\lambda +s}
\end{array}\right)
\end{array}
\]
qui préserve les conditions précédentes.
En appliquant cette rétraction, nous obtenons une homotopie $h(s,f)$ joignant $f$ à $r_{\theta}$.
Ainsi, une classe d'homotopie d'éléments de
    $SL((\RR N)_i)$ correspond à une classe d'homotopie
    d'applications $\theta:[0,1]\rightarrow \RR/2\pi\ZZ$ avec
    $\theta(0) = -\theta(1)\mod(2\pi)$, c'est-à-dire à une classe d'homotopie de sections d'une bouteille de Klein (vue comme fibré en cercles au-dessus du cercle). Or,
    il n'y a que deux telles classes qui sont définies par les lacets $\theta(t) = 0$ et $\theta(t) = \pi$. Le premier correspond à $\id$, le second à $-\id$.   
  \end{itemize}
\end{proof}

\begin{Notation}
On note $\res$ la composée de $\rest$ avec la réduction modulo $2$
\[
\pi_0(SL(\RR N)) \rightarrow \pi_0(SL(\RR N))/2\pi_0(SL(\RR N)) \cong (\ZZ/2\ZZ)^k,
\]
 et on notera avec un indice $i$ entre $1$ et $k$ les composantes de $\res$.
\end{Notation}

\section{\'Etude des automorphismes réels de déterminant $1$}\label{partie1}

Nous nous intéressons dans cette section à l'action du groupe
$\rsaut{N}$ des automorphismes réels de déterminant $1$ d'un fibré
$(N,c_N)$ sur le fibré $\Det$.

Comme nous l'avons noté dans la Remarque \ref{rem1}, lorsque le rang
de $N$ est égal à $1$ le groupe $\rsaut{N}$ est
trivial et agit donc trivialement sur le fibré $\Det$. 
Nous supposerons dans cette section que $N$ est de rang
au moins deux.

D'après le Lemme \ref{decomposition}, cette action se factorise à travers le morphisme $\res$ et nous commençons par réinterpréter ce morphisme en termes
d'action des automorphismes sur les structures $Pin^{\pm}$ de $\RR
N$. Nous étudions ensuite l'action des
automorphismes de déterminant $1$ sur le fibré $\Det$ et nous démontrons en particulier le Théorème \ref{action1} qui est un cas particulier du Théorème \ref{toutdeg}.

\begin{Theoreme}\label{action1}
  Soit $t$ un automorphisme de $(N,c_N)$ relevant l'identité et de déterminant $1$. Alors $t$ préserve les orientations du fibré $\Det$ si et seulement si le nombre de composantes de $\RR N$ dont les structures $Pin^\pm$ sont échangées sous l'action de $t$ est pair.
\end{Theoreme}

Celui-ci fournit une interprétation topologique de l'action de $\rsaut{N}$ sur les orientations du fibré $\Det$. En guise d'application, nous donnons une condition nécessaire et suffisante pour que le fibré déterminant au-dessus d'un lacet d'opérateurs de Cauchy-Riemann soit orientable.

\subsection{Automorphismes réels et structures $Pin^{\pm}$ sur $\RR
  N$}\label{partiepin}
Commençons par quelques rappels au sujet des structures
$Pin^{\pm}$. Fixons un fibré vectoriel réel $V$ de rang
$n$ au moins $2$ au-dessus d'une base $B$ et notons $R_{V}$ le fibré principal des repères
associé. Ce fibré a pour groupe structural $GL_n(\RR)$; celui-ci a un revêtement
double (topologique) non trivial au-dessus de chaque composante connexe qui admet deux structures de groupe distinctes
relevant celle de $GL_n(\RR)$. Ces deux groupes sont notés
$\widetilde{GL_n}^+(\RR)$ et $\widetilde{GL_n}^-(\RR)$. Une structure
$\widetilde{GL_n}^{\pm}$ sur $V$ est la donnée d'une classe d'isomorphisme de paires
composées d'un fibré principal $P_{\widetilde{GL}^{\pm}}$ sur $B$ de
groupe structural $\widetilde{GL_n}^\pm(\RR)$ et d'un morphisme de fibrés
$P_{\widetilde{GL}^{\pm}}\rightarrow R_{V}$, revêtement double
équivariant pour les actions de $GL_n(\RR)$ et
$\widetilde{GL_n}^{\pm}(\RR)$.

\begin{Rem}\label{pinmetrique}
 La
  définition que nous donnons ici ne diffère que peu de la
  définition usuelle des structures $\pinpm$ utilisant les groupes $O_n(\RR)$ et
  $\pinpm_n(\RR)$ (voir par exemple \cite{kirby}) au lieu de $GL_n(\RR)$ et $\widetilde{GL_n}^\pm(\RR)$, car $GL_n(\RR)$
  se rétracte sur $O_n(\RR)$. Nous avons une correspondance naturelle entre les structures $\pinpm$ et $\widetilde{GL_n}^{\pm}$. Ceci nous permet donc de parler de structures $\pinpm$ sans se soucier de fixer de métrique sur les fibrés considérés.
\end{Rem}

Pour terminer, notons que l'obstruction à l'existence d'une structure
$Pin^+$ (resp. $Pin^-$) sur le fibré $V$ est donnée par la classe caractéristique
$w_2(V)$ (resp. $w_2(V) + w_1^2(V)$) (voir \cite{kirby}).

Les fibrés $(\RR N)_i$ ayant pour bases des cercles, nous avons
$w_2((\RR N)_i) = w_1^2((\RR N)_i) = 0$. Ils admettent donc un
ensemble de structures $\pinpm$ que nous noterons $\pinpm((\RR
N)_i)$. Ces ensembles sont des espaces affines sur
$\hczdeux{(\RR\surg)_i}$ (voir \cite{kirby}).

D'autre part, les éléments de $\rsaut{N}$ agissent sur ces structures
par tiré-en-arrière: pour $\Phi\in\rsaut{N}$, $P$ et $P'$ deux
structures $\pinpm$ sur $(\RR N)_i$, nous avons $\Phi^*(P') = P$ s'il
existe un morphisme $\tilde{\Phi} : P \rightarrow P'$ tel que le
diagramme suivant commute
\[
\xymatrix{
 P \ar[d] \ar[r]^{\tilde{\Phi}} & P'\ar[d]  \\
 R_{(\RR N)_i} \ar[r]_{\Phi_{|(\RR N)_i}} & R_{(\RR N)_i}.
}
\]
Cette action est en fait une action par translation sur $\pinpm((\RR
N)_i)$ et fournit un morphisme
\[
\tau :\rsaut{N} \rightarrow
\hczdeux{\RR\surg} = (\ZZ/ 2\ZZ)^k.
\]

Nous rappelons enfin le lemme suivant (voir \cite{Solomon}, Lemme $2.6$) qui relie
l'action des automorphismes réels de $N$ sur les structures $\pinpm$
de $\RR N$ et les classes d'homotopie de leurs restriction à $\RR N$.

\begin{Lemme}\label{homotopin}
  Les applications $\res$ et $\tau$ sont égales.\qed
\end{Lemme}

 Plus généralement,
  si $M$ est un fibré vectoriel réel de rang au moins $2$ sur un
  cercle, alors
  \begin{itemize}
  \item lorsque $M$ n'est pas orientable, il y a exactement deux
    classes d'homotopie d'automorphismes de déterminant $1$ de $M$;
    l'une (celle de l'identité) préserve les structures $\pinpm$ de
    $M$ l'autre les échange.
  \item lorsque $M$ est orientable et de rang au moins $3$, il y a
    exactement deux classes d'homotopie d'automorphismes de déterminant
    $1$ de $M$; l'une préserve les structures
    $\pinpm$ de $M$ l'autre les échange.
  \item lorsque $M$ est orientable et de rang $2$, les classes
    d'homotopie d'automorphismes de déterminant $1$ de $M$ sont toutes
    multiples de celle provenant d'un générateur de $\pi_1(SO_{2}(\RR))$; les multiples pairs préservent les structures
    $\pinpm$ de $M$, les multiples impairs les échangent.
  \end{itemize}

\subsection[Action des automorphismes réels de déterminant $1$]{Action des automorphismes réels de déterminant $1$ sur les
  orientations du fibré déterminant}

Prenons tout d'abord une structure complexe $J\in \RR J(\surg)$
qui sera fixée pendant tout ce paragraphe.

Nous allons établir le Lemme \ref{refl} avant de démontrer le
Théorème \ref{action1}. Si $(L,c_L)$ est
un fibré en droites holomorphe muni d'une structure réelle
sur $(\surg,c_\sur)$, nous noterons $\MM(L)_{+1}$ l'espace des sections
méromorphes de $L$ qui sont $\ZZ/2\ZZ$-équivariantes. Pour une telle
section $\sigma\in\MM(L)_{+1}$, nous noterons respectivement
$\dv_0(\sigma)$ et $\dv_\infty(\sigma)$ les diviseurs de zéros et de
pôles de $\sigma$.

Si $(M,c_M)$ est un autre fibré en droites holomorphe sur
$(\surg,c_\sur)$ alors une section $\sigma\in\MM(M\otimes L^*)_{+1}$
définit un nouveau fibré en droites holomorphe $(F_\sigma =
\text{graphe}(\sigma),c_F)$ de la façon suivante. Si on note
$U_{0,\sigma} = \surg-\dv_\infty(\sigma)$ et $U_{\infty,\sigma} =
\surg - \dv_0(\sigma)$, alors
\[
F_{|U_0} = \{(v,\sigma(v)),\ v\in L\}\subset L\oplus M,\ F_{|U_{\infty}} =
\{(\sigma^*(w),w),\ w \in M\}\subset L\oplus M.
\]
Ici $\sigma^*$ est l'élément de $\MM(L\otimes M^*)_{+1}$ vérifiant
$\sigma^* \sigma = 1$. D'autre part, $F_{\sigma}\subset L\oplus M$ est
naturellement muni de sa structure réelle $c_F$ venant de celle de
$L\oplus M$. On remarque de plus que
\[
\begin{array}{l}
  F_x = L_x,\text{ quand } x\in \dv_0(\sigma)\\
  F_x = M_x,\text{ quand } x\in \dv_{\infty}(\sigma).
\end{array}
\]

 Pour tout supplémentaire complexe (non nécessairement holomorphe) $(G,c_G)\subset (L\oplus M,c_{L\oplus M})$ de $F_\sigma$, on note
  respectivement $r_{F_\sigma}, r_L, r_M$ les réflexions par rapport à
  $G,M,L$ et d'axes $F_\sigma,L,M$, et $t_{L,\sigma} =
  r_{F_\sigma}\circ r_L$, $t_{M,\sigma} = r_{F_\sigma}\circ r_M$. Les classes d'homotopie de $t_{L,\sigma}$ et $t_{M,\sigma}$ ne dépendent pas du choix de $(G,c_G)$. 

Le lemme suivant étudie un cas particulier de l'action qui nous
intéresse.

\begin{Lemme}\label{refl}
  Soient $(L,c_L)$ et $(M,c_M)$ deux fibrés en droites holomorphes
  munis de structures réelles sur $(\Sigma_g,c_\Sigma)$. Fixons
  $\sigma\in \MM(M\otimes L^*)_{+1}$, et $(F_\sigma =
  \text{graphe}(\sigma),c_F)\subset (L\oplus M,c_{L\oplus M})$. Choisissons un supplémentaire complexe $(G,c_G)$ de $F_\sigma$. On a
  alors:
  \begin{enumerate}[(i)]
  \item $t_{L,\sigma}$ (resp. $t_{M,\sigma}$) préserve les
    orientations du fibré $\Det[L\oplus M]$ si et seulement si
    $\deg(L) - \deg(F_\sigma)$ (resp. $\deg(M) - \deg(F_\sigma)$) est
    pair.
  \item $\deg(L) = \deg(F_\sigma) + \card(\dv_\infty(\sigma))$ et
    $\deg(M) = \deg(F_\sigma) + \card(\dv_0(\sigma))$.
  \item $(\res)_i(t_{L,\sigma}) =
    \card(\dv_\infty(\sigma_{|\RR\Sigma_i})) \mod 2$ et
    $(\res)_i(t_{M,\sigma}) = \card(\dv_0(\sigma_{|\RR\Sigma_i})) \mod
    2$.
  \end{enumerate} 
\end{Lemme}

Remarquons avant de le démontrer que ce lemme implique le Théorème \ref{action1} pour les automorphismes
  $t_{F,\sigma}$ et $t_{M,\sigma}$. Ces trois points nous assurent en effet que les automorphismes
  $t_{L,\sigma}$ et $t_{M,\sigma}$ préservent les orientations du
  fibré $\Det[L\oplus M]$ si et seulement si
  $\dd\sum_{i=1}^k(\res)_i(t_{L,\sigma})$ et
  $\dd\sum_{i=1}^k(\res)_i(t_{M,\sigma})$ sont pairs. En appliquant le
  Lemme \ref{homotopin} nous obtenons le Théorème \ref{action1} dans ce cas.

\begin{proof}
  \begin{enumerate}[(i)]
  \item Nous démontrons ce premier point seulement pour
    $t_{L,\sigma}$, le résultat pour $t_{M,\sigma}$ se traitant de
    façon tout à fait analogue. Notons $\DB_L$ (resp. $\DB_{L\oplus
      M}$) l'opérateur de Cauchy-Riemann sur $L$ (resp. sur $L\oplus
    M$) associé à la structure holomorphe du fibré (voir
    \cite{koba}). L'automorphisme $-1 : L \rightarrow L$ fixe
    l'opérateur $\DB_L$ et agit sur la droite $\ddet (\DB_L)$ par
    multiplication par $(-1)^{\dim(\ker(\DB_L)) +
      \dim(\coker(\DB_L))}$. Ainsi, l'automorphisme $r_L : L\oplus
    M\rightarrow L\oplus M$ fixe l'opérateur $\DB_{L\oplus M}$ et agit
    sur la droite $\ddet (\DB_{L\oplus M}) = \ddet(\DB_L) \otimes
    \ddet(\DB_M)$ par multiplication par $(-1)^{\dim(\ker(\DB_L)) +
      \dim(\coker(\DB_L))}$. Donc $r_L$ préserve les orientations de
    la droite $\ddet(\DB_{L\oplus M})$ (et donc du fibré $\Det[L\oplus
    M]$) si et seulement si $\dim(\ker(\DB_L)) + \dim(\coker(\DB_L))$
    est pair. Or, le Théorème de Riemann-Roch affirme que
    $\dim(\ker(\DB_L)) - \dim(\coker(\DB_L)) = \deg(L) + 1 - g$.

    Fixons ensuite un opérateur de Cauchy-Riemann réel $\DB_G$ sur $G$
    et appliquons le même raisonnement que précédemment aux fibrés
    $F_\sigma$ et $G$. Nous voyons alors que l'automorphisme
    $r_{F_\sigma}$ préserve les orientations de la droite
    $\ddet(\DB_{F_{\sigma}})\otimes \ddet(\DB_G)$ (et donc du fibré
    $\Det[F_{\sigma}\oplus G] = \Det[L\oplus M]$) si et seulement si
    $\deg(F_\sigma) + 1 - g$ est pair.

    Ainsi, l'automorphisme $t_{L,\sigma} = r_{F_\sigma}\circ r_L$
    préserve les orientations du fibré $\Det[L\oplus M]$ si et
    seulement si $\deg(L) + \deg(F_\sigma) + 2 - 2g$ est pair. D'où le
    premier point.
  \item Soit une section $\sigma_L \in \MM(L)_{+1}$ telle que ses pôles
    et zéros soient disjoints de ceux de $\sigma$. Posons $\sigma_F =
    (\sigma_L, \sigma(\sigma_L)) \in \MM(F)_{+1}$, et notons
     $pr_M : L\oplus
    M \rightarrow M$ la projection naturelle. Alors
    \[
    \begin{array}{l}
      \dv_0(\sigma_L) = \dv_0(\sigma_F)\\
      \dv_{\infty}(\sigma_L) = \dv_{\infty}(\sigma_F) - \dv_{\infty}(\sigma)
    \end{array}
    \]
    et
    \[
    \begin{array}{l}
      \dv_0(pr_M(\sigma_F)) = \dv_0(\sigma_F) + \dv_0(\sigma)\\
      \dv_{\infty}(pr_M(\sigma_F)) = \dv_{\infty}(\sigma_F).
    \end{array}
    \]
    D'où en soustrayant ces égalités, $\deg(L) = \deg(F_\sigma) +
    \card(\dv_{\infty}(\sigma))$ et $\deg(M) = \deg(F_\sigma) +
    \card(\dv_0(\sigma))$.
  \item Nous distinguons deux cas, suivant que $(\RR L\oplus \RR M)_i$
    est orientable ou non.

    Si $(\RR L\oplus \RR M)_i$ est orientable, on le trivialise. Les
    automorphismes $t_{L,\sigma}$ et $t_{M,\sigma}$ induisent alors
    deux lacets dans $SL_2(\RR)$. On calcule $(\res)_i(t_{L,\sigma})$ et
    $(\res)_i(t_{M,\sigma})$ en comptant la parité des classes d'homotopie des lacets précédents
    dans $\pi_1(SL_2(\RR))\cong \ZZ$. Or ces parités se mesurent en comptant le nombre de demi-tours que
    fait $(\RR F)_i$ respectivement par rapport à $(\RR L)_i$ et $(\RR
    M)_i$. En effet, en prenant une trivialisation telle que $(\RR L)_i$ et $(\RR M)_i$ soient orthogonaux, et en homotopant $(\RR G)_i$ sur l'orthogonal de $(\RR F_\sigma)_i$ ces
    deux lacets dans $SL_2(\RR)$ sont en fait homotopes à des lacets de
    rotations d'angles égaux respectivement au double de l'angle entre
    $(\RR L)_i$ et $(\RR F_\sigma)_i$ et au double de l'angle entre
    $(\RR M)_i$ et $(\RR F_\sigma)_i$.

    Nous avons donc $(\res)_i(t_{L,\sigma})=
    \card(\dv_0(\sigma_{|\RR\sur_i})) \mod 2$ et $(\res)_i(t_{M,\sigma})=
    \card(\dv_\infty(\sigma_{|\RR\sur_i})) \mod 2$. On conclut
    en remarquant que dans ce cas \newline $\card(\dv_\infty(\sigma_{|\RR\Sigma_i})) =
    \card(\dv_0(\sigma_{|\RR\Sigma_i})) \mod 2$.

    Si maintenant $(\RR L\oplus \RR M)_i$ n'est pas orientable, deux
    cas se
    présentent:
    \begin{itemize}
    \item soit on peut homotoper le fibré $(\RR F)_i$ sur le fibré
      $(\RR L)_i$ (si $\card(\dv_\infty(\sigma_{|\RR\Sigma_i})) = 0
      \mod 2$) et alors, d'après le Lemme \ref{tauid},
      $(\res)_i(t_{L,\sigma}) = (\res)_i(\id) = 0 =
      \card(\dv_\infty(\sigma_{|\RR\Sigma_i})) \mod 2$ et
      $(\res)_i(t_{M,\sigma}) = (\res)_i(-\id) = 1 =
      \card(\dv_0(\sigma_{|\RR\Sigma_i})) \mod 2$.
    \item soit on peut homotoper $(\RR F)_i$ sur $(\RR M)_i$ (si
      $\card(\dv_0(\sigma_{|\RR\Sigma_i})) = 0 \mod 2$) et alors
      $(\res)_i(t_{L,\sigma}) = (\res)_i(-\id) = 1 =
      \card(\dv_\infty(\sigma_{|\RR\Sigma_i}))$ et
      $(\res)_i(t_{M,\sigma}) = (\res)_i(\id) = 0 =
      \card(\dv_0(\sigma_{|\RR\Sigma_i}))$.
    \end{itemize}
    Ce qui conclut la démonstration du troisième point.
  \end{enumerate}
 \end{proof}

\begin{proof}[Démonstration du Théorème \ref{action1}]
  Remarquons tout d'abord qu'il nous suffit de traiter le cas $\rang
  (N) = 2$. En effet, pour tout sous-fibré $(N',c_{N'})$ de $N$ de rang
  $2$, l'application
  \[
  \begin{array}{c c c}
    \rsaut{N'} & \rightarrow & \rsaut{N}\\
    \tilde{t}  & \mapsto     & \tilde{t} \oplus \id\\
  \end{array}
  \]
  commute avec $\tau_N$ et $\tau_{N'}$ et induit une surjection au
  niveau des classes d'homotopie.

  Supposons donc que $\rang(N) = 2$ et prenons $t\in\rsaut{N}$. Fixons
  de plus un opérateur de Cauchy-Riemann induisant une structure
  holomorphe sur $N$ faisant de ce dernier un fibré vectoriel
  holomorphe décomposable en une somme de deux fibrés en droites
  holomorphes $(L,c_L)$ et $(M,c_M)$.

  Prenons $\sigma_{M\otimes L^*}$ une section méromorphe et
  $\ZZ/2\ZZ$-équivariante de $M\otimes L^*$. Pour toute fonction
  méromorphe $f$ sur $\Sigma_g$, on obtient une nouvelle section $\sigma =
  f\sigma_{M\otimes{L}^*}$ dont on contrôle la parité des pôles et des
  zéros sur chaque composante réelle de $\Sigma_g$ en modifiant les
  pôles et les zéros de $f$. On peut en effet considérer $f$ comme
  section holomorphe d'un fibré $\OO_{\Sigma}(x_1+\ldots+x_n)$ (pour
  $n$ assez grand) où l'on place les points $x_i$ sur $\RR \surg$ pour
  changer à volonté la parité de ses zéros et pôles (qui est la même
  sur chaque composante de $\RR\Sigma_g$). En particulier, prenons $f$
  de telle sorte que $(\res)_i(t) =
  \card(\dv_0(\sigma_{|\RR\sur_i}))\mod 2$ pour tout $1\leq i \leq
  k$. Le Lemme \ref{refl} nous assure alors que $\res(t) =
  \res(t_{L,\sigma})$, donc $t$ et $t_{L,\sigma}$ ont même action sur
  les structures $\pinpm$ sur $\RR N$ d'après le Lemme
  \ref{homotopin}.
  D'autre part, l'action d'un automorphisme sur les orientations du
  fibré $\Det$ ne dépend que de son image par $\res$ et non par
  $\rest$. La conclusion du Théorème découle ainsi du
  Lemme \ref{refl}.
\end{proof}

\subsection{Lacets d'opérateurs de Cauchy-Riemann et structures
  $\pinpm$}

Nous décrivons ici l'orientabilité du fibré déterminant au-dessus de certains lacets d'opérateurs de
Cauchy-Riemann réels. Pour ce faire, supposons fixée une structure
complexe $J\in\RR J(\surg)$. Considérons une famille $N_z \rightarrow
\CC P^1$, $z\in\CC P^1$, de fibrés vectoriels complexes sur $(\surg,c_\sur)$. Nous
supposerons de plus que les fibrés $N_z$ admettent des structures
réelles $c_{N_z}$ lorsque $z$ est dans $\RR P^1$.

Nous disposons donc de $k$ familles $(\RR N_z)_i$, $z\in\RR P^1$, de
fibrés vectoriels réels sur $\RR \surg$. \`A $z\in\RR P^1$ fixé,
chacun de ces fibrés $(\RR N_z)_i$ admet une structure $\pinpm$;
toutefois, lorsque $z$ varie, il n'est pas garanti qu'il existe une
famille continue de telles structures sur tout $\RR P^1$. Si c'est le cas, nous
dirons que la famille $(\RR N_z)_i$ admet une structure $\pinpm$.

Si nous prenons maintenant un lacet $(\DB_z)_{z\in\RR P^1}$
d'opérateurs de Cauchy-Riemann réels sur $(N_z,c_{N_z})_{z\in\RR
  P^1}$, nous obtenons naturellement un fibré en doites réelles
$\ddet(\DB_z)\rightarrow \RR P^1$. Nous dirons que le lacet $(\DB_z)_{z\in\RR P^1}$ est
orientable si le fibré $\ddet(\DB_z)$ est orientable.

\begin{Exple}
  Un exemple de tel lacet est obtenu de la façon suivante. Considérons
  une extension $0\rightarrow (F,c_F)\rightarrow (M,c_{M}) \rightarrow (G,c_G) \rightarrow 0$, où $F,G$ (resp. $M$) sont des
  fibrés en droites (resp. de rang $2$) holomorphes sur $\Sigma_g$. Celle-ci admet une
  classe d'extension $\mu \in H^1(\Sigma_g, F\otimes G^*)_{+1}$. Comme
  $H^1(\Sigma_g, F\otimes G^*)_{+1} =  \Gamma
  (\Sigma_g,\Lambda^{0,1}\Sigma_g \otimes (F\otimes G^*))_{+1} /
  \im(\DB_{F\otimes G^*})_{+1}$, on notera $\tilde{\mu}$ un
  représentant de $\mu$ dans $ \Gamma
  (\Sigma_g,\Lambda^{0,1}\Sigma_g \otimes (F\otimes G^*))_{+1}$. La classe
  $\mu$ induit une déformation $(N_z)_{z\in\CC}$, de fibre
  exceptionnelle $N_0 = F\oplus G$ et trivialisable sur $\CC^*$, de
  fibre générique $N_z \cong M$, $z\in \CC^*$. Cette déformation s'étend
  donc à tout $\CC P^1$, et on obtient un lacet d'opérateurs donnés
  par $\DB_x = \left(\begin{array}{c c} \DB_F & x\mu \\ 0 &
      \DB_G \end{array}\right)$, $x\in \RR$. L'orientabilité de ce
  lacet est alors décrite par le Théorème \ref{familles}.
\end{Exple}

Le résultat suivant donne une condition topologique sur
l'orientabilité des lacets d'opérateurs considérés (comparer avec la Proposition 8.1.7 de \cite{fooo}).

\begin{Theoreme}\label{familles}
  Fixons une structure complexe $J$ sur $(\surg,c_\sur)$. Soit $N_z
  \rightarrow \CC P^1$, $z\in \CC P^1$, une famille de fibrés vectoriels complexes de rang au moins $2$ sur
  $(\surg,c_\sur)$ admettant des structures réelles $c_{N_z}$ lorsque
  $z\in \RR P^1$. Supposons de plus qu'il existe un point réel $p\in\RR\surg$ tel que le fibré $(\RR N_z)_p$ sur $\RR P^1$ soit orientable. Alors, un lacet $(\DB_z)_{z\in\RR P^1}$ d'opérateurs
  de Cauchy-Riemann réels sur $(N_z)_{z\in\RR P^1}$ est orientable si
  et seulement si les familles $(\RR N_z)_i$ qui n'admettent pas de
  structure $\pinpm$ sont en nombre pair.
\end{Theoreme}
Faisons tout d'abord quelques remarques préliminaires à la
démonstration de ce Théorème. Fixons les deux ouverts de cartes affines $0\in U_0$ et $\infty\in U_1$ de $\CC P^1$. Les familles $(N_z,c_{N_z}) \rightarrow \RR P^1$, $z\in\RR P^1$, de fibrés vectoriels complexes munis de structures réelles sur $(\Sigma_g,c_\sur)$ sont classifiées par $H^1(\RR P^1,\RR GL(N))$ (où $(N_z,c_{N_z}) \cong (N,c_N)$ pour $z\in \RR
P^1$). Une telle famille est donnée par une application de recollement
\[
\left\{
  \begin{array}{c c c}
    \RR U_0\cap \RR U_1 \times N & \rightarrow & \RR U_0\cap \RR U_1 \times N \\
    (z , \xi)            & \mapsto     & (z, t_z(\xi))
  \end{array}
\right.
\]
où $t_z = \id_N$ pour $z > 0$ et $t_z\in \RR GL(N)$ pour $z<0$. Toutefois, si l'on impose que cette famille s'étende sur tout $\CC P^1$, il faut que $t_z$ pour $z<0$ soit homotope dans $GL(N)$ à $\id_N$, et on obtient alors un prolongement donné par l'application de recollement
\[
\left\{
  \begin{array}{c c c}
    U_0\cap U_1 \times N & \rightarrow & U_0\cap U_1 \times N \\
    (z , \xi)            & \mapsto     & (z, t_z(\xi))
  \end{array}
\right.
\]
où $(t_z)_{z\in\CC^*}$ est une famille d'éléments de $GL(N)$ prolongeant $(t_z)_{z\in\RR^*}$. Or, l'existence d'une telle homotopie est équivalente à ce que le déterminant $\det(t_z)\in \RR\CCC^\infty(\surg,\CC^*)$ soit homotope à la fonction constante $1$ dans l'espace $\CCC^\infty(\surg,\CC^*)$ pour tout $z\in\RR^*$. 
 Si nous supposons de plus qu'il existe un point réel $p\in\RR\surg$ tel que le fibré $(\RR N_z)_p$ sur $\RR P^1$ soit orientable alors forcément $\det(t_z)$ pour $z<0$ est homotope dans $\RR\CCC^\infty(\surg,\CC^*)$ à la fonction constante égale à $1$.

Nous pouvons ainsi supposer que les familles considérées sont données par des applications de transition qui, restreintes à $\RR P^1$ sont dans $\rsaut{N}$.

\begin{proof}[Démonstration du Théorème \ref{familles}]
Prenons une famille $N_z \rightarrow \CC P^1$ de fibrés vectoriels complexes sur $(\surg,c_\sur)$ comme décrit dans l'énoncé du Théorème \ref{familles}. Elle est donnée par une application de recollement
\[
\left\{
  \begin{array}{c c c}
    U_0\cap  U_1 \times N & \rightarrow &  U_0\cap  U_1 \times N \\
    (z , \xi)            & \mapsto     & (z, t_z(\xi))
  \end{array}
\right.
\]
où lorsque $z\in\RR^*$, $t_z = \id_N$ pour $z>0$ et $t_z \in \rsaut{N}$ pour $z<0$.

  D'une part, une des familles $(\RR N_z)_i$, $z\in\RR P^1$, admet une structure
  $\pinpm$ si et seulement si les $t_z$, $z\in\RR^*$, agissent trivialement sur les structures
  $\pinpm$ sur $(\RR N)_i$.

  D'autre part, le fibré déterminant associé à un lacet
  $(\DB_z)_{z\in\RR P^1}$ d'opérateurs de Cauchy-Riemann réels sur
  $(N_z)_{z\in\RR P^1}$ est donné par les changements de cartes
  \[
  \left\{
    \begin{array}{c c c}
      U_0\cap U_1 \cap \RR P^1 \times \RR & \rightarrow & U_0\cap U_1 \cap \RR P^1 \times \RR \\
      (z,v)                                & \mapsto     & (z,t_z.v).
    \end{array}
  \right.
  \]
  Le Théorème \ref{action1} nous assure alors que ce fibré est
  orientable si et seulement si le nombre de composantes de $\RR N$
  dont les structures $\pinpm$ sont échangées sous l'action de $t_z$ est
  pair. Ce qui conclut le démonstration de Théorème \ref{familles}.
\end{proof}

\section{\'Etude des automorphismes réels au-dessus de
  l'identité}\label{automid}

Nous nous intéressons dans ce paragraphe à l'action des éléments de
$\raut{N}$ sur les orientations du fibré $\Det$.  Remarquons tout d'abord qu'un élément de $\raut{N}$ induit une permutation sur l'ensemble à $2k$ éléments formé des structures $Pin^\pm$ des composantes de $\RR N$. Nous noterons $\varepsilon_{\PPP^\pm} : \raut{N} \rightarrow \ZZ/2\ZZ$ le morphisme qui calcule la signature de cette permutation.

\begin{Lemme}\label{retourrangun}
  Soit $(N,c_N)$ un fibré vectoriel complexe sur $(\surg,c_\sur)$ de partie réelle non vide. L'action d'un automorphisme $f\in \raut{N}$ sur les orientations du fibré $\Det$ est la même que celle de $\det(f)\in\raut{\det(N)}$ sur les orientations de $\Det[\det(N)]$ si et seulement si $\varepsilon_{\PPP^\pm}(f) = 1$.
\end{Lemme}

\begin{proof}
  D'après le Lemme \ref{classification} nous avons un isomorphisme
\begin{equation}\label{decdet2}
  (N,c_N) \cong (\det(N),c_{\det(N)})\oplus (\TCC^{\rang(N)-1},conj).
  \tag{$*$}
\end{equation}
Ainsi, un élément $f\in\raut{N}$ se décompose sous l'isomorphisme
(\ref{decdet2}) en 
\[
f = (\det(f)\oplus \id_{\TCC^{\rang(N)-1}})\circ
\left((\det(f)\oplus \id_{\TCC^{\rang(N)-1}})^{-1}\circ f\right).
\]
Donc le signe de l'action de $f$ sur les orientations du fibré $\Det$ est donné par le produit de l'action de $(\det(f)\oplus \id_{\TCC^{\rang(N)-1}})$ de $(\det(f)\oplus \id_{\TCC^{\rang(N)-1}})^{-1}\circ f$ sur les orientations de $\Det[\det(N)\oplus\TCC^{\rang(N)-1}]$.

D'autre part, on remarque (voir \cite{kirby}) 
 que $(\det(f)\oplus \id_{\TCC^{\rang(N)-1}})$ agit trivialement sur les
structures $Pin^{\pm}$ sur les $(\RR N)_i$, donc $\varepsilon_{\PPP^\pm}((\det(f)\oplus \id_{\TCC^{\rang(N)-1}})\circ f) = \varepsilon_{\PPP^\pm}(f)$. L'automorphisme $(\det(f)\oplus \id_{\TCC^{\rang(N)-1}})\circ f$ est un
élément de $\rsaut{N}$ donc d'après le Théorème \ref{action1}, le signe de son action sur les orientations de $\Det[\det(N)\oplus\TCC^{\rang(N)-1}]$ est donné par $\varepsilon_{\PPP^\pm}(f)$.  

Enfin, étudier le signe de l'action de l'automorphisme
$(\det(f)\oplus \id_{\TCC^{\rang(N)-1}})$ sur les orientations de
$\Det$ est égal à celui de l'action de $\det(f)$ sur les orientations de $\Det[\det(N)]$.
\end{proof}
  
Notons $s_N : \raut{N}\rightarrow \ZZ/2\ZZ$ le morphisme associant à $f\in\raut{N}$ le signe de l'action de $\det(f)$ sur les orientations de $\Det[\det(N)]$. Pour résumer le Lemme \ref{retourrangun} en reprenant la suite exacte (\ref{suite2}) du Lemme \ref{decomposition}, nous avons la situation suivante
\[
\xymatrix{
0 \ar[r] & \rsaut{N}\ar[r] \ar[d]_{\varepsilon_{\PPP^\pm}} & \raut{N} \ar[r] \ar[d]_{\varepsilon_{\PPP^\pm}}^{s_N} \ar@/_3pc/[dd] & \raut{\det(N)} \ar[r] \ar[d]_{s_N} & 0 \\
 &\ZZ/2\ZZ & \ZZ/2\ZZ \times \ZZ/2\ZZ \ar[d]^{\text{augm}} & \ZZ/2\ZZ & \\
& & \ZZ/2\ZZ & &
}
\]
où augm$: \ZZ/2\ZZ \times \ZZ/2\ZZ \rightarrow \ZZ/2\ZZ$ est le morphisme d'augmentation et chaque morphisme vertical calcule le signe de l'action des différents groupes d'automorphismes sur les orientations des fibrés déterminants.

Pour terminer la démonstration du Théorème \ref{toutdeg} il ne nous reste plus qu'à étudier le signe $s_N$. Nous supposerons pour cela à partir de maintenant que $N$ est de rang $1$. Les éléments de $\raut{N}$ s'identifient alors canoniquement aux fonctions
\[
\RR \CCC^{\infty}(\surg,\CC^*) = \{f : \surg \rightarrow \CC^*\ |\
\overline{f\circ c_{\sur}} = f\}.
\]
Nous noterons $\modsig$ l'espace des fonctions $\CCC^{\infty}$ sur
$\surg$ à valeurs dans $\CC^*$, $\ZZ/2\ZZ$-équivariantes, prises à homotopie près.

Cette section est divisée en trois parties. Dans un premier temps, nous étudions l'action des automorphismes réels sur les structures $Spin$ réelles de $(\surg,c_\sur)$. Dans un deuxième temps, nous relions cette dernière action à celle sur les orientations du fibré déterminant et démontrons le Théorème \ref{toutdeg} pour un fibré en droites (voir \S \ref{enonce}). Nous obtenons alors comme au paragraphe précédent une interprétation topologique de l'action de $\raut{N}$ sur les orientations de $\Det$.

Nous terminons en déduisant la première classe de Stiefel-Whitney du fibré induit par le fibré déterminant sur le groupe de Picard réel d'une courbe $(\surg,c_\sur)$.

\subsection{Automorphismes réels et structures $Spin$ réelles}

Nous exposons dans ce paragraphe quelques préliminaires algébriques et topologiques à l'étude de l'action des éléments de $\raut{N}$ sur les orientations du fibré $\Det$. Nous commençons par expliciter la structure du groupe $\modsig$. Puis, nous rappelons quelques notions à propos des
structures $Spin$ réelles sur une courbe réelle. Nous décrivons
enfin l'action des éléments de $\raut{N}$ sur ces structures. Rappelons aussi que toutes les courbes réelles que nous considérons sont de partie réelle non vide.

\subsubsection{Structure de
  $\modsig$}\label{modsig}

Rappelons tout d'abord quelques résultats sur la topologie des courbes
réelles (voir par exemple \cite{nat} et \cite{Natanzon}).  Si $\surg \setminus \RR\surg$ a
deux composantes connexes, nous dirons que la courbe réelle est
séparante. On rappelle que
lorsque la courbe est séparante $k = b_0(\RR\surg)$ est congru à $g+1$
modulo $2$. On pose dans ce cas $m = \dd\frac{g+1-k}{2}$, qui est le
genre d'une des composantes de $\surg-\RR\surg$.
\begin{Definition}[voir \cite{Natanzon} et Figure \ref{base4sep}]
  On suppose que $\RR\surg$ est non vide, et soit $p \in \RR\surg$. Une base symplectique réelle de $\hhz{\surg}$
  est une base symplectique $(a_i,b_i)_{i=1\ldots g}$ vérifiant de
  plus
  \begin{itemize}
  \item[si $(\surg,c_\sur)$ n'est pas séparante]:

    \begin{enumerate}
    \item $(c_{\sur})_*(a_i) = a_i$, $i = 1,\ldots, g$
    \item $(c_{\sur})_*(b_i) = -b_i$, $i = 1,\ldots, k-1$
    \item $(c_{\sur})_*(b_i) = -b_i + a_i$, $i = k,\ldots, g$
    \item $p$ se trouve sur la composante de $\RR\surg$ homologue à
      $\dd\sum_{i=1}^g a_i$
    \end{enumerate}
  \item[si $(\surg,c_\sur)$ est séparante]:

    \begin{enumerate}
    \item $(c_{\sur})_*(a_i) = a_i$, $i = 1,\ldots, k-1$
    \item $(c_{\sur})_*(b_i) = -b_i$, $i = 1,\ldots, k-1$
    \item $(c_{\sur})_*(a_i) = a_{i+m}$, $i = k,\ldots, k+m-1$
    \item $(c_{\sur})_*(b_i) = -b_{i+m}$, $i = k,\ldots, k+m-1$
    \item $p$ se trouve sur la composante de $\RR\surg$ homologue à
      $\dd\sum_{i=1}^{k-1} a_i$.
    \end{enumerate}
  \end{itemize}
\end{Definition}
L'existence d'une telle base est démontrée par Natanzon dans
\cite{Natanzon}.

Nous allons décrire une famille $\mathcal{B}$ génératrice d'éléments
de $\modsig$, ce qui nous sera utile par la suite. Pour cela, fixons
une base symplectique réelle $(a_i,b_i)$ de $\hhz{\surg}$ en imposant
de plus que les éléments de cette base soient représentés par des
courbes simples disjointes lorsque c'est possible, que $a_i =
[\RR\surg]_i$ pour $i\in\{1,\ldots,k-1\}$, et que pour une courbe séparante, les courbes
$a_k,\ldots,a_g$ soient globalement stables par $c_\sur$. Nous numéroterons de plus
$(\RR\surg)_0$ la composante sur laquelle se trouve $p$ (voir Figure \ref{base4sep}).
\begin{figure}[h!]
\centering
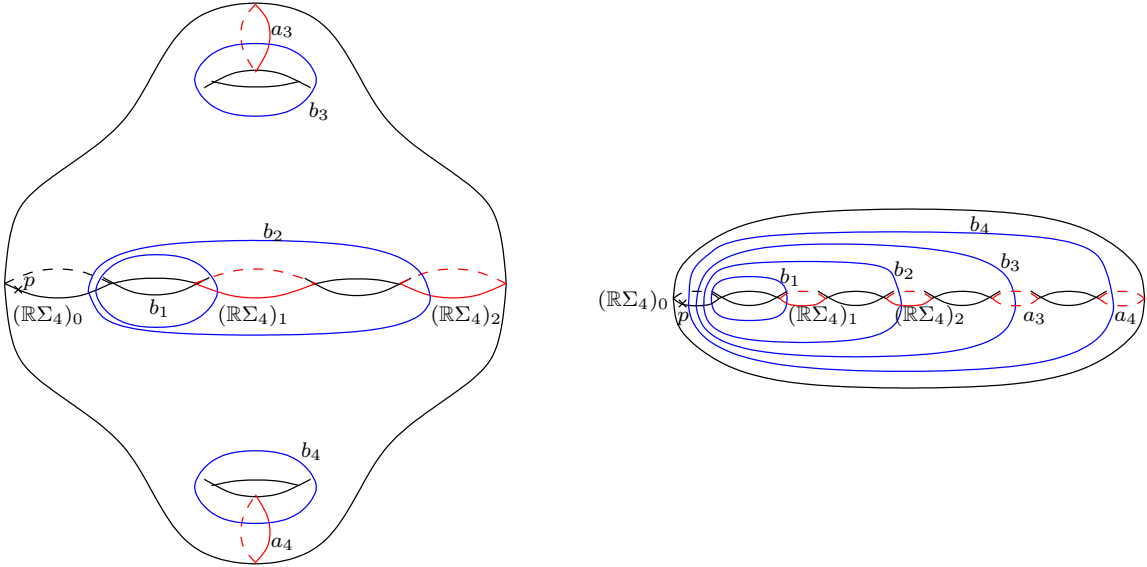  
  \caption{Une base symplectique réelle pour $g=4$, $(\sur_4,c_\sur)$ séparante et $k = 3$ à gauche et $g=4$, $(\sur_4,c_\sur)$ non séparante et $k=2$ à droite}
  \label{base4sep}
\end{figure}

Remarquons qu'une fois fixée une telle base symplectique, les composantes de $\RR\surg$ sont orientées.

 Construisons maintenant la famille génératrice $\mathcal{B}$:
\begin{itemize}
\item[si $(\surg,c_\sur)$ n'est pas séparante:] On commence par choisir des petits
  voisinages tubulaires $A_i\subset \surg$, invariants par $c_\sur$, ayant pour âmes
  les courbes $a_i$, et disjoints les uns des autres. On choisit $t\in
  [-1,1]$ une coordonnée transverse à $a_i$ dans $A_i$ de telle sorte
  que $a_i$ soit la courbe de niveau $t = 0$, et $\theta \in a_i = S^1\subset\CC$
  l'autre coordonnée. Nous imposons de plus que $\theta$ croît le long de $a_i$ et que le difféomorphisme entre $A_i$ et $S^1\times [-1,1]$ donné par les coordonnées $(\theta,t)$ préserve les orientations. On suppose aussi que dans ces coordonnées, la
  structure réelle s'écrit $c_{\sur}(t,\theta) = (-t,\theta)$ si $i =
  0,\ldots,k-1$ et $c_{\sur}(t,\theta) = (-t,-\theta)$ si $i =
  k,\ldots,g$. On obtient alors les fonctions (à homotopie près) $f_i$,
  $i=0,\ldots,g$, réelles, en posant $f_i (t,\theta) = -\e^{i\pi t}$
  sur $A_i$ et en prolongeant par $1$ ailleurs (voir Figure \ref{ann}).
  \begin{figure}[!h]
    \centering
    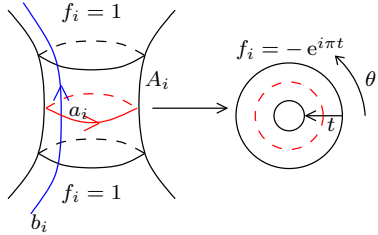
    \caption{Construction de $f_i$}
    \label{ann}
  \end{figure}
\item[si $(\surg,c_\sur)$ est séparante:] On procède de la même façon que
  précédemment pour obtenir $k$ fonctions $f_0,\ldots,f_{k-1}$
  associées aux $k$ courbes $a_0,\ldots,a_{k-1}$. On construit ensuite
  $g+1-k$ autres fonctions $f_k,g_k,\ldots,f_{k+m-1},g_{k+m-1}$ en définissant
  $f_i$ (resp. $g_i$) de la même façon localement autour de $a_i$
  (resp. $b_i$) puis en prolongeant grâce à $c_{\sur}$ au voisinage de
  $a_{i+m}$ (resp. $b_{i+m}$) et par $1$ ailleurs.
\end{itemize}
 
Nous montrons dans la suite que les éléments de $\modsig$ sont, à peu
de choses près, totalement déterminés par leurs indices le long des
courbes dans $\surg$.

\begin{Lemme}
  L'application de calcul d'indice
  \[
  \begin{array}{c c c c}
    \widetilde{\ind}: & \RR \CCC^{\infty} (\surg,\CC^*) & \rightarrow & \RR\Hom(\hhz{\surg},\ZZ)\\
    & f                      & \mapsto     & (c\mapsto \dd\frac{1}{2i\pi} \dd\int_{c}\frac{\d f}{f})
  \end{array}
  \]
  passe au quotient en un morphisme
  \[
  \ind : \modsig \rightarrow \hcz{\surg}_{-1}.
  \]
\end{Lemme}

\begin{proof}
  Remarquons tout d'abord que le changement d'espace d'arrivée est
  conséquence du théorème des coefficients universels et du calcul
  suivant: pour $f\in\RR\CCC^{\infty}(\surg,\CC^*)$ et
  $c\in\hhz{\surg}$
  \[
  \begin{array}{c c l}
    ((c_{\sur})^*(\widetilde{\ind}(f)))(c) &=& (\widetilde{\ind}(f))((c_{\sur})_*(c))\\
    &=& \dd\frac{1}{2i\pi} \dd\int_{(c_{\sur})_*(c)}\frac{\d f}{f}\\
    &=& \dd\frac{1}{2i\pi} \dd\int_{c}\frac{\d (f\circ c_{\sur})}{f\circ c_{\sur}}\\
    &=& (\widetilde{\ind}(\bar{f}))(c)\\
    &=& -(\widetilde{\ind}(f))(c).
  \end{array}
  \]
  Puis, comme $\hcz{\surg}_{-1}$ est discret, deux éléments de $\RR \CCC^{\infty}(\surg,\CC^*)$ qui sont homotopes ont même indice, et $\widetilde{\ind}$ passe bien au quotient.
\end{proof}

\begin{Exple}\label{exind}
  Calculons les indices pour la famille $\mathcal{B}$; l'exposant
  $\pd$ indique que nous prenons le dual de Poincaré de la classe de
  (co)homologie considérée. Nous avons les cas suivant:
  \begin{itemize}
  \item[si $(\surg,c_\sur)$ n'est pas séparante:] $\ind(f_i) = (a_i)^{\pd}$ pour
    $i = 1,\ldots,g$.
  \item[si $(\surg,c_\sur)$ est séparante:] $\ind(f_i) = (a_i)^{\pd}$ pour
    $i = 1,\ldots,k-1$, $\ind(f_i) = (a_i+a_{i+m})^{\pd}$ et
    $\ind(g_i) = (b_i-b_{i+m})^{\pd}$ pour $i = k,\ldots,k+m-1$.
  \end{itemize}
\end{Exple}

  Posons ensuite
  \[
  \RR\CCC^{\infty}(\surg,\CC^*)_0 :=
  \{f\in\RR\CCC^{\infty}(\surg,\CC^*)\ |\ \exists g:\surg\rightarrow
  \CC,\ZZ/2\ZZ\text{-équivariante},\ f = \exp(g)\}.
  \]
  C'est donc l'ensemble des fonctions sur $\surg$ qui admettent un
  logarithme $\ZZ/2\ZZ$-équivariant.

  \begin{Lemme}\label{compo} 
  $\RR\CCC^{\infty}(\surg,\CC^*)_0$ est la composante
  connexe de la fonction constante égale à $1$ dans
  $\RR\CCC^{\infty}(\surg,\CC^*)$. Ainsi, on a
  \[
  \modsig =
  \RR\CCC^{\infty}(\surg,\CC^*)/\RR\CCC^{\infty}(\surg,\CC^*)_0.
  \]
\end{Lemme}

\begin{proof}
  \begin{itemize}
  \item $\RR\CCC^{\infty}(\surg,\CC^*)_0$ est ouvert : si
    $f\in\RR\CCC^{\infty}(\surg,\CC^*)$ et
    $g\in\RR\CCC^{\infty}(\surg,\CC)$ vérifient
    \[
    \left| f-\exp(g)\right| < \frac{1}{2} \inf \left| \exp(g)\right|
    \]
    alors
    \[
    \left| \frac{f}{\exp(g)} - 1 \right| < \frac{1}{2}
    \]
    donc il existe $h\in\RR\CCC^{\infty}(\surg,\CC)$ avec
    \[
    \frac{f}{\exp(g)} = \exp(h).
    \]
    De plus, $\exp(h_{\rvert\RR\sur})\in\RR_+^*$, et donc $h\circ
    c_{\sur} - \bar{h} = 4i\pi k$, pour un $k\in\ZZ$. Ainsi $f =
    \exp(h-2i\pi k + g) \in \RR\CCC^{\infty}(\surg,\CC^*)_0$.
  \item $\RR\CCC^{\infty}(\surg,\CC^*)_0$ est fermé car
    \[
    \RR\CCC^{\infty}(\surg,\CC^*)\setminus\RR\CCC^{\infty}(\surg,\CC^*)_0 =
    \dd\bigsqcup_{f\in
      \RR\CCC^{\infty}(\surg,\CC^*)\setminus\RR\CCC^{\infty}(\surg,\CC^*)_0}
    f\RR\CCC^{\infty}(\surg,\CC^*)_0.
    \]
  \item Enfin, $1\in \RR\CCC^{\infty}(\surg,\CC^*)_0$.
  \end{itemize}
\end{proof}

\begin{Lemme}\label{lemmodsig}
  La suite suivante est exacte :

\begin{equation}\label{suitemodsig}
  0\rightarrow \ZZ/2\ZZ \xrightarrow{-1} \modsig\xrightarrow{\ind}\hcz{\surg}_{-1}\rightarrow 0.
  \tag{$*$}
\end{equation}

De plus, la famille $\mathcal{B}$ engendre $\modsig$.
\end{Lemme}

\begin{proof}
  \begin{itemize}
  \item La surjectivité du morphisme $\ind$ provient directement de
    l'Exemple \ref{exind}: l'image de $\mathcal{B}$ par $\ind$
    engendre $\hcz{\surg}_{-1}$.
  \item Si $f$ est dans le noyau de $\ind$, alors en particulier elle a le même signe sur chaque composante de $\RR\surg$. Quitte à la multiplier par $-1$, nous pouvons supposer que $f$ est positive sur $\RR\surg$. Montrons que la fonction $f$ admet un logarithme $\ZZ/2\ZZ$-équivariant. Notons pour cela $\pi:\widetilde{\surg}\rightarrow \surg$ le revêtement universel de $\surg$. Comme $\widetilde{\surg}$ est simplement connexe, il existe $G : \widetilde{\surg}\rightarrow \CC$ telle que le diagramme suivant commute
\[
\xymatrix{
\widetilde{\surg} \ar@{.>}[r]^G \ar[d]_\pi \ar[rd] & \CC \ar[d]^\exp \\
\surg \ar[r]_f & \CC^*.  
}
\]
De plus, si $z,z'\in\widetilde{\surg}$ vérifient $\pi(z) = \pi(z')$, alors on a $G(z) = G(z')$. En effet, prenons un chemin continu $\gamma : [0,1]\rightarrow \widetilde{\surg}$ reliant $z$ à $z'$. On obtient un lacet $\pi\circ\gamma\in\pi_1(\surg,\pi(z))$ dont l'image par $f$ est un lacet contractible dans $\CC^*$ car $f$ est d'indice nul. Par définition, $G(z)$ et $G(z')$ ne diffèrent que par l'action de $f\circ\pi\circ\gamma\in\pi_1(\CC^*)$ qui est nul. Ainsi $G(z) = G(z')$ et on en déduit qu'il existe $g\in \CCC^{\infty}(\surg,\CC)$ telle que $\exp(g) = f$. De plus, comme $f$ est $\ZZ/2\ZZ$-équivariante et positive sur $\RR\surg$, il existe $k\in\ZZ$ tel que $g\circ c_\sur - \overline{g} = 4i\pi k$. Ainsi $g-2i\pi k$ donne un logarithme $\ZZ/2\ZZ$-équivariant de $f$.

D'après le Lemme \ref{compo}, la classe de $f$ dans $\modsig$ est donc celle de la fonction constante égale à $1$.

Le noyau de $\ind$ est donc constitué des deux éléments $1$ (qui admet un logarithme $\ZZ/2\ZZ$-équivariant) et $-1$ (qui admet un logarithme qui n'est pas $\ZZ/2\ZZ$-équivariant).
  \end{itemize}
\end{proof}

\begin{Rem}\label{remprod}
  La démonstration du Lemme \ref{lemmodsig} montre même que la suite
  (\ref{suitemodsig}) est scindée. En fait nous avons un isomorphisme
  non canonique
  \[
  \modsig \cong \ZZ/2\ZZ \times \modsig_{+a_0} = \ZZ/2\ZZ \times
  \hcz{\surg}_{-1}
  \]
  où $\modsig_{+a_0}$ est le sous groupe de $\modsig$ composé des
  fonctions positives sur la composante $(\RR\surg)_0$. Si on place le
  point $p\in \RR\surg$ sur une autre composante $(\RR\surg)_i$, nous
  obtenons l'isomorphisme
  \[
  \begin{array}{c c c}
    \ZZ/2\ZZ\times \hcz{\surg}_{-1} & \rightarrow & \ZZ/2\ZZ \times \hcz{\surg}_{-1} \\
    (\epsilon, \alpha)             & \mapsto     & (\epsilon + \alpha(b_{i}), \alpha).
  \end{array}
  \]  
\end{Rem}

\subsubsection{Structures $Spin$ réelles sur
  $(\surg,c_\sur)$}\label{spin}

Nous rappelons quelques faits à propos des structures $Spin$ sur
$\surg$ (voir \cite{atiyah}, \cite{john}, \cite{Lawson} et
\cite{Natanzon} par exemple) Nous supposerons que
$(\Sigma_g,c_{\Sigma})$ est munie d'une métrique riemannienne réelle,
que son genre $g$ est au moins $1$, et que $\RR \Sigma_g \neq
\emptyset$. Le choix de la métrique n'influe pas sur les résultats que
nous allons énoncer, pour la même raison mentionnée dans la Remarque
\ref{pinmetrique}.

Notons $R^+_{\sur}$ le fibré principal des repères orthonormés directs
du fibré tangent à $\surg$. Ce fibré a pour groupe $SO_2(\RR)$;
celui-ci admet un revêtement double non trivial $Spin(2)$ (voir
\cite{atibott}). Une structure $Spin$ sur $\Sigma_g$ est la donnée
d'une classe d'isomorphisme de paires composées d'un fibré principal
$\pi : P_{Spin}\rightarrow \Sigma_g$ de groupe $Spin(2)$ et d'un
morphisme de fibrés $P_{Spin}\rightarrow R^+_\sur$ équivariant pour
les actions de $SO_2(\RR)$ et $Spin(2)$.  L'obstruction à l'existence
d'une telle structure $Spin$ est donnée par la seconde classe de
Stiefel-Whitney de la surface. Or comme celle-ci est orientable, sa caractéristique d'Euler est paire. Nous
avons donc $w_2(T\surg) = 0$, et la surface admet une structure $Spin$.

Les structures $Spin$ sur $\Sigma_g$ forment un ensemble
$Spin(\Sigma_g)$ qui est en bijection avec les éléments de
$H^1(R^+_{\sur},\ZZ/2\ZZ)$ non triviaux dans les fibres. Plus
précisément, la suite spectrale de Leray-Serre associée à la fibration
$SO_2(\RR) \xhookrightarrow{i} R^+_{\sur} \xrightarrow{\pi} \Sigma_g$ donne:
\[
0 \rightarrow H^1(\Sigma_g,\ZZ/2\ZZ) \xrightarrow{\pi^*}
H^1(R^+_{\sur},\ZZ/2\ZZ) \xrightarrow{i^*} H^1(SO_2(\RR),\ZZ/2\ZZ)
\rightarrow 0
\label{leray}\tag{$*$}
\]
et $Spin(\Sigma_g) = \{\zeta\in H^1(R^+_{\sur},\ZZ/2\ZZ)\ |\
i^*(\zeta)\neq 0\}$. C'est un espace affine sur
$H^1(\Sigma_g,\ZZ/2\ZZ)$.

On note $z \in H_1(R^+_{\sur},\ZZ/2\ZZ)$ la classe de la
fibre. Johnson dans \cite{john} affirme alors que pour tout $a\in
\hhzdeux{\Sigma_g}$, on a un relevé canoniquement associé
$\tilde{a}\in\hhzdeux{R^+_{\sur}}$. Si $a$ s'écrit comme une somme de courbes simples $a_i$, $1\leq i\leq m$, on note alors 
$\overrightarrow{a_i}$ les éléments de $\hhzdeux{R^+_{\sur}}$ induits par les champs de vecteurs tangents unitaires de ces courbes, et on pose $\tilde{a} = \dd\sum_{i=1}^m\overrightarrow{a_i} + mz$. Toutefois, ce relevé ne scinde pas la
suite duale à (\ref{leray}). En effet, si on note $\bullet$ le produit
d'intersection sur $\hhzdeux{\surg}$, nous avons
\[
\widetilde{a + b} = \tilde{a} + \tilde{b} + (a\bullet b)z.
\]
Cette dernière égalité nous fournit une identification entre
structures $Spin$ sur $\Sigma_g$ et formes quadratiques sur
$\hhzdeux{\Sigma_g}$ en associant à $\zeta \in Spin(\Sigma_g)$
l'application
\[
\begin{array}{c c c c}
  q_{\zeta} : & \hhzdeux{\Sigma_g} & \rightarrow & \ZZ/2\ZZ \\
  & a                  & \mapsto     & \langle \zeta,\tilde{a} \rangle\\
\end{array}
\]
où le crochet de dualité est donné par $\hczdeux{R^+_{\sur}} =
Hom(\hhzdeux{R^+_{\sur}},\ZZ/2\ZZ)$ (coefficients universels).

D'autre part, comme remarqué par Atiyah \cite{atiyah}, étant fixée une
structure complexe sur $\Sigma_g$, ses structures $Spin$ sont aussi en
correspondance avec l'ensemble $S(\Sigma_g)$ des racines carrées holomorphes du
fibré canonique $K_{\Sigma}$. Si $\zeta$ est un élément de
$Spin(\Sigma_g)$, on notera $L_{\zeta}$ l'élément de
$S(\Sigma_g)$ associé de la façon suivante. La structure $\zeta$ vue comme élément de $H^1(R_{\sur}^+,\ZZ/2\ZZ)$ correspond à un revêtement double (à isomorphisme près) du $U_1 = SO_2(\RR)$ fibré principal $R_{\sur}^+$ dont la restriction à la fibre est $z\in U_1\mapsto z^2\in U_1$, et donc à un fibré en droites complexes $L_{\zeta}$ et un isomorphisme $\alpha : L_{\zeta}^2\rightarrow K_{\sur}$ (à isomorphisme près). Le fibré $L_{\zeta}$ hérite alors d'une structure holomorphe rendant $\alpha$ holomorphe. L'ensemble $S(\Sigma_g)$ admet de plus une forme quadratique
naturelle (voir \cite{atiyah} ou \cite{harris}).
\begin{Theoreme}[Relation de Riemann-Mumford]
L'application 
\[
\begin{array}{c c c c}
  \varphi: & S(\Sigma_g) & \rightarrow & \ZZ/2\ZZ \\
  & L           & \mapsto     & \dim H^0(L) \mod 2
\end{array}
\]
est une forme quadratique associée au produit cup sur $\hczdeux{\Sigma_g}$. \qed
\end{Theoreme}
Lorsqu'on \og
linéarise\fg cette forme en $L_\zeta$, on obtient
\[
\begin{array}{c}
  \varphi_{L_\zeta}(a) = q_{\zeta}(a),\ \forall a\in \hhzdeux{\Sigma_g}\\
  \text{où } \varphi_{L_\zeta} (a) = \varphi(L_{\zeta} + a^{\pd}) - \varphi(L_{\zeta}).
\end{array}
\]
On identifiera dans la suite
\[
\begin{array}{c c l}
  Spin(\Sigma_g) & = & \{ q: \hhzdeux{\Sigma_g} \rightarrow \ZZ/2\ZZ,\text{ quadratique}\}\\
  & = & \{ L\in S(\Sigma_g)\}.
\end{array}
\]
Nous allons maintenant utiliser la structure réelle sur $\surg$. Nous
noterons $(c_{\sur})_*$ et $(c_{\sur})^*$ les automorphismes de
$\hhzdeux{\surg}$ et $\hczdeux{\surg}$ induits par
$c_{\sur}$. L'action de $c_{\sur}$ sur $S(\surg)$ est induite par celle de $c_\sur$ sur les diviseurs. 

Notons $J$ la structure complexe fixée sur $\surg$, et supposons quitte à changer la métrique que $J$ est orthogonale pour la métrique fixée sur $\surg$. Posons $\overline{\surg}$ la surface $\surg$ munie de l'orientation opposée et de la structure complexe $-J$. Considérons le morphisme
    \[
    \begin{array}{c c c c}
      R : & R^+_\sur        & \rightarrow & R^+_{\bar{\sur}} \\
      & (z, (v,Jv)) & \mapsto     & (z, (v,-Jv)).
    \end{array}
    \]
    Nous obtenons alors une involution
    \[
    (\d c_\sur \circ R)^* : Spin(\surg) \rightarrow Spin(\surg).
    \]
    
\begin{Definition}
  L'ensemble des structures $Spin$ réelles sur $(\surg,c_\sur)$ est
  l'ensemble $\rspin(\surg) = \Fix((\d c_\sur \circ R)^*)$.
\end{Definition}

Nous rappelons le lemme suivant (voir par exemple
\cite{degitkha}).

\begin{Lemme}\label{lemspinre}
Supposons que $\RR\surg\neq\emptyset$, et soit $\zeta \in Spin(\surg)$. Alors les trois conditions suivantes sont équivalentes:
\begin{enumerate}[(i)]
\item\label{petitun} $\zeta\in\rspin(\surg)$,
\item\label{petitdeux} $(c_\sur)^*q_\zeta = q_\zeta$,
\item\label{petittrois} $(c_\sur)^*L_\zeta = L_\zeta$.
\end{enumerate}\qed
\end{Lemme}

\begin{Rem}
Lorsque $\RR\surg\neq\emptyset$, la condition $(c_\sur)^*L = L$ pour un $L\in S(\surg)$ équivaut
     à l'existence d'une structure réelle anti-holomorphe relevant
    $c_\sur$ sur $L$ (voir par exemple \cite{wel2}). Une telle structure n'est
    toutefois pas unique: deux relevés diffèrent d'un facteur complexe
    constant.
\end{Rem}

\begin{Lemme}
  L'ensemble $\rspin(\surg)$ est un espace affine de direction $F =
  \Fix((c_{\sur})^*) \subset \hczdeux{\surg}$.
\end{Lemme}
\begin{proof}
  En effet, si $q$ et $q'$ sont deux éléments de $\rspin(\surg)$,
  alors $q-q'$ est un élément de $\hczdeux{\surg}$ et $(q-q')\circ
  (c_\sur)_* = q-q'$. Réciproquement, si $q\in\RR Spin(\surg)$ et $\alpha\in F$, alors $q+\alpha$ est un élément de $Spin(\surg)$ qui vérifie de plus $(c_\sur)^*(q+\alpha) = q + \alpha$. Donc $q+\alpha\in\RR Spin(\surg)$.
\end{proof}

\begin{Rem}\label{fmoins}
L'isomorphisme de dualité de Poincaré
\[
\begin{array}{c c c}
  \hhzdeux{\surg} & \rightarrow & \hczdeux{\surg}\\
   a              & \mapsto     & a\bullet
\end{array}
\]
commute avec $(c_\sur)_*$ et $(c_\sur)^*$, de sorte que $F$ est le dual de Poincaré de $\Fix((c_\sur)_*)\subset\hhzdeux{\surg}$. De plus, si l'on prend une base symplectique réelle $(a_i,b_i)$ de $\hhz{\surg}$, elle induit une base de $\hhzdeux{\surg}$, et $\Fix((c_\sur)_*)$ est le sous-espace engendré par $a_1,\ldots,a_g,b_1,\ldots,b_{k-1}$ lorsque la courbe n'est pas séparante, et par $a_1,\ldots,a_{k-1},a_k + a_{k+m},\ldots,a_{k+m-1} + a_g,b_1,\ldots,b_{k-1},b_k + b_{k+m},\ldots,b_{k+m-1} + b_g$ si la courbe est séparante. Dans les deux cas, nous notons que le $\ZZ/2\ZZ$-espace vectoriel $F$ est de dimension $g + k - 1$ et qu'il contient la réduction modulo $2$ $F^-$ de $\hcz{\surg}_{-1}$ comme sous-espace de dimension $g$.
\end{Rem}

\subsubsection{Action de $\modsig$ sur
  $\rspin(\surg)$}\label{actmodspin}

Revenons maintenant aux automorphismes réels sur les fibrés en droites
complexes munis d'une structure réelle. Notons $F_+$ la réduction modulo $2$ de $\hhz{\surg}_{+1}$, de sorte que $F_+$ et $F^-$ sont duaux de Poincaré. La
réduction modulo $2$ du morphisme d'indice
\[
\begin{array}{c c c c}
  \ind_2: & \modsig & \rightarrow & \hczdeux{\surg} = Hom(\hhzdeux{\surg}, \hhzdeux{\CC^*})\\
  & f       & \mapsto     & f_*
\end{array}
\]
a pour image $F^-$ d'après le Lemme \ref{lemmodsig} (voir la Remarque \ref{fmoins}).

\begin{Rem}\label{calculind}
  On peut même expliciter ce morphisme à l'aide d'une base $(d_i)$ de
  $F_+$ que l'on complète en une base $(d_i,e_i)$
  symplectique de $\hhzdeux{\surg}$:
  \[
  (f_*)^{\pd} = \dd\sum_{i} f_*(e_i) d_i.
  \]
\end{Rem}

D'autre part, les éléments de $\modsig$ s'identifiant canoniquement
aux classes d'homotopie des automorphismes réels du fibré tangent à
$\surg$, ils agissent par tiré en arrière sur les structures $Spin$
réelles sur $\surg$.

\begin{Lemme}\label{actiontrans}
  L'action de $\modsig$ sur $\rspin(\surg)$ s'identifie à l'action par
  translation de $F^-$ sur ce même espace grâce au morphisme $\ind_2$.
\end{Lemme}
\begin{proof}
Prenons un élément $f$ de $\modsig$, un automorphisme $\Phi_f\in\raut{T\surg}$ égal à $f$ dans $\modsig$ et une structure $Spin$ $q_\zeta$ sur $\surg$. Soit $a\in\hhzdeux{\surg}$ une courbe simple sur $\surg$. Nous avons par définition
\[
\begin{array}{c c l}
f.q_\zeta(a) & = & \langle (\Phi_f)^*\zeta , \tilde{a}\rangle\\
            & = & \langle \zeta , (\Phi_f)_* (\overrightarrow{a} + z)\rangle\\
\end{array}
\]
Or, l'action de $(\Phi_f)_*$ sur la classe de la fibre $z$ est triviale, et en trivialisant $T\surg$ le long de $a$, on voit que l'action de $(\Phi_f)_*$ sur le champ de vecteur tangent à $a$ est donnée par l'indice de $f$ le long de $a$. Plus précisément, $(\Phi_f)_*(\overrightarrow{a}) = \overrightarrow{a} + f_*(a)z$. Ainsi
  \[
\begin{array}{c c l}
  f.q_\zeta(a) & = & \langle \zeta, \overrightarrow{a} + f_*(a)z + z\rangle\\
              & = & q_\zeta(a) + f_*(a)
\end{array}
  \]
car $\langle \zeta,z\rangle = 1$.
\end{proof}

Supposons que $\RR\surg\neq\emptyset$ et notons 
\[\hczdeux[g+1]{\RR\surg} = \{w\in\hczdeux{\RR\surg}\ |\ w([\RR\surg]) = g+1 \mod 2\},
\]
et pour tout $w\in \hczdeux[g+1]{\RR\surg}$,
\[
    \rspin(\surg,w)  := \{ L\in \rspin(\surg)\ |\ w_1(\RR L) = w\}
    \]
On en déduit la partition
  \[\label{pa}
  \rspin(\surg) =
  \dd\bigsqcup_{w\in\hczdeux[g+1]{\RR\surg}} \rspin(\surg,w).\tag{$*$}
  \]

\begin{Proposition}\label{partition}
Supposons que $\RR\surg\neq\emptyset$. La partition (\ref{pa}) est la réunion des orbites de l'action de $\modsig$ sur $\rspin(\surg)$. En particulier, ce sont des espaces affines sur $F^-$. De plus, chaque orbite se réécrit
\[
 \rspin(\surg,w) = \{ q\in \rspin(\surg)\ |\ q_{|\hhzdeux{\RR\surg}} = \mathds{1} - w\}.
\]
\end{Proposition}

Avant de passer à la démonstration, donnons le Lemme \ref{tr} que l'on peut trouver chez Natanzon (Lemma 3.2 p.72 de \cite{Natanzon}).

\begin{Lemme}\label{tr}
  Soit $q\in\rspin(\surg)$ et $a$ une courbe simple connexe
  sur $\surg$. Supposons que $\RR\surg\neq\emptyset$.
\begin{enumerate} 
\item Si $c_\sur(a) = a$ et $a\cap\RR\surg = \emptyset$ (c'est-à-dire que $a$ est une courbe globalement stable par $c_\sur$) alors $q(a) = 1$.
\item Si $c_\sur(a)\cap a = \emptyset$ alors $q(a + c_\sur(a)) = 0$.
\end{enumerate}\qed
\end{Lemme}

\begin{proof}[Démonstration de la Proposition \ref{partition}]
  Commençons par démontrer que pour $w\in \hczdeux[g+1]{\RR\surg}$ nous
  avons l'égalité
  \[
  \{ L\in \rspin(\surg)\ |\ w_1(\RR L) = w\} = \{ q\in \rspin(\surg)\
  |\ q_{|\hhzdeux{\RR\surg}} = \mathds{1} - w\}.
  \]
  Prenons pour cela $\zeta\in\rspin(\surg)$. Par définition
  \[
  q_\zeta([\RR\surg]_i) = \langle\zeta,\widetilde{[\RR\surg]_i}
  \rangle,
  \]
  ce qui signifie que $q_\zeta([\RR\surg]_i)$ vaut $1$ si et seulement
  si le lacet $\widetilde{[\RR\surg]_i}$ dans $R^+_\sur$ ne se relève pas
  dans la structure $Spin$ donnée par $\zeta$. Or, ceci est équivalent
  à l'existence d'une section de $(\RR
  L_\zeta)_i$ ne s'annulant pas (voir \cite{Natanzon}). Donc
  \[
  q_{\zeta|\hhzdeux{\RR\surg}} = \mathds{1} - w_1(\RR L).
  \]

  D'autre part, si $f$ est un élément de $\modsig$ et
  $q\in\rspin(\surg,w)$, alors grâce au Lemme \ref{actiontrans}
  \[
  (f.q)_{|\hhzdeux{\RR\surg}} = q_{|\hhzdeux{\RR\surg}} +
  f_{*|\hhzdeux{\RR\surg}} = \mathds{1} - w
  \]
  car $f_*$ est un élément de $F^-$. Donc l'action de $\modsig$
  préserve chacun des ensembles $\rspin(\surg,w)$.

  Puis, si $q,q'\in\rspin(\surg,w)$, le Lemme \ref{tr} nous assure que $(q-q')_{|F_+} = 0$ et donc que $q-q'$ est un
  élément de $F^-$. Ainsi l'action de
  $\modsig$ est transitive sur $\rspin(\surg,w)$.
\end{proof}

\begin{Definition}
  Si $\zeta$ est un élément de $\rspin(\surg,w)$, nous dirons que
  $w\in \hczdeux{\RR\surg}$ est la première classe de Stiefel-Whitney
  de la structure $Spin$ $\zeta$, et nous écrirons $w_{\zeta} = w$.
\end{Definition}

Rappelons d'autre part (voir \cite{john} par exemple) qu'étant donnée une structure $Spin$ $q$ sur $\surg$, on définit son invariant d'Arf:
\[
\arf(q) = \dd\sum_iq(a_i)q(b_i) \in \ZZ/2\ZZ
\]
pour $(a_i,b_i)$ une base symplectique de $\hhzdeux{\surg}$. Cette
expression ne dépend pas de la base symplectique choisie et détermine
la classe de bordisme de $(\surg,q)$, élément de $\Omega_2^{Spin}\cong
\ZZ/2\ZZ$ (voir \cite{kirby}).

\begin{Lemme}\label{actarf}
  Pour $f\in\modsig$ et $q\in\rspin(\surg)$,
  \[
  \arf(f.q) = \arf(q) + q((f_*)^{\pd}).
  \]
  En particulier, si $q'$ est une structure $Spin$ réelle de même
  classe de Stiefel-Whitney que $q$, alors,
  \[
  \arf(f.q)-\arf(q) = \arf(f.q') - \arf(q').
  \]
\end{Lemme}
\begin{proof}
  Prenons une base symplectique $(a_i,b_i)$ de $\hhzdeux{\surg}$ comme
  dans la Remarque \ref{calculind}. Alors
  \[
  \begin{array}{c l}
    \arf(f.q) & = \dd\sum_i(q(a_i) + f_*(a_i))(q(b_i) + f_*(b_i)) \\
    & = \dd\sum_iq(a_i)q(b_i) +\dd\sum_iq(a_i)f_*(b_i) + \dd\sum_iq(b_i)f_*(a_i) + \dd\sum_if_*(a_i)f_*(b_i).
  \end{array}
  \]
  Mais, comme $f_*$ est un élément de $F^-$ qui est le dual de
  Poincaré du lagrangien $F_+$,
  \[
  \arf(f.q) = \arf(q) + \dd\sum_iq(a_i)f_*(b_i)
  \]
  et $q$ étant linéaire sur ce même lagrangien, nous obtenons la
  première affirmation,
  \[
  \arf(f.q) = \arf(q) + q\left(\dd\sum_if_*(b_i)a_i\right).
  \]
  Puis, pour $q$, $q'$ et $f$ comme dans l'énoncé, il existe $g$ dans
  $\modsig$ tel que $g.q = q'$ (Lemme \ref{actiontrans} et Proposition \ref{partition}). Donc $q'((f_*)^{\pd}) = q((f_*)^{\pd})
  + g_*((f_*)^{\pd})$. Mais comme $(f_*)^{\pd}$ et $(g_*)^{\pd}$ sont
  tous les deux dans le lagrangien $F_+$, nous avons
  $g_*((f_*)^{\pd}) = 0$, ce qui prouve la seconde assertion.
\end{proof}

Le Lemme \ref{actarf} assure que l'application
\begin{multline*}
  \AA : (f,w)\in \modsig\times \hczdeux[g+1]{\RR\surg}  \\
    \mapsto q((f_*)^{\pd}) = \arf(f.q)-\arf(q) \in\ZZ/2\ZZ
\end{multline*}
 est bien définie indépendamment du choix de $q\in\rspin(\surg,w)$.
Cette application $\AA$ mesure l'action de $\modsig$ sur les
classes de bordisme des structures $Spin$ réelles de $\surg$.

\begin{Lemme}
   Pour tout $w$ dans $\hczdeux[g+1]{\RR\surg}$, l'application $\AA^w = \AA(.,w):\modsig\rightarrow\ZZ/2\ZZ$ est un morphisme.
\end{Lemme}
\begin{proof}
  Ce lemme résulte du fait que les formes quadratiques $q
  \in \rspin(\surg)$ sont linéaires en restriction au lagrangien
  $F_+$ qui est le dual de Poincaré de $F^-$.
\end{proof}

\begin{Definition}
  On dit que deux structures $Spin$ réelles $\zeta$ et $\zeta'$ sont
  bordantes si et seulement si elles ont même première classe de Stiefel-Whitney et si $(\surg,\zeta)$ et $(\surg,\zeta')$ sont bordantes.
\end{Definition}

Les lemmes précédents montrent que le groupe $\modsig$ agit sur les classes de
bordisme réelles sans changer leur première classe de Stiefel-Whitney et que cette action est donnée par l'application
$\AA$.

\subsection[Actions des automorphismes réels au-dessus de l'identité]{Action des automorphismes réels sur les orientations du fibré déterminant et structures $Spin$ réelles}\label{enonce}

\subsubsection{\'Enoncés}\label{subenon}

Nous pouvons maintenant rappeler l'énoncé suivant qui donne une interprétation topologique de l'action de $\raut{N}$ sur les orientations de $\Det$ lorsque $\deg(N) = g+1 \mod 2$.

\begin{Theoreme}\label{enoncea}
Soit $(N,c_N)$ un fibré vectoriel complexe de rang $1$ muni d'une structure réelle sur $(\surg,c_\sur)$, de partie réelle non vide. Soit $f \in \raut{N} = \RR \CCC^\infty(\surg,\CC^*)$. Si le degré de $N$ n'est pas de même parité que le genre de $\surg$, alors l'action de $f$ sur les orientations du fibré $\Det$ coïncide avec l'action de $f$ sur les classes de bordisme de structures $Spin$ réelles de $(\surg,c_\sur)$ de même première classe de Stiefel-Whitney que $\RR N$. Ainsi, $f$ préserve les orientations de $\Det$ si et seulement si $\AA(f,w_1(\RR N)) = 0$. 
\end{Theoreme}

\begin{Rem}\label{parite}
\begin{itemize}
\item Nous traitons le cas où $\deg(N) = g \mod 2$ dans le Théorème \ref{enoncesn} mais notons dès maintenant que l'énoncé précédent n'est plus vrai dans ce cas. En effet, la fonction constante égale à $-1$ préserve toujours les classes de bordisme de structures $Spin$ réelles de $\surg$, mais son action sur les orientations du fibré $\Det$ n'est triviale que lorsque $\deg(N)+1-g = 0\mod 2$. 

Nous pouvons formuler cette remarque de façon différente. Soit $\mathcal{N}$ une déformation réelle de fibrés en droites holomorphes sur $(\surg,c_\sur)$ au-dessus d'une base $(B,c_B)$. C'est-à-dire que pour $b$ dans $B$ (resp. dans $\RR B$), $\mathcal{N}_b$ est un fibré en droites holomorphe (resp. réel) sur $\surg$, et $\mathcal{N}$ est un fibré en droites holomorphe réel sur $B\times\surg$. Notons $\pi : B\times\surg\rightarrow B$. Alors le groupe de Picard réel $\RR\Pic (B)$ agit sur $\mathcal{N}$ par $L\in\RR\Pic (B)\mapsto \pi^*L\otimes\mathcal{N}$. Fibre à fibre, cette action ne change pas $\mathcal{N}_b$. Par contre la déformation globale n'est plus la même.

En prenant en particulier $B = \cp$, $L = \OO_{\cp}(1)$ et $\deg(\mathcal{N}_b) = g\mod 2$, cette action change l'orientabilité du fibré déterminant $\Det[\mathcal{N}]$ au-dessus de $\RR B = \rp$.
\item Dans le cas particulier où la courbe est de genre zéro, le Lemme \ref{lemmodsig} nous dit qu'il n'y a essentiellement que deux automorphismes de fibrés en droites holomorphes réels sur la sphère au-dessus de l'identité : l'identité et la multiplication par $-1$. Pour la sphère, nous n'avons donc pas besoin de structures $Spin$ réelles. Un automorphisme n'échangera les orientations du fibré déterminant que s'il est négatif sur l'équateur et si le degré du fibré est pair.
\end{itemize}
\end{Rem}

Comme le suggère la Remarque \ref{parite}, nous devons prendre en compte le signe des éléments de $\raut{N}$ sur la partie réelle de $\surg$ de façon plus précise pour obtenir un résultat plus général. C'est ce que nous faisons maintenant.

 Tout élément $f$ de $\modsig$ induit un morphisme de
calcul de signe
\[
\epsilon_f : H_0(\RR\surg,\ZZ/2\ZZ) \rightarrow \ZZ/2\ZZ
\]
de la façon suivante. Le quotient $\dd\frac{f}{|f|}$ restreint à chaque composante de $\RR \surg$ est un élément de $\{-1,1\}$. La fonction
\[
\frac{f}{|f|} : \RR\surg\rightarrow \ZZ/2\ZZ
\]
est donc localement constante. Elle définit donc un morphisme de $H_0(\RR\surg,\ZZ/2\ZZ)\rightarrow (\ZZ/2\ZZ)^k$ que l'on compose avec le morphisme d'augmentation $(\ZZ/2\ZZ)^k\rightarrow\ZZ/2\ZZ$ pour obtenir $\epsilon_f$.

\begin{Rem}\label{epsf}
  Si nous choisissons une base symplectique réelle $(a_i,b_i)$ de
  $\hhz{\surg}$, alors les morphismes $f_*\in F^-$ et $\epsilon_f$
  sont reliés par
  \[
  \epsilon_f([\RR\surg]_i^{\pd}) = f_*(b_i) + \epsilon_f([\RR\surg]_0^{\pd}).
  \]
\end{Rem}

Posons
\[
\begin{array}{c c c c}
  \beta_0(f): & \hczdeux{\RR\surg} & \rightarrow & \ZZ/2\ZZ\\
  & w                  & \mapsto     & \epsilon_f(\mathds{1}-w^{\pd}).
\end{array}
\]
Cette application mesure l'action de $f$ sur les orientations de la droite réelle
$\underset{w_1(\RR L)([\RR \surg]_i) = 0}{\bigotimes(\RR L)_{x_i}}$ pour tout $L\in
\rspin(\surg,w)$ et $x_i\in(\RR\surg)_i$.

\begin{Lemme}\label{affine}
  Pout tout $f$ élément de $\modsig$, les applications $\beta_0(f)$ et
  $\AA(f)$ sont deux formes affines de même direction.
\end{Lemme}
\begin{proof}
  Soient $w$ et $w'$ deux éléments de $\hczdeux[g+1]{\RR\surg}$,
  $q\in\rspin(\surg,w)$ et $q'\in\rspin(\surg,w')$. Prenons une base symplectique réelle $(a_i,b_i)$ de $\hhz{\surg}$. Elle induit une base $c_1 = a_1,\ldots, c_{k-1} =a_{k-1}, c_k = a_k + a_{k+m},\ldots,c_{k+m-1} = a_{k+m-1} + a_g, c_{k+m} = b_k - b_{k+m},\ldots,c_g = b_{k+m-1} - b_g$ (resp. $c_1 = a_1,\ldots,c_g = a_g$) de $\hhz{\surg}_{+1}$ lorsque la courbe est séparante (resp. non séparante). On complète cette dernière en une nouvelle base symplectique de $\hhz{\surg}$ en posant $d_1=b_1,\ldots,d_{k-1} = b_{k-1},d_k = b_k,\ldots,d_{k+m-1} = b_{k+m-1}, d_{k+m} = a_k,\ldots,d_g = a_{k+m-1}$ (resp. $d_1 = b_1,\ldots,d_g=b_g$). On a
\[
(f_*)^{\pd} = \dd\sum_{i = 1}^g f_*(d_i)c_i.
\]
D'après le Lemme \ref{tr}, nous avons $q(c_i) = q'(c_i)$ pour $i\geq k$. Ainsi,
  \[
  \begin{array}{c c l}
    \AA(f)(w) - \AA(f)(w') & = & (q-q')(\dd\sum_{i = 1}^g f_*(d_i)c_i) \\
                           & = & (q-q')(\dd\sum_{i = 1}^{k-1} f_*(b_i)[\RR\surg]_i)\\
                           & = & \dd\sum_{i=1}^{k-1}f_*(b_i)(w-w')([\RR\surg]_i)\\
                           & = & \dd\sum_{i=1}^{k-1}(\epsilon_f([\RR\surg]_i^{\pd})-\epsilon_f([\RR\surg]_0^{\pd}))(w-w')([\RR\surg]_i)
                         \end{array}
\]
d'après la Remarque \ref{epsf}. Puisque
  $\dd\sum_{i=1}^{k-1}(w-w')([\RR\surg]_i) =
  (w-w')([\RR\surg]_0)$, on a 
\[
\begin{array}{c c l}
    \AA(f)(w) - \AA(f)(w') & = & \epsilon_f\left(\dd\sum_{i=1}^{k-1}(w-w')([\RR\surg]_i)[\RR\surg]_i^{\pd}\right) \\
                           &   & - \epsilon_f\left(\left(\dd\sum_{i=1}^{k-1}(w-w')([\RR\surg]_i)\right)[\RR\surg]_0^{\pd}\right)\\
                           & = & \epsilon_f\left(\dd\sum_{i=0}^{k-1}(w-w')([\RR\surg]_i)[\RR\surg]_i^{\pd}\right)\\
                           & = & \epsilon_f(w-w')^{\pd}.
  \end{array}
  \]
  Ainsi
  \[
  \AA(f)(w) - \AA(f)(w') = \beta_0(f)(w) - \beta_0(f)(w').
  \]
\end{proof}

Autrement dit, un élément $f\in\modsig$ agit de la même façon sur les
classes de bordisme de structures $Spin$ réelles de première classe de
Stiefel-Whitney $w$ et sur les orientations de $\dd\bigotimes_{w_1(\RR
  L)([\RR \surg]_i) = 0} (\RR L)_{x_i}$ pour tout $L \in \rspin(\surg,w)$ et $x_i\in(\RR\surg)_i$, à
une constante (ne dépendant pas de $w$) près. C'est cette constante $s_{top}(f)$
qui nous intéresse.

\begin{Definition}
  \begin{itemize}
  \item Pour $f$ dans $\modsig$ et $w\in\hczdeux[g+1]{\RR\surg}$, on pose $s_{top}(f) := \beta_0(f)(w) +
    \AA(f,w) \in \ZZ/2\ZZ$.
  \item \'Etant donné $(N,c_N) \rightarrow (\surg,c_{\sur})$ un fibré
    en droites complexes muni d'une structure réelle, et
    $f\in\raut{N}$, on pose
    \[
    s_N(f) := s_{top}(f) + \beta_0(f)(w_1(\RR N)) \in\ZZ/2\ZZ.
    \]
  \end{itemize}
\end{Definition}

\begin{Exple}\label{sncalcul}
  Calculons la valeur de $s_{top}$ sur la famille $\mathcal{B}$. Fixons pour cela $q\in\rspin(\surg,w)$.
  \begin{itemize}
  \item Pour $i = 0,\ldots,k-1$, d'une part $\AA(f_i,w) = q([\RR\surg]_i) = 1-w([\RR\surg]_i)$, et d'autre part, comme $f_i$ n'est négative que sur la composante $(\RR\surg)_i$ de $\RR\surg$, $\beta_0(f)(w) = 1-w([\RR\surg]_i)$. Nous avons donc dans ce cas $s_{top}(f_i) = 0$.
  \item Si $(\surg,c_\sur)$ n'est pas séparante et $i \in\{k,\ldots,g\}$, d'une part $\AA(f_i,w) = q(a_i) = 1$ car $a_i$ est globalement stable (voir Lemme \ref{tr}), d'autre part $\beta_0(f)(w) = 0$ car $f_i$ est positive sur $\RR\surg$. Nous avons donc $s_{top}(f_i) = 1$ dans ce cas.
  \item Si $(\surg,c_\sur)$ est séparante et $i \in\{k,\ldots,k+m-1\}$, d'une part $\AA(f_i,w) = q(a_i+c_\sur(a_i)) = 0$ et $\AA(g_i,w) = q(b_i+c_\sur(b_i)) = 0$ d'après le Lemme \ref{lemspinre}, d'autre part $f_i$ est positive sur $\RR\surg$ donc $\beta_0(f_i) = 0$. Nous avons donc dans ce cas $s_{top}(f_i) = s_{top}(g_i) = 0$.
  \end{itemize}
\end{Exple}

\begin{Rem}\label{snrem}
  \begin{itemize}
\item Lorsque $\deg(N) = g+1\mod 2$, nous pouvons prendre $w_1(\RR N)\in\hczdeux[g+1]{\RR\surg}$ pour calculer $s_{top}$. Nous avons alors dans ce cas $s_N = \AA^{w_1(\RR N)}$.
 \item Comme $-1 = \dd\prod_{i=0}^{k-1}f_i \in \modsig$ (resp. $-1 = \dd\prod_{i=0}^gf_i \in\modsig$) lorsque la courbe est séparante (resp. non séparante), il suit de l'exemple \ref{sncalcul} que $s_{top}(-1) = 0$ (resp. $s_{top}(-1) = g-k+1 \mod 2$). D'après le Lemme \ref{lemmodsig}, si $g-k+1 = 0\mod 2$, $s_{top}$ passe au quotient en un morphisme de $F^-$ à valeurs dans $\ZZ/2\ZZ$. Il en est de même pour $s_N$ si $\deg(N) = g+1\mod 2$.
  \item Lorsque la courbe est séparante, l'Exemple \ref{sncalcul} et
    le Lemme \ref{lemmodsig} nous montrent que le morphisme $s_{top}$
    est nul (comparer avec \cite{grossharris}). Ainsi, dans ce cas, le signe $s_N(f)$ se calcule
    simplement en regardant le signe de $f$ sur les composantes de
    $\RR\surg$ où $\RR N$ est orientable et en en faisant le produit.
  \end{itemize}
  Dans le cas général, si nous choisissons une composante connexe
  particulière $(\RR\surg)_0$ de $\RR\surg$, la Remarque \ref{remprod}
  nous permet de définir $s_{top}$ sur le produit $\ZZ/2\ZZ \times
  F^-$. Si nous prenons une autre composante $(\RR\surg)_i$, alors le
  morphisme obtenu se déduit du précédent en le composant à droite par
  l'isomorphisme donné dans la Remarque \ref{remprod}.
\end{Rem}

Nous pouvons maintenant énoncer un résultat en tout degré.

\begin{Theoreme}\label{enoncesn}
  Soit $(N,c_N)$ un fibré en droites complexes sur $(\surg,c_\sur)$ de partie réelle non vide et $f\in\raut{N}$. Alors l'automorphisme $f$ préserve les orientations du fibré $\Det$ si et seulement si $s_N(f) = 0$. 
\end{Theoreme}

\begin{Corollaire}\label{enoncesep}
  Si $f$ est positive sur $\RR\surg$ et préserve une structure $Spin$ réelle de $\surg$, alors $f$ préserve les orientations de $\Det$.

Si $(\surg,c_\sur)$ est séparante, alors $f$ préserve les orientations de $\Det$ si et seulement si $f$ échange les orientations d'un nombre pair de composantes orientables de $\RR N$.
\end{Corollaire}
\begin{proof}
  Ceci suit du Théorème \ref{enoncesn}, du Lemme \ref{affine} et de la Remarque \ref{snrem}.
\end{proof}

Le Théorème \ref{enoncea} est un cas particulier du Théorème \ref{enoncesn} grâce à la Remarque \ref{snrem}. Remarquons aussi que le signe $s_N$ se décompose en une partie \og purement topologique \fg ne dépendant pas du fibré $N$ et une contribution de la partie réelle de $N$.

Pour démontrer le Théorème \ref{enoncesn} nous commençons par étudier le cas où $N$ est de degré $0$ et $w_1(\RR N)$ est nul en décrivant explicitement l'action des automorphismes sur les conoyaux des opérateurs de Cauchy-Riemann réels sur $N$. Dans un deuxième temps, nous réduisons le cas général au cas du degré $0$.

\subsubsection{Conoyaux d'opérateurs de Cauchy-Riemann réels}

L'objet de cette section est de décrire le conoyau réel d'un opérateur
de Cauchy-Riemann sur le fibré en droites complexes trivial
$(\TCC,conj)\rightarrow (\surg,c_{\sur})$. Nous supposons de plus que $g$ est non nul.

Pour une structure complexe $J$ sur $\surg$ qui rend $c_\sur$ anti-holomorphe, nous notons $\DDJ$
l'opérateur de Cauchy-Riemann réel $\dd\frac{1}{2}(\d + i\circ\d\circ J)$ sur le fibré en droites complexes trivial.  Dans le
cas où $(\surg,c_\sur)$ est séparante, nous aurons aussi besoin d'une
fonction $u : \surg\rightarrow \CC$, $\ZZ/2\ZZ$-équivariante, valant
$0$ au voisinage de $\RR\surg$, $i$ sur une hémisphère et $-i$ sur
l'autre. Cette fonction n'est pas unique, nous faisons donc un choix. Nous choisissons aussi une base symplectique réelle $(a_i,b_i)$ de $\hhz{\surg}$. Remarquons qu'elle induit une orientation sur chaque composante de $\RR\surg$. D'autre
part, les fonctions de la famille $\mathcal{B}$ construites au \S \ref{modsig} et que l'on
utilise ici ne sont pas (généralement) à prendre à isotopie près, mais
sont bien des représentants précis et choisis de leurs classes. En effet, si $\DB$ est un opérateur de Cauchy-Riemann réel sur $(\TCC,conj)$,
alors nous utiliserons le morphisme
\[
\begin{array}{c c c}
  \RR\CCC^{\infty}(\surg,\CC^*) & \rightarrow & H^1_{\DB}(\surg,\TCC)_{+1}\\
  f                            & \mapsto     & \dd\frac{\DB(f)}{f}.
\end{array}
\]
Lorsque $\DB = \DDJ$, une fonction $f\in\RR\CCC^{\infty}(\surg,\CC^*)$ admettant un logarithme $g\in\RR\CCC^{\infty}(\surg,\CC)$ s'envoie sur $\DDJ(g) = 0\in H^1_{\DDJ}(\surg,\TCC)_{+1}$ par ce morphisme. Cette application est donc bien définie sur $\modsig$ lorsque $\DB = \DDJ$. Ce n'est toutefois plus le cas lorsque $\DB \neq\DDJ$. En effet, un opérateur $\DB$ sur le fibré trivial s'écrit $\DDJ + \alpha$, avec $\alpha\in\Gamma(\surg,\Lambda_{J}^{0,1}\surg)_{+1}$, et pour $g\in\RR\CCC^{\infty}(\surg,\CC)$,
\[
\dd\frac{\DB(\exp(g))}{\exp(g)} = \DDJ(g) + \alpha
\]
ce qui n'est pas en général dans l'image de $\DB$.

\begin{Lemme}\label{optriv}
  \begin{itemize}
  \item[Si $(\surg,c_\sur)$ n'est pas séparante:] $H_{\DDJ}^1(\surg,\TCC)_{+1}$ est
    engendré par les formes $\DBS{f_1},\ldots,\DBS{f_g}$.
  \item[Si $(\surg,c_\sur)$ est séparante:] $H_{\DDJ}^1(\surg,\TCC)_{+1}$ est
    engendré par les formes \newline
    $\DBS{f_1},\ldots,\DBS{f_{k-1}},\DBS{f_k},u\DBS{f_k},\ldots,\DBS{f_{k+m-1}},u\DBS{f_{k+m-1}}$ ou
    encore \newline
    $\DBS{f_1},\ldots,\DBS{f_{k-1}},\DBS{g_k},u\DBS{g_k},\ldots,\DBS{g_{k+m-1}},u\DBS{g_{k+m-1}}$.
  \end{itemize}
\end{Lemme}

\begin{proof}
  Soit $(\omega_1,\ldots,\omega_g)$ une base de
  $H_{\DDJ}^0(\surg,K_{\sur})$ vérifiant
  \[
  \dd\int_{a_l}\omega_j = \frac{1}{2i\pi}\delta_{j,l}.
  \]
  Nous différencions les deux cas:
  \begin{itemize}
  \item[Si $(\surg,c_\sur)$ n'est pas séparante:] Posons $\omega'_j =
    \frac{\omega_j-\overline{c_{\sur}^*\omega_j}}{2}$; cette nouvelle famille
    forme une base de $H_{\DDJ}^0(\surg,K_{\sur})_{-1}$. D'autre part, en reprenant les notations du \S \ref{modsig}, on remarque que pour tout $l\in\{1,\ldots,g\}$, $f_{l|A_l}$ admet un logarithme $\log(f_{l|A_l})$ valant $i\pi$ sur $a_l$. On a alors
    \[
    \begin{array}{c c l}
    \dd\int_\sur\DBS{f_l}\wedge\omega'_j  & = & \dd\int_{A_l}\DBS{f_l}\wedge\omega'_j \\
                                         & = & \dd\int_{A_l}\d\left(\log(f_{l|A_l})\omega'_j\right)\\
                                         & = & \dd\int_{\partial(A_l)} \log(f_{l|A_l})\omega'_j \\
                                         & = & 2i\pi\dd\int_{a_l}\omega'_j\\
                                         & = & \delta_{j,l}.      
    \end{array}
    \]
  \item[Si $(\surg,c_\sur)$ est séparante:] Posons $\omega'_j =
    \frac{\omega_j-\overline{c_{\sur}^*\omega_j}}{2}$ pour $j = 1,\ldots,k+m-1$ et
    $\omega''_j = \frac{\omega_j + \overline{c_{\sur}^*\omega_j}}{2i}$ pour
    $j=k,\ldots,k+m-1$; de même, cette famille forme une base de
    $H_{\DDJ}^0(\surg,K_{\sur})_{-1}$. Dans ce cas nous avons
    \[
    \begin{array}{c}
      \dd\int \DBS{f_l}\wedge\omega'_j = \delta_{j,l},\text{ pour } j,l=1,\ldots,k+m-1\\
      \dd\int u\DBS{f_l}\wedge\omega''_j = \delta_{j,l},\text{ pour } j,l=k,\ldots,k+m-1.
    \end{array}
    \]
    Faisons le calcul dans le second cas, et fixons $j,l\in\{k,\ldots,k+m-1\}$. De même que précédemment, $f_{l|A_l}$ admet un logarithme $\log(f_{l|A_l})$ valant $i\pi$ sur $a_l$, et on a
    \[
    \begin{array}{c c l}
    \dd\int_{\sur}u\DBS{f_l}\wedge\omega''_j  & = & \dd\int_{A_l}i\DBS{f_l}\wedge\omega''_j + \dd\int_{A_{l+m}}-i\DBS{f_l}\wedge\omega''_j  \\
                                           & = &  \dd\int_{A_l}i\DBS{f_l}\wedge\omega''_j + \dd\int_{A_{l+m}}-\overline{c_{\sur}^*\left(i\DBS{f_l}\wedge\omega''_j\right)} \\
                                           & = & 2\Re\left(\dd\int_{A_l}i\DBS{f_l}\wedge\omega''_j\right)  \\
                                         & = & 2\Re\left(i\dd\int_{A_l}\d\left( \ln(f_{l|A_l})\omega''_j\right)\right)\\
                                         & = & 2\Re\left(i\dd\int_{\partial(A_l)} \ln(f_{l|A_l})\omega''_j\right) \\
                                         & = & -4\pi\Re\left(\dd\int_{a_l}\omega''_j\right)\\
                                         & = & 2\pi\Re\left(i\dd\int_{a_l}\omega_j - i\dd\int_{a_{l+m}}\omega_j\right)\\
                                         & = & \delta_{j,l}.      
    \end{array}
    \]
  \end{itemize}
  La dualité de Serre (voir \cite{wel1}) nous fournit alors le résultat.
\end{proof}

Nous allons maintenant rayonner dans $\ropj[\TCC]$ à partir
de l'opérateur $\DDJ$ en partant dans les directions fournies par
l'action des automorphismes construits au \S \ref{modsig} sur cet opérateur. Posons pour
$f\in\RR\CCC^{\infty}(\surg,\CC^*)$
\[
\DB_{J,f}^t = \DDJ + t \DBS{f}.
\]
On remarque que d'après le Lemme \ref{optriv}, si $\ind(f)\neq 0$, alors
le fibré holomorphe $(\TCC,\DB_{J,f}^t)$ est isomorphe au fibré holomorphe trivial
si et seulement si $t$ est entier. En effet, si $g\in\raut{\TCC}$ est un tel isomorphisme, alors nous avons
\[\label{fg}
\tag{$*$}
\DDJ + t\DBS{f} = g^*\DDJ = \DDJ + \DBS{g}.
\]
Si l'on décompose $f$ et $g$ grâce à la famille $\mathcal{B}$, $\DBS{f}$ et $\DBS{g}$ s'écrivent comme combinaisons linéaires à coefficients entiers, donnés par les indices de $f$ et $g$, des formes apparaissant dans le Lemme \ref{optriv}. L'égalité (\ref{fg}) impose donc que $t$ doit être entier.

 En particulier, lorsque $t$ n'est pas entier, le noyau de l'opérateur
$\DB_{J,f}^t$ est nul, et son conoyau est de dimension $g-1$. Plus
précisément, nous avons le résultat suivant.

\begin{Proposition}\label{conoyau}
Soit $(\surg,c_\sur)$ une surface de Riemann réelle de genre $g$ non nul et de structure complexe fixée $J\in\RR J(\surg)$. Si $t$ est un réel non entier, alors:
  \begin{itemize}
  \item[Si $(\surg,c_\sur)$ n'est pas séparante:] pour tout $l\in\{1,\ldots,g\}$, $H_{\DDJTL}^1(\surg,\TCC)_{+1}$ est
    engendré par les formes\newline
    $\alpha_1=\DBS{f_1},\ldots,\widehat{\DBS{f_l}},\ldots,\alpha_g = \DBS{f_g}$.
  \item[Si $(\surg,c_\sur)$ est séparante:] $H_{\DDJTL}^1(\surg,\TCC)_{+1}$ est
    engendré par les formes\newline
    $\alpha_1 = \DBS{f_1},\ldots,\widehat{\DBS{f_l}},\ldots,\alpha_{k-1}=\DBS{f_{k-1}},\alpha_{k} = \DBS{f_k},$ \newline $\alpha_{k+m} = u\DBS{f_k},\ldots,\alpha_{g-m} = \DBS{f_{g-m}},$ $\alpha_g = u\DBS{f_{g-m}}$ si $1\leq
    l\leq k-1$, \newline 
ou par les formes
    $\alpha_1 = \DBS{f_1},\ldots,\widehat{\DBS{f_l}},\widehat{u\DBS{f_l}},\ldots,$ $\alpha_{g-m} = \DBS{f_{g-m}},\alpha_g = u\DBS{f_{g-m}},
    \alpha_{l+m} = \DDJ(u)$ si $k \leq l \leq k+m-1 = g-m$.

    De même, $H_{\DDJTG}^1(\surg,\TCC)_{+1}$ est engendré par les formes\newline
    $\alpha_1 = \DBS{f_1},\ldots,\widehat{\DBS{g_l}},\widehat{u\DBS{g_l}},\ldots,\alpha_{g-m} = \DBS{g_{g-m}},\alpha_g = u\DBS{g_{g-m}},$ \newline
    $\alpha_{l+m} = \DDJ(u)$ pour $l=k,\ldots,k+m-1 = g-m$.
  \end{itemize}
\end{Proposition}

Nous commençons par un résultat intermédiaire. Prenons $f = f_l$, $l\in\{1,\ldots,g\}$, une des fonctions de la famille $\mathcal{B}$, et soit $\underline{\lambda} = (\lambda_1,\ldots,\hat{\lambda_l},\ldots,\lambda_g)$ un ($g-1$)-uplet de réels ($g$ est non nul). Posons
  \[
  \begin{array}{c c c l}
    F_{f,\underline{\lambda}}: & \RR J(\surg)\times \RR\setminus\ZZ\times L^{1,p}(\surg,\CC)_{+1} & \rightarrow & \mathcal{E}^{p}_{+1}\\
    & (J,t,v)       & \mapsto & \DDJT v -\dd\sum_{
    \begin{subarray}{c}
      j=1 \\ j\neq l
    \end{subarray}}^g \lambda_j \alpha_j.
  \end{array}
  \]
  où $p>2$ et $\mathcal{E}^p_{+1}$ est le fibré sur $\RR J(\surg)$ de fibre au-dessus de $J$ les $(0,1)$-formes pour $J$ à valeurs complexes $L^p(\surg,\Lambda_J^{0,1}\surg)_{+1}$. Considérons $\mathcal{M}_{f,\underline{\lambda}} = F_{f,\underline{\lambda}}^{-1}(\{0\})$.  Notons $\pi_{f,\underline{\lambda}} : \mathcal{M}_{f,\underline{\lambda}} \rightarrow \RR\setminus \ZZ$ la projection sur le second facteur. 

  \begin{Lemme}\label{lemmetexnik}
Soit $(\surg,c_\sur)$ une surface compacte réelle orientée de genre $g$ non nul. Supposons que $\underline{\lambda}$ n'est pas nul. Alors:
\begin{enumerate}
\item $\mathcal{M}_{f,\underline{\lambda}}$  est une sous-variété de $\RR J(\surg)\times \RR\setminus\ZZ\times L^{1,p}(\surg,\CC)_{+1}$. Plus précisément, $0$ est une valeur régulière de $F_{f,\underline{\lambda}}$.
\item La différentielle $\d\pi_{f,\underline{\lambda}}$ est surjective en tout point de $\mathcal{M}_{f,\underline{\lambda}}$.
\end{enumerate}
  \end{Lemme}
  \begin{proof}
Supposons que $\mathcal{M}_{f,\underline{\lambda}}$ est non vide. Nous montrons en fait un résultat qui implique directement les deux points du Lemme. Fixons $(J,t,v) \in \mathcal{M}_{f,\underline{\lambda}}$, et montrons que 
\begin{multline*}
  \d_{(J,t,v)}F_{f,\underline{\lambda}}(.,0,.) : (\dot{J},w) \in T_J\RR J(\surg) \times L^{1,p}(\surg,\CC)_{+1} \\
 \mapsto  \DDJT w + \dd\frac{1}{2}(i\circ\d v\circ \dot{J} + t\dd\frac{v}{f} i\circ\d f\circ \dot{J}) - \dd\sum_{
    \begin{subarray}{c}
      j=1 \\ j\neq l
    \end{subarray}}^g \lambda_j \dot{\alpha_j}(\dot{J}) \in L^{p}(\surg,\Lambda_J^{0,1}\surg)_{+1}
\end{multline*}
est surjective. Ici, $\dot{\alpha_j}(\dot{J})$ désigne $\dd\frac{1}{2}(\dd\frac{1}{f_j}i\circ\d f_j\circ\dot{J})$, $\dd\frac{1}{2}(\dd\frac{u}{f_j}i\circ\d f_j\circ\dot{J})$, ou $\dd\frac{1}{2}(i\circ\d u\circ\dot{J})$ selon que $\alpha_j$ vaut $\DBS{f_j}$, $u\DBS{f_j}$ ou $\DDJ(u)$.

Déjà, l'image de $\d_{(J,t,v)}F_{f,\underline{\lambda}}(.,0,.)$ est fermée et son conoyau est de dimension finie, car l'opérateur $\DDJT$ est Fredholm. Supposons que $\d_{(J,t,v)}F_{f,\underline{\lambda}}(.,0,.)$ n'est pas surjective et prenons donc $\omega \in L^{p'}(\surg,K_\sur)$ non nulle telle que pour tout $(\dot{J},w) \in T_J\RR J(\surg) \times L^{1,p}(\surg,\CC)_{+1}$
\[\label{alp}
\dd\int_\sur\omega\wedge\d_{(J,t,v)}F_{f,\underline{\lambda}}(\dot{J},0,w) = 0.
\tag{$*$}
\]
En particulier, en prenant $\dot{J}= 0$, nous avons pour tout $w\in L^{1,p}(\surg,\CC)_{+1}$
\[
\dd\int_\sur\omega\wedge \DDJT w = 0.
\]
D'après le Théorème C.2.3 de \cite{MDS}, $\omega\in L^{1,p}(\surg,K_\sur)$ et vérifie $\DDJ \omega + t \omega\wedge\dd\frac{\DB_J(f)}{f} = 0$. Par régularité elliptique et d'après le Lemme de prolongement unique (voir \cite{ivshev}, \cite{hoflizsik} ou \cite{audin} par exemple), $\omega$ est lisse et ne s'annule qu'en un nombre fini de points.

Supposons dans un premier temps que nous ne sommes pas dans le cas où la courbe est séparante et $l\geq k$ et $\lambda_{l+m}\neq 0$. Montrons qu'il existe un ouvert $U$ non vide de $\overline{\surg \setminus \dd\bigcup_{j = 1}^g A_j}$, invariant par $c_\sur$, et sur lequel $\omega$ et $\d v$ sont non nulles. Comme le lieu des zéros de $\d v$ est fermé, l'existence de $U$ est équivalente au fait que $\d v$ ne soit pas partout nulle sur $\overline{\surg \setminus \dd\bigcup_{j = 1}^g A_j}$. Nous raisonnons par l'absurde et supposons donc que $v$ est localement constante sur $\overline{\surg \setminus \dd\bigcup_{j = 1}^g A_j}$. Comme $\overline{\surg \setminus \dd\bigcup_{j = 1}^g A_j}$ est connexe, $v$ y est constante. Par hypothèse, il existe $n \in \{1\ldots,g\}\setminus\{l\}$ tel que $\lambda_n\neq 0$. Nous distinguons plusieurs cas.
\begin{itemize}
\item Supposons tout d'abord que la courbe n'est pas séparante. Découpons alors $\surg$ le long d'une courbe homologue à $0$ de sorte à obtenir un tore privé d'un disque sur lequel seule la fonction $f_n$ varie. En recollant un disque pour obtenir un tore $T$ (voir Figure \ref{tore}) nous pouvons y prolonger $v$ et $f_n$, et nous avons alors l'équation
\[
\DB_{J_T} v = \lambda_n\frac{\DB_{J_T}{f_n}}{f_n}
\]
sur le tore $T$. Or, d'après le Lemme \ref{optriv}, ceci n'est pas possible, et nous obtenons une contradiction.
\begin{figure}[h]
  \centering
  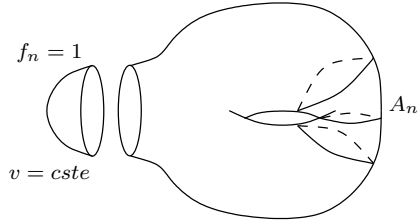
  \caption{Le tore $T$ découpé}
  \label{tore}
\end{figure}

\item Si la courbe est séparante et $n<k$, nous pouvons utiliser le même raisonnement.

\item Supposons enfin que la courbe est séparante, que $n\geq k$ et que si $l\geq k$ alors $l+m\neq n$. Découpons alors $\surg$ le long de deux courbes complexes conjuguées et homologues l'une de l'autre de sorte à obtenir une surface de genre deux privée de deux disques sur laquelle seule la fonction $f_n$ (resp. $f_{n-m}$) varie si $n\leq g-m$ (resp. si $n> g-m$). En recollant deux disques pour obtenir une surface de genre deux $\sur'$ (voir Figure \ref{tored}) nous pouvons y prolonger $v$, $u$ et $f_n$ (resp. $f_{n-m}$) si $n\leq g-m$ (resp. si $n>g-m$), et nous avons alors les équations
\[
\begin{array}{c}
\DB_{J_{\sur'}} v = \lambda_n\frac{\DB_{J_{\sur'}}{f_n}}{f_n} + \lambda_{n+m} u\frac{\DB_{J_{\sur'}}{f_n}}{f_n} \text{ si } n\leq g-m,\\
\DB_{J_{\sur'}} v = \lambda_{n-m}\frac{\DB_{J_{\sur'}}{f_{n-m}}}{f_{n-m}} + \lambda_{n} u\frac{\DB_{J_{\sur'}}{f_{n-m}}}{f_{n-m}} \text{ si } n> g-m,
\end{array}
\]
sur la surface $\sur'$. Or, d'après le Lemme \ref{optriv}, ceci n'est pas possible, et nous obtenons à nouveau une contradiction.
\begin{figure}[h]
  \centering
  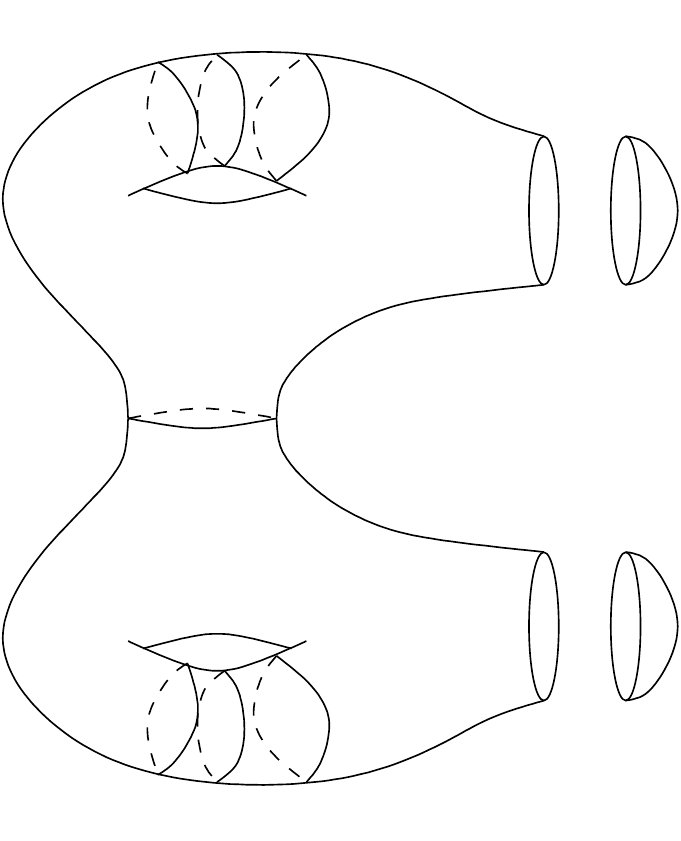
  \caption{La surface $\sur'$ découpée}
  \label{tored}
\end{figure}
\end{itemize}

Prenons alors $U \subset \overline{\surg - \dd\bigcup_{j = 1}^g A_j}$ un ouvert invariant par $c_\sur$ tel que $\omega_{|U}$ et $\d v_{|U}$ ne s'annulent pas, et $\xi$ une $(0,1)$-forme pour $J$ à valeurs complexes, $\ZZ/2\ZZ$-équivariante sur $\surg$, à support inclus dans $U$ et telle que $\dd\int_\sur\omega\wedge\xi \neq 0$.
Comme $\d v$ ne s'annule pas sur $U$, posons $\dot{J} = (\d v)^{-1}\circ (-i)\circ \xi \in T_J\RR J(\surg)$ et nous avons
\[
\dd\int_\sur\omega\wedge\d_{(J,t,v)}F_{f,\underline{\lambda}}(\dot{J}, 0,0) = \dd\int_\sur\omega\wedge\xi \neq 0,
\]
car le support de $\dot{J}$ est inclus dans $\overline{\surg - \dd\bigcup_{j = 1}^g A_j}$. Ceci est en contradiction avec (\ref{alp}), et $\d_{(J,t,v)}F_{f,\underline{\lambda}}(.,0,.)$ est donc surjective.

 Supposons maintenant que la courbe est séparante, que $l\geq k$ et que $\lambda_{l+m}$ est non nul. Montrons qu'il existe un ouvert $U$ non vide de $\overline{\surg \setminus \dd\bigcup_{j = 0}^g A_j}$, invariant par $c_\sur$, et sur lequel $\omega$ et $\d v$ sont non nulles. On raisonne à nouveau par l'absurde, et on suppose que $v$ est localement constante sur $\overline{\surg \setminus \dd\bigcup_{j = 0}^g A_j}$. Sur  $\overline{\surg \setminus \dd\bigcup_{j = 1}^g A_j}$, $v$ vérifie $\DDJ v = \lambda_{l+m}\DDJ u$.
Il existe donc $c\in\RR$ tel que $v = \lambda_{l+m} u+c$ sur $\overline{\surg \setminus \dd\bigcup_{j = 1}^g A_j}$ car $\d(v-u)$ s'annule un nombre infini de fois sur ce même ensemble. Si l'on découpe maintenant un tore contenant $A_l$ et qu'on le rebouche comme précédemment, nous avons sur ce tore $T$
\[
\DB_{J_T} v + tv\dd\frac{\DB_{J_T} f_l}{f_l} = 0.
\]
Or, l'opérateur $\DB_{J_T,f_l}^t$ est injectif lorsque $t$ n'est pas entier. En effet, dans le cas contraire il existerait $g\in\CCC^\infty(T,\CC^*)$ tel que $t\dd\frac{\DB_{J_T} f_l}{f_l} = \dd\frac{\DB_{J_T} g}{g}$. En utilisant les mêmes techniques que dans le Lemme \ref{lemmodsig}, on peut démontrer que $g$ est homotope à un produit de la forme $(f_l)^p(g_l)^q$ avec $p,q\in\ZZ$. On a donc l'égalité
\[\label{eg}
t\dd\frac{\DB_{J_T} f_l}{f_l} = \dd\frac{\DB_{J_T} g}{g} = p\dd\frac{\DB_{J_T} f_l}{f_l} + q\dd\frac{\DB_{J_T} g_l}{g_l}\tag{$*$}
\]
dans $H^1_{\DB_{J_T}}(T,\TCC)$. Prenons alors $\omega\in H^0_{\DB_{J_T}}(T,K_T)$ tel que 
\[
\int_{a_l}\omega = \frac{1}{2i\pi}.
\]
On fait le produit extérieur par $\omega$ puis on intègre l'égalité (\ref{eg}) pour obtenir comme dans le Lemme \ref{optriv}
\[
t = p + q\int_{b_l}\omega.
\]
Le second terme de cette égalité est entier ou complexe alors que $t$ ne l'est pas. Donc $\DB_{J_T,f_l}^t$ est bien injectif et $v$ est nul sur $T$. Toutefois, ceci n'est pas possible car on aurait alors $c = \lambda_{l+m} i \in\RR$.

On raisonne ensuite de même que précédemment pour montrer que $\d_{(J,t,v)}F_{f,\underline{\lambda}}(.,0,.)$ est surjective.

Ceci nous donne immédiatement le premier point du Lemme. Quant au deuxième, $\d_{(J,t,v)}\pi_{f,\underline{\lambda}}$ est surjective si et seulement si pour tout $\dot{t}\in\RR$ il existe $(\dot{J},w) \in T_J\RR J(\surg) \times L^{1,p}(\surg,\CC)_{+1}$ tel que $(\dot{J},\dot{t},w)\in T_{(J,t,v)}\MM_{f,\underline{\lambda}} = \ker(\d_{(J,t,v)} F_{f,\underline{\lambda}})$. Autrement dit, $\d_{(J,t,v)}\pi_{f,\underline{\lambda}}$ est surjective si et seulement si
\[
\forall \dot{t}\in\RR,\ \exists(\dot{J},w) \in T_J\RR J(\surg) \times L^{1,p}(\surg,\CC)_{+1},\ \d_{(J,t,v)}F_{f,\underline{\lambda}}(\dot{J},0,w)= - \d_{(J,t,v)}F_{f,\underline{\lambda}}(0,\dot{t},0).
\]
Mais cette dernière assertion est vraie car $\d_{(J,t,v)}F_{f,\underline{\lambda}}(.,0,.)$ est surjective.
  \end{proof}

\begin{proof}[Démonstration de la Proposition \ref{conoyau}]
  {\bf Premier cas : le support de $\dd\frac{\DDJ(f_l)}{f_l}$ est connexe.} Prenons $f = f_l$ correspondant à une courbe stable ou
  globalement stable $a = a_l$, et soient $\lambda_1,\ldots,\hat{\lambda_l},\ldots,\lambda_g$ des réels.

  Nous établissons tout d'abord le résultat quand $t$ est
  rationnel. \'Ecrivons donc $t = \frac{p}{q}$ ($p\wedge q = 1$, et
  $q>1$). Supposons qu'il existe une fonction $v_t$
  vérifiant
  \[
  \DDJT (v_t) = \dd\sum_{
    \begin{subarray}{c}
      j=1 \\ j\neq l
    \end{subarray}}^g \lambda_j \alpha_j
  \]
  Notons $r_q:\widetilde{\surg}^q \rightarrow \surg$ le revêtement à
  $q$ feuillet associé à $a\bullet .\in
  \Hom(\hhz{\surg},\ZZ/q\ZZ)$ (voir Figure \ref{revet}). Alors le relevé $\tilde{f} = f\circ r_q$ de $f$ sur
  $\widetilde{\surg}^q$ admet une puissance $t$-ième (non unique, on
  en choisit une notée $\tilde{f}^t$). On note $\tilde{v}_t$ et
  $\tilde{f_j}$ les relevés respectifs de $v_t$ et $f_j$, $j\in\{1,\ldots,g\}\setminus\{l\}$; $\DDJ$ désignera
  aussi le relevé de $\DDJ$. Les fonctions $\tilde{f_j}$ s'écrivent
  $\tilde{f_j} = f_{j,0}\times\ldots\times f_{j,q-1}$, où les $f_{j,n}$ sont des
  fonctions du même type que celles construites au \S
  \ref{modsig}, à partir des relevés $a_{j,0},\ldots,a_{j,q-1}$ de $a_j$; chacune ne varie donc que dans un des
  feuillets (voir Figure \ref{revet}).
  \begin{figure}[h]
    \centering
    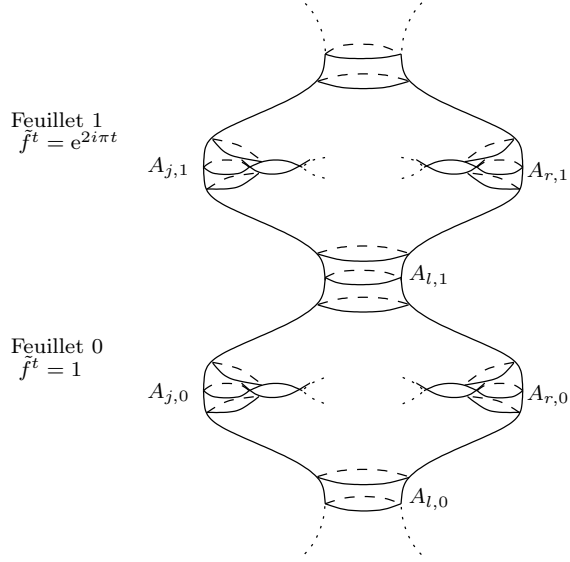
    \caption{Revêtement et relevés}
    \label{revet}
  \end{figure}

 Celles-ci sont numérotées de sorte que
  $\tilde{f}^t\alpha_{j,n} = \e^{2i\pi nt}\alpha_{j,n}$. Comme
\[
\DDJ(\tilde{v}_t\tilde{f}^t) = \tilde{f}^t\DB^t_{J,\tilde{f}}(\tilde{v}_t),
\]
on a
  \[
  \DDJ(\tilde{v}_t\tilde{f}^t) = \dd\sum_{    
    \begin{subarray}{c}
      j=1 \\ j\neq l
    \end{subarray}}^g\lambda_{j}\dd\sum_{n=0}^{q-1}\e^{2i\pi nt}\alpha_{j,n}.
  \]
  Or, si le terme de gauche est dans l'image de l'opérateur $\DDJ$, ce
  n'est pas le cas de celui de droite si un des $\lambda_j$ est non nul. Ceci se vérifie en effet en
  utilisant par exemple une famille libre de $q(g-1)$ formes holomorphes $(\omega_{j,n})_{
    \begin{subarray}{c}
      1\leq j\neq l\leq g\\ 0\leq n\leq q-1
    \end{subarray}
}$ sur $\widetilde{\surg}^q$ associée à la famille $(a_{j,n})_{
    \begin{subarray}{c}
      1\leq j\neq l\leq g\\ 0\leq n\leq q-1
    \end{subarray}
}$, c'est-à-dire telles que 
\[
\dd\int_{a_{j,n}}\omega_{r,s} = \dd\frac{1}{2i\pi}\delta_{(j,n),(r,s)}.
\]
Alors, pour tout $j,r\in\{1,\ldots,g\}\setminus\{l\}$ et $n,s\in\{0,\ldots,q-1\}$, on a comme dans le Lemme \ref{optriv} 
\[
  \dd\int_{\widetilde{\surg}^q}\alpha_{j,n}\wedge\omega_{r,s} = \dd\int_{\widetilde{\surg}^q}\DBS{f_{j,n}}\wedge\omega_{r,s} = \delta_{(j,n),(r,s)},
\]
si $(\surg,c_\sur)$ n'est pas séparante ou si $(\surg,c_\sur)$ est séparante et $j,r\leq k+m-1$, et
\[
  \dd\int_{\widetilde{\surg}^q}\alpha_{j,n}\wedge\omega_{r,s} = \dd\int_{\widetilde{\surg}^q}u\DBS{f_{j-m,n}}\wedge\omega_{r,s} = i\delta_{(j-m,n),(r,s)},
\]
si $(\surg,c_\sur)$ est séparante et $j > k+m-1$ et $r\leq k+m-1$. Nous pouvons maintenant calculer
\[
  \dd\int_{\widetilde{\surg}^q}\omega_{r,s}\wedge\DDJ(\tilde{v_t}\tilde{f}^t) = \dd\int_{\widetilde{\surg}^q}\omega_{r,s}\wedge\left(\dd\sum_{    
    \begin{subarray}{c}
      j=1 \\ j\neq l
    \end{subarray}}^g\lambda_j\dd\sum_{n=0}^{q-1}\e^{2i\pi nt}\alpha_{j,n}\right).
\]
ce qui donne
\[
\begin{array}{c}
0 = \lambda_{r}\e^{2i\pi st} \text{ si }(\surg,c_\sur) \text{ n'est pas séparante ou si } (\surg,c_\sur) \text{ est séparante et } r<k\\
0 = (\lambda_{r}+i\lambda_{r+m})\e^{2i\pi st} \text{ si } (\surg,c_\sur)\text{ est séparante et } k\leq r\leq k+m-1.
\end{array}
\]
Ainsi tous les $\lambda_j$ sont nuls.

Passons maintenant au cas où $t$ n'est plus forcément rationnel. Si nous supposons par l'absurde que $\mathcal{M}_{f,\underline{\lambda}}$ est non vide et que $\underline{\lambda}$ n'est pas nul, alors d'après le Lemme \ref{lemmetexnik}, $\pi_{f,\underline{\lambda}}$ est une application ouverte et son image contient donc un rationnel. Mais ceci contredit la première partie de la preuve. $\mathcal{M}_{f,\underline{\lambda}}$ est donc vide, ce qui conclut ce premier cas.

  {\bf Deuxième cas : le support de $\dd\frac{\DDJ(f_l)}{f_l}$ a deux composantes connexes.} Prenons $f = f_l$ correspondant à une paire de
  courbes simples notée $c$. Nous ne montrons ici que le fait que la
  forme $\alpha_{l+m} = \DDJ(u)$ n'est pas dans l'image de l'opérateur $\DDJT$, le
  reste se montrant comme dans le premier cas.

  Commençons par le cas $t = \frac{p}{q}$, $p\wedge q = 1$ et
  $q>1$. Supposons par l'absurde l'existence d'une fonction $v_t$
  vérifiant
  \[
  \DDJT (v_t) = \DDJ (u).
  \]
  En raisonnant comme précédemment sur le revêtement à $q$ feuillets
  associé à $c\bullet .$, et en dénotant avec un tilde les
  relevés des diverses fonctions, voir Figure \ref{revetsep}, on a
  \[
  \DDJ(\tilde{v}_t\tilde{f}^t) = \tilde{f}^t\DDJ(\tilde{u}).
  \]
  La surface $\widetilde{\surg}^q$ est séparée en deux composantes connexes par les relevés de $\RR\surg$.
  \begin{figure}[h]
    \centering
    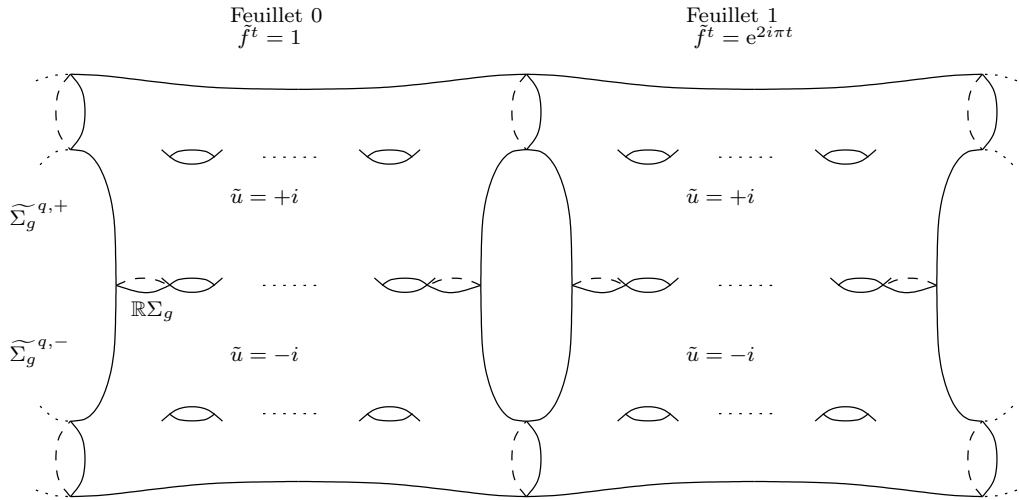
    \caption{Revêtement et relevés pour une courbe séparante}
    \label{revetsep}
  \end{figure}

 Prenons
  $\widetilde{\surg}^{q,+}$ celle où la partie imaginaire de
  $\tilde{u}$ est positive, et restreignons-y l'équation précédente. Définissons la fonction $\tilde{u}^+$ par $\tilde{u}^+ =
  \dd\sum_{n=0}^{q-1}u_n$, où $u_n = \tilde{u}$ sur le $n$-ième
  feuillet et $u_n = i$ ailleurs. 
Ainsi $\DDJ(\tilde{u}^+) =
  \DDJ(\tilde{u})$, et $\tilde{f}^t\DDJ(\tilde{u}^+) =
  \dd\sum_{n=0}^{q-1}\e^{2i\pi nt}\DDJ(u_n)$, de sorte que nous obtenons une
  fonction holomorphe $h_t$ sur $\widetilde{\surg}^{q,+}$ donnée par
  \[
  h_t = \tilde{f}^t\tilde{v}_t - \dd\sum_{n=0}^{q-1}\e^{2i\pi nt}u_n.
  \]
  De plus,
  \[
  h_t(\partial(\widetilde{\surg}^{q,+})) \subset
  \dd\bigcup_{n=0}^{q-1}\left(\e^{2i\pi nt}(i+\RR)\right) = P_q,
  \]
car sur une composante réelle du $n$-ième feuillet, $n\in\{0,\ldots,q-1\}$, on a $\tilde{v}_t\in\RR$, $\tilde{f}^t = \e^{2i\pi nt}$, $u_n = 0$ et $u_j = i$ si $j\neq n$, tandis que $\dd\sum_{\begin{subarray}{c} 
j=0 \\
j\neq n
\end{subarray}}^{q-1}\e^{2i\pi jt} = -\e^{2i\pi nt}$.

Comme $h_t$ est une application ouverte, nous avons $\partial(h_t(\widetilde{\surg}^{q,+}))\subset  h_t(\partial(\widetilde{\surg}^{q,+}))$, et comme $\widetilde{\surg}^{q,+}$ est compacte, l'image de $h_t$ est compacte. Prenons un disque $B$ dans $\CC$ centré en $0$ contenant l'image de $h_t$ et de rayon minimal. Notons $I$ un des points d'intersection du bord de ce disque avec le bord de l'image de $h_t$. Alors $I$ est un point d'intersection entre deux droites de $P_q$. En effet, d'une part $I$ est dans $h_t(\partial(\widetilde{\surg}^{q,+}))$ qui est inclus dans $P_q$, d'autre part si $I$ n'était pas à l'intersection de deux droites, alors on aurait un petit intervalle centré en $I$ sur la droite contenant $I$ qui serait inclus dans le bord de l'image de $h_t$. Ceci contredirait la minimalité du rayon du disque $B$. 

Notons $D$ et $D'$ les deux droites de $P_q$ s'intersectant en $I$. Comme $P_q$ ne contient qu'un nombre fini de droites, au voisinage de $I$ le bord de l'image de $h_t$ est inclus dans $D\cup D'$ et l'intérieur de l'image est contenue dans un des secteurs découpés par $D$ et $D'$ (voir Figure \ref{imageh}). 

  \begin{figure}[h]
    \centering
    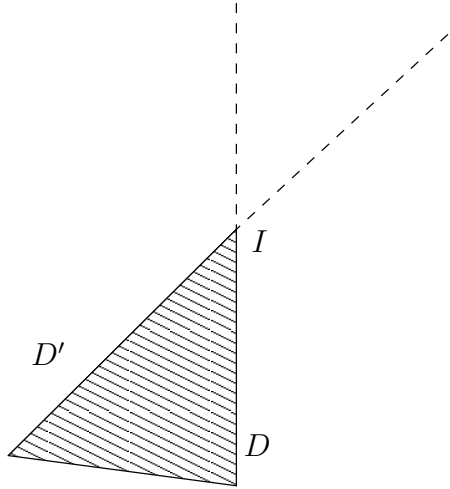
    \caption{Image de $h_t$ au voisinage de $I$ (partie hachurée)}
    \label{imageh}
  \end{figure}

Prenons $n\in\{0,\ldots,q-1\}$ tel que $D = \e^{2i\pi nt}(i+\RR)$ et considérons la restriction $h_{n,t}$ de $h_t$ au feuillet $n$. Prenons un petit voisinage $V$ de $I$. Comme $h_{n,t}$ est une application ouverte, $V\cap \im(h_{n,t})$ est d'intérieur non vide. Le bord de l'image de $h_{n,t}$ au voisinage de $I$ n'est donc pas complètement inclus dans $D$. Ainsi, dans tout voisinage de $I$ il existe un point qui est dans $\partial(\im(h_{n,t}))$ mais pas dans $D$. On peut donc construire une suite de points $(I_r)_{r\in\NN}$ de $\partial(\im(h_{n,t}))\setminus D$ qui converge vers $I$. De plus, les antécédents de ces points par $h_{n,t}$ sont dans le bord du $n$-ième feuillet mais pas dans la partie réelle. Quitte à extraire, on obtient donc une suite de points du $n$-ième feuillet qui converge vers un point $x$ qui est dans l'intérieur de $\widetilde{\surg}^{q,+}$ et tel que $h_t(x) = I$. Ceci est en contradiction avec le fait que $h_t$ est une
  application ouverte sur $\widetilde{\surg}^{q,+}$.

  Le cas où $t$ n'est pas rationnel découle comme
  précédemment du Lemme \ref{lemmetexnik}.
\end{proof}

\subsubsection{Action des automorphismes réels sur le fibré déterminant}

Passons maintenant à l'étude de l'action des automorphismes réels sur les orientations du fibré déterminant. Démontrons tout d'abord le résultat suivant.

\begin{Proposition}\label{enonceaction}
Soit $(N,c_N)$ un fibré vectoriel complexe de rang $1$ sur $(\surg,c_\sur)$ de partie réelle non vide. Fixons une base symplectique réelle $(a_i,b_i)$ de $\hhz{\surg}$ et notons $\mathcal{B}$ la famille génératrice de $\raut{N}$ associée (voir \S \ref{modsig}). Nous avons alors les deux cas suivant.
\begin{itemize}
\item  Si la courbe $(\surg,c_\sur)$ est séparante, l'automorphisme $f_j\in\mathcal{B}$, $0\leq j\leq k-1$ renverse les orientations du fibré $\Det$ si et seulement si $(\RR N)_j$ est orientable, alors que les automorphismes $f_k,g_k,\ldots,f_{m+k-1},g_{m+k-1}\in\mathcal{B}$ les préservent toujours.

\item Si la courbe $(\surg,c_\sur)$ n'est pas séparante, l'automorphisme $f_j\in\mathcal{B}$, $0\leq j\leq k-1$ renverse les orientations du fibré $\Det$ si et seulement si $(\RR N)_j$ est orientable, alors que les automorphismes $f_k,\ldots,f_g\in\mathcal{B}$ les renversent toujours.
\end{itemize}
\end{Proposition}

\begin{Rem}\label{casgz}
Le cas où le genre $g$ est nul se traite aisément. Nous avons dans ce cas un seul élément dans $\mathcal{B}$, et celui-ci est homotope à la fonction constante égale à $-1$. Or celle-ci préserve les orientations si et seulement si $\dim(H^0(\surg,N)_{+1})-\dim(H^1(\surg,N)_{+1}) = \deg(N) + 1$ est pair, donc si et seulement si $\RR N$ n'est pas orientable.
\end{Rem}

La démonstration de la Proposition \ref{enonceaction} se fait en deux temps. Nous commençons par traiter le cas du fibré trivial, puis nous utilisons des transformations élémentaires négatives pour nous ramener à ce cas.

\paragraph{Le cas du fibré trivial}

\begin{Proposition}\label{castr}
Soit $(\surg,c_\sur)$ une courbe de partie réelle non vide. Fixons une base symplectique réelle $(a_i,b_i)$ de $\hhz{\surg}$ et notons $\mathcal{B}$ la famille génératrice de $\raut{\TCC}$ associée (voir \S \ref{modsig}). Nous avons alors les deux cas suivant.
\begin{itemize}
\item  Si la courbe $(\surg,c_\sur)$ est séparante, l'automorphisme $f_j\in\mathcal{B}$, $0\leq j\leq k-1$ renverse les orientations du fibré $\Det[\TCC]$, alors que les automorphismes $f_k,g_k,\ldots,f_{m+k-1},g_{m+k-1}\in\mathcal{B}$ les préservent.

\item Si la courbe $(\surg,c_\sur)$ n'est pas séparante, tous les éléments de $\mathcal{B}$ renversent les orientations du fibré $\Det[\TCC]$.
\end{itemize}
\end{Proposition}

Fixons une structure complexe $J\in\RR J(\surg)$. La Proposition \ref{conoyau} donne une trivialisation du fibré $\Det[\TCC]$ au-dessus des intervalles ouverts $(\DB^t_{J,.})_{t\in ]n,n+1[}$, $n\in\ZZ$. Toutefois, sous l'action d'un des automorphismes $f$ considérés dans l'énoncé, un intervalle ouvert $(\DB^t_{J,f})_{t\in ]n,n+1[}$, $n\in\ZZ$, est envoyé sur l'intervalle consécutif $(\DB^t_{J,f})_{t\in ]n+1,n+2[}$. Avant de démontrer la Proposition \ref{castr}, commençons par décrire le rapport entre les trivialisations du fibré $\Det[\TCC]$ au-dessus de ces intervalles consécutifs. Nous introduisons pour cela les opérateurs suivants:
\begin{itemize}
\item[Si $(\surg,c_\sur)$ n'est pas séparante:]
  \begin{multline*}
    \Phi_{f_l,0}^t: (v,\lambda_1,\ldots,\widehat{\lambda_l},\ldots,\lambda_{g})\in L^{1,p}(\surg,\CC)\times\RR^{g-1} \\ \mapsto  \DDJTL(v) + \dd\sum_{\begin{subarray}{c}n=0\\n\neq l\end{subarray}}^g \lambda_n\DBS{f_n}\in L^p(\surg,\Lambda_J^{0,1}(\surg))
  \end{multline*}
\item[Si $(\surg,c_\sur)$ est séparante et $1\leq l\leq k-1$:]
  \begin{multline*}
    \Phi_{f_l,1}^t: (v,\lambda_1,\ldots,\widehat{\lambda_l},\ldots,\lambda_{g})\in L^{1,p}(\surg,\CC)\times\RR^{g-1} \\
 \mapsto \DDJTL(v) + \dd\sum_{\begin{subarray}{c}n=0\\n\neq l\end{subarray}}^{k-1} \lambda_n\DBS{f_n}\hspace{4cm} \\
+ \dd\sum_{n=k}^{m+k-1} \left(\lambda_n\DBS{f_n} + \lambda_{n+m}u\DBS{f_n}\right) \in L^p(\surg,\Lambda_J^{0,1}(\surg))
  \end{multline*}
\item[Si $(\surg,c_\sur)$ est séparante et $k\leq l\leq k+m-1$:]
  \begin{multline*}
    \Phi_{f_l,1}^t: (v,\lambda_1,\ldots,\widehat{\lambda_l},\ldots,\lambda_{g})\in L^{1,p}(\surg,\CC)\times\RR^{g-1}\\
 \mapsto  \DDJTL(v) + \dd\sum_{n=0}^{k-1} \lambda_n\DBS{f_n} \hspace{8cm}\\
 + \dd\sum_{\begin{subarray}{c}n=k\\n\neq l\end{subarray}}^{m+k-1}\left(\lambda_n\DBS{f_n} + \lambda_{n+m}u\DBS{f_n}\right)+\lambda_{l+m}\DDJ(u)\in L^p(\surg,\Lambda_J^{0,1}(\surg))
  \end{multline*}
  et
  \begin{multline*}
    \Phi_{g_l,1}^t: (v,\lambda_1,\ldots,\widehat{\lambda_l},\ldots,\lambda_{g})\in L^{1,p}(\surg,\CC)\times\RR^{g-1} \\
\mapsto  \DDJTG(v) + \dd\sum_{n=0}^{k-1} \lambda_n\DBS{f_n}  \hspace{8cm}\\
+ \dd\sum_{\begin{subarray}{c}n=k\\n\neq l\end{subarray}}^{m+k-1}\left(\lambda_n\DBS{g_n} + \lambda_{n+m}u\DBS{g_n}\right)+ \lambda_{l+m}\DDJ(u)\in L^p(\surg,\Lambda_J^{0,1}(\surg)).
  \end{multline*}
\end{itemize}
Ceux-ci sont Fredholm d'indice $0$ et sont
surjectifs lorsque $t$ n'est pas entier d'après la Proposition \ref{conoyau}. Ces familles d'opérateurs
rencontrent toutes des murs lorsque $t$ est entier: leurs noyaux
deviennent d'un coup non triviaux, de dimension $1$ pour les deux premiers et de dimension $2$ pour les deux derniers.

D'autre part, les déterminants des opérateurs $\Phi^t$ sont par
définition canoniquement isomorphes à ceux des opérateurs $\DDJ^t$ correspondant. On a donc (dans tous les cas) pour $t$ non entier
\[
\ddet(\DDJ^t) = \ddet(\Phi^t) = \RR.
\]

Le Lemme suivant permet de décrire les trivialisations du fibré $\Det[\TCC]$ au dessus des droites $\DDJ^t$.
\begin{Lemme}\label{traverse}
  Les familles d'opérateurs $(\Phi^t_{.,.})_{t\in\RR}$ ont des traversées de murs
  régulières, c'est-à-dire que leurs dérivées par rapport à $t$ aux
  valeurs entières fournissent des isomorphismes entre leurs noyaux et
  leurs conoyaux.
\end{Lemme}

\begin{proof}
Montrons le par exemple pour $t = 0$. Si $(\surg,c_\sur)$ n'est pas séparante et $1\leq l\leq g$ (resp. si $(\surg,c_\sur)$ est séparante et $1\leq l\leq k-1$), alors le noyau de $\Phi_{f_l,0}^0$ (resp. $\Phi^t_{f_l,1}$) est engendré par $(1,0,\ldots,0)$ et son conoyau par $\DBS{f_l}$ d'après le Lemme \ref{optriv}. Or
\[
\begin{array}{l}
\dot{\Phi}_{f_l,0}^0(1,0,\ldots,0) = \DBS{f_l}\\
\dot{\Phi}_{f_l,1}^0(1,0,\ldots,0) = \DBS{f_l}.
\end{array}
\]
Ce qui conclut ce cas.

 Si $(\surg,c_\sur)$ est séparante et $k\leq l\leq k+m-1$, alors le noyau de $\Phi_{f_l,1}^0$ est engendré par $(1,0,\ldots,0)$ et $(u,0\ldots,0,-1,0,\ldots,0)$ (le $-1$ est en position $l+m$) et son conoyau par $\DBS{f_l}$ et $u\DBS{f_l}$ d'après le Lemme \ref{optriv}. Or
\[
\begin{array}{l}
\dot{\Phi}_{f_l,1}^0(1,0,\ldots,0) = \DBS{f_l}\\
\dot{\Phi}_{f_l,1}^0(u,0,\ldots,0,-1,0,\ldots,0) = u\DBS{f_l}.
\end{array}
\]
Ce qui termine la démonstration.
\end{proof}

Le Lemme \ref{traverse} nous permet d'énoncer la Proposition suivante, tirée de \cite{MDS} (Proposition
A.2.4 p.499), qui nous renseigne sur la trivialisation du fibré
$\Det[\TCC]$ au dessus des familles d'opérateurs voulues.
\begin{Proposition}[\cite{MDS}]\label{MDS}
  Soient $t_1 < t_2$ deux réels non entiers séparés par un unique
  entier $t_0$. Toute trivialisation du fibré $\Det[\TCC]$
  au-dessus du chemin $\DDJ^t$ fournit un isomorphisme
  $\ddet(\DDJ^{t_1}) = \ddet(\Phi^{t_1}) =
  \RR\rightarrow\ddet(\DDJ^{t_2}) = \ddet(\Phi^{t_2}) = \RR$ de signe
  donné par la parité de la dimension du noyau de $\Phi^{t_0}$.\qed
\end{Proposition}

\begin{proof}[Démonstration de la Proposition \ref{castr}]
  Fixons une structure complexe $J\in\RR J(\surg)$. Supposons que $g$ est non nul, le cas de la sphère étant traité dans la Remarque \ref{casgz}. Notons $f$ un des automorphismes considérés dans l'enoncé. Prenons l'opérateur de Cauchy-Riemann réel $\DB_{J,f}^{-\frac{1}{2}}$ sur
$(\TCC,conj)$. En agissant sur cet opérateur, l'automorphisme $f\in
\RR\CCC^{\infty}(\surg,\CC^*)$ fournit d'après la formule de Leibniz
l'opérateur
\[
f^*(\DB_{J,f}^{-\frac{1}{2}}) = \DB + \frac{1}{2}\dd\frac{\DB(f)}{f} = \DB_{J,f}^{\frac{1}{2}},
\]
et un isomorphisme
\[
\label{isodet}
\tag{$*$}
\ddet(\DB_{J,f}^{-\frac{1}{2}}) = \ddet(\Phi_{f,\varepsilon(\surg)}^{-\frac{1}{2}}) = \RR\rightarrow\ddet(\DB_{J,f}^{\frac{1}{2}}) =
\ddet(\Phi_{f,\varepsilon(\surg)}^{\frac{1}{2}}) = \RR.
\]
Comme $f$ envoie la base de $H^1_{\DB_{J,f}^{-\frac{1}{2}}}(\surg,\TCC)_{+1}$ donnée dans la Proposition \ref{conoyau} sur la base de $H^1_{\DB_{J,f}^{\frac{1}{2}}}(\surg,\TCC)_{+1}$ donnée de la même façon et que $\Phi_{f,\varepsilon(\surg)}^{-\frac{1}{2}}$ et $\Phi_{f,\varepsilon(\surg)}^{\frac{1}{2}}$ sont construits à l'aide de ces bases, l'isomorphisme (\ref{isodet}) est positif. D'après la Proposition \ref{MDS}, il ne coïncide avec celui induit par la trivialisation du fibré déterminant que lorsque la dimension du noyau de $\Phi_{f,\varepsilon(\surg)}^0$ est paire.

Nous retrouvons ainsi les différents cas mentionnés dans l'énoncé de la Proposition.
\begin{itemize}
\item Si $(\surg,c_\sur)$ est séparante, les noyaux des $\Phi^0$ induis par les automorphismes $f_k,g_k,\ldots,f_{k+m-1},g_{k+m-1}$ sont tous de dimension $2$. Ces automorphismes préservent donc les orientations de $\Det[\TCC]$.
\item Si $(\surg,c_\sur)$ n'est pas séparante, le noyau de $\Phi_{f_j,1}^0$, $j\in\{k,\ldots,g\}$, est de dimension $1$. Un automorphisme $f_j$, $j\in\{k,\ldots,g\}$ renverse donc les orientations de $\Det$.
\item Le noyau de $\Phi_{f_j,1}^0$, $j\in\{0,\ldots,k-1\}$, est de dimension $1$. L'automorphisme $f_j$, $j\in\{0,\ldots,k-1\}$ renverse donc les orientations de $\Det[\TCC]$.
\end{itemize}
\end{proof}

\paragraph{Transformations élementaires}\label{trans}

Nous considérons maintenant $(N,c_N)$ un fibré en droites complexes muni d'une structure réelle sur $(\surg,c_\sur)$, que l'on suppose de partie réelle non vide. Fixons une structure complexe $J\in\RR J(\surg)$. Prenons un opérateur de Cauchy-Riemann réel $\DB\in\ropj$. Notons $\NF$ le faisceau des sections holomorphes de $(N,\DB)$.

\begin{Definition}\label{transfoelem}
  La transformation élémentaire réelle négative en $x\in\RR\surg$ du faisceau $\NF$ est le faisceau localement libre $\NF_{-x}$ muni d'une structure réelle défini par la suite exacte
\[\label{transreel}
0\rightarrow \NF_{-x} \rightarrow \NF \xrightarrow{ev_x} N_x \rightarrow 0.\tag{$*$}
\]
Le faisceau $\NF_{-x}$ est de rang un, de degré $\deg(N)-1$, et sa partie réelle est de première classe de Stiefel-Whitney $w_1(\RR N) - [\RR\surg]_x^*$.

Si $\{z,\overline{z}\}$ est une paire de points complexes conjugués de $\surg$, alors la transformation élémentaire réelle négative en $\{z,\overline{z}\}$ du faisceau $\NF$ est le faisceau localement libre $\NF_{-z-\overline{z}}$ muni d'une structure réelle défini par la suite exacte
\[\label{transcompl}
0\rightarrow \NF_{-z-\overline{z}} \rightarrow \NF \xrightarrow{ev_{z,\overline{z}}} N_{z}\oplus N_{\overline{z}} \rightarrow 0.\tag{$**$}
\]
Le faisceau $\NF_{-z-\overline{z}}$ est de rang un, de degré $\deg(N)-2$, et sa partie réelle est de première classe de Stiefel-Whitney $w_1(\RR N)$.
\end{Definition}

Si $(\mathcal{F},c_{\mathcal{F}})$ est un faisceau holomorphe muni d'une structure réelle sur $(\surg,c_\sur)$, l'involution $c_{\mathcal{F}}$ induit des involutions sur $H^0(\surg,\mathcal{F})$ et $H^1(\surg,\mathcal{F})$ et on notera avec des indices $+1$ leurs espaces propres associés à la valeur propre $1$. On écrira aussi $\ddet(\mathcal{F}) = \Lambda^{\max}\left(H^0(\surg,\mathcal{F})_{+1}\right)\otimes \Lambda^{\max}\left(H^1(\surg,\mathcal{F})_{+1}\right)^*$.

\begin{Lemme}\label{lemtrans}
Si $\NF_{-x}$ est la transformation élémentaire réelle négative en $x\in\RR\surg$ de $\NF$, alors on a un isomorphisme canonique
\[
\ddet(\NF) = \ddet(\NF_{-x})\otimes \RR(N_x).
\]  
Si $\NF_{-z-\overline{z}}$ est la transformation élémentaire réelle négative en $\{z,\overline{z}\}$ de $\NF$, alors on a un isomorphisme canonique
\[
\ddet(\NF) = \ddet(\NF_{-z-\overline{z}})\otimes \det(\Delta_z),
\]  
où
\[
\Delta_z = \{(v,c_N(v)) \in N_z\oplus N_{\overline{z}}\}.
\]
\end{Lemme}

\begin{proof}
  Ces deux isomorphismes proviennent directement des suites exactes longues de cohomologie associées aux suites exactes (\ref{transreel}) et (\ref{transcompl}).
\end{proof}

Après une transformation élémentaire négative en $x\in\RR\surg$ du faisceau $\NF$ nous retrouvons un fibré en droites holomorphe réel $(N_{-x},\DB_{-x},c_{N,-x})$ dont le faisceau des sections holomorphes est naturellement isomorphe au faisceau $\NF_{-x}$ de la façon suivante. Prenons une carte locale holomorphe $(U,\xi)$ de $\surg$ centrée en $x$ et $\ZZ/2\ZZ$-équivariante. Le fibré $N_{-x}$ est alors obtenu en recollant $N_{|(\surg\setminus\{x\})}$ et $N_{|U}$  par l'application de recollement
\[\label{recoll}
\begin{array}{c c c}
  N_{|U\cap(\surg\setminus\{x\})} & \rightarrow & N_{|(\surg\setminus\{x\})\cap U}\\
  (\xi,v) & \mapsto & (\xi, \xi v).
\end{array}\tag{$*$}
\] 
Nous pouvons faire de même avec une transformation en une paire complexe conjuguée $\{z,\overline{z}\}$ pour obtenir un fibré holomorphe $(N_{-z-\overline{z}},\DB_{-z-\overline{z}},c_{N,-z-\overline{z}})$.

Une fois le fibré $N_{-x}$ fixé, nous avons deux inclusions naturelles $\ZZ/2\ZZ$-équivariantes (voir aussi \cite{Shev})
\[
\begin{array}{c}
  i_{-x} : L^{k,p}(\surg,N_{-x}) \rightarrow L^{k,p}(\surg,N)_{-x}\\
  j_{-x} : L^{k-1,p}(\surg,\Lambda^{0,1}\surg\otimes N_{-x}) \rightarrow L^{k-1,p}(\surg,\Lambda^{0,1}\surg\otimes N),
\end{array}
\]
où $k\geq 1$, $p>2$ et $L^{k,p}(\surg,N)_{-x}$ désigne les éléments de $L^{k,p}(\surg,N)$ qui s'annulent en $x$. On remarque que si $\DB'$ est un autre opérateur de Cauchy-Riemann sur $N$, alors sa restriction à $L^{k,p}(\surg,N_{-x})$ à travers $i_{-x}$ est à valeurs $L^{k-1,p}(\surg,\Lambda^{0,1}\surg\otimes N_{-x})$. En effet, $s\in L^{k,p}(\surg,N)_{-x}$ est dans l'image de $i_{-x}$ si et seulement si il existe $s'\in L^{k,p}(\surg,N)$ tel qu'on ait $s = \xi s'$ au voisinage de $x$. Donc $\DB'(s) = \DB'(\xi s') = \xi\DB'(s') \in j_{-x}\left(L^{k-1,p}(\surg,\Lambda^{0,1}\surg\otimes N_{-x})\right)$ car $\xi$ est holomorphe. On a l'analogue pour les transformations en deux points complexes conjugués. On obtient donc deux applications
\[
\begin{array}{c}
t_{-x} : \ropj \rightarrow \ropj[N_{-x}]\\
\text{et } t_{-z-\overline{z}} : \ropj \rightarrow \ropj[N_{-z-\overline{z}}].
\end{array}
\]
 De plus, l'injection $i_{-x}$ (resp. $i_{-z-\overline{z}}$) induit un isomorphisme entre le faisceau des sections holomorphes de $(N_{-x},t_{-x}(\DB'))$ (resp. de $(N_{-z-\overline{z}},t_{-z-\overline{z}}(\DB'))$) et le faisceau des sections holomorphes de $(N,\DB')$ qui s'annulent en $x$ (resp. en $z$ et $\overline{z}$).

Ces fibrés $N_{-x}$ et $N_{-z-\overline{z}}$ ne sont pas canoniquement définis. Toutefois les isomorphismes $i_{-x}$ et $i_{-z-\overline{z}}$ nous donnent avec le Lemme \ref{lemtrans} (voir aussi \cite{Shev}, Lemma 2.4.1),

\begin{Lemme}\label{lemshev}
 Soit $(N_{-x},c_{N,-x})$ (resp. $(N_{-z-\overline{z}},c_{N,-z-\overline{z}})$) un fibré vectoriel complexe muni d'une structure réelle associé à une transformation élémentaire réelle négative de $N$ en $x\in\RR\surg$ (resp. en $\{z,\overline{z}\}$). Nous avons des isomorphismes canoniques 
\[
\begin{array}{c}
\Det = \left((t_{-x})^*\Det[N_{-x}]\right)\otimes \RR(N_x)\\
\Det = \left((t_{-z-\overline{z}})^*\Det[N_{-z-\overline{z}}]\right)\otimes \det(\Delta_z).
\end{array}
\]\qed
\end{Lemme}

Nous pouvons de plus déterminer la fibre de $N_{-x}$ (resp. $N_{-z-\overline{z}}$) au-dessus de $x$ (resp. $z$ et $\overline{z}$).

\begin{Lemme}\label{lemfibre}
  Les injections $i_{-x}$ et $i_{-z-\overline{z}}$ induisent naturellement des isomorphismes
\[
\begin{array}{c}
(N_{-x})_x = T_x^*\surg\otimes N_x,\\
(N_{-z-\overline{z}})_z = T_z^*\surg\otimes N_z ,\\
 \text{et } (N_{-z-\overline{z}})_{\overline{z}} = T_{\overline{z}}^*\surg\otimes N_{\overline{z}}).\\
\end{array}
\]
\end{Lemme}

\begin{proof}
Nous faisons la démonstration dans le cas d'une transformation élémentaire en un point réel, l'autre cas étant analogue. Fixons un opérateur $\DB\in\ropj$. Prenons un élément $v$ dans $(N_{-x})_x$ et considérons une section locale holomorphe $s_v$ de $(N_{-x},t_{-x}(\DB))$ définie sur un voisinage de $x$ et telle que $s_v(x) = v$. Alors $i_{-x}(s_v)$ est une section locale holomorphe de $(N,\DB)$ qui s'annule au point $x$. On peut donc définir sa dérivée en $x$ $\nabla_x(i_{-x}(s_v)) = pr_x\circ \d_x(i_{-x}(s_v)) \in T^*_x\surg\otimes N_x$, où $pr_x : T_{(x,0)} N \rightarrow N_x$ est la projection parallèlement au tangent à la section nulle de $N$.

Vérifions tout d'abord que si $\DB'\in\ropj$ est un autre opérateur et $s'_v$ est une section locale holomorphe de $(N_{-x},t_{-x}(\DB'))$ valant $v$ en $x$, on a $\nabla_x(i_{-x}(s_v)) = \nabla_x(i_{-x}(s'_v))$. Supposons dans un premier temps que $v$ est non nul. Il existe alors une fonction $f$ à valeurs complexes définie sur un petit voisinage de $x$, telle que $s'_v = f s_v$. En particulier $f(x) = 1$. De plus, $\d_x(i_{-x}(s'_v)) = \d_x(f i_{-x}(s_v)) = i_{-x}(s_v)(x)\d_x(f) + f(x)\d_x(i_{-x}(s_v)) = \d_x(i_{-x}(s_v))$. Donc $\nabla_x(i_{-x}(s_v)) = \nabla_x(i_{-x}(s'_v))$. Si $v$ est nul, une section $s_0$ s'écrit $f s$ où $s$ est une trivialisation locale de $N_{-x}$ au voisinage de $x$ et $f$ est une fonction à valeurs complexes définie sur un petit voisinage de $x$ et s'annulant en $x$. On trouve alors $\d_x(i_{-x}(s_0)) = 0$, donc $\nabla_{x}(i_{-x}(s_0)) = 0$.

On obtient donc une application linéaire
\[
\label{nabl}
\begin{array}{c c c}
  (N_{-x})_x & \rightarrow & T_x^*\surg\otimes N_x\\
   v        & \mapsto     & \nabla_x (i_{-x}s_v),
\end{array}
\tag{$\nabla$}
\]
bien définie indépendamment du choix de l'opérateur et de $s_v$.
Vérifions maintenant que celle-ci est injective. Prenons $\tilde{s}$ une section locale holomorphe de $(N,\DB)$ ne s'annulant pas au voisinage de $x$. Alors, il existe $s$ une section locale de $(N_{-x},t_{-x}(\DB))$ telle que $i_{-x}(s) = \xi \tilde{s}$. Donc $\nabla_x(i_{-x}(s)) = pr_x(\d_x\xi \tilde{s})\neq 0$. Comme $(N_{-x})_{x}$ est de dimension $1$, on en conclut que l'application (\ref{nabl}) est injective. Par égalité des dimensions, c'est un isomorphisme.
\end{proof}

\begin{proof}[Démonstration de la Proposition \ref{enonceaction}]
Fixons une structure complexe $J\in\RR J(\surg)$. Prenons un point $x_i$, $1\leq i\leq k_-$, sur chaque composante connexe de $\RR \surg$ où $\RR N$ est non-orientable, et choisissons $l = \dd\frac{|\deg(N)-k_-|}{2}$ points $z_1,\ldots,z_l\in\surg\setminus\RR \surg$. Supposons tout d'abord que $\deg(N)-k_-$ est positif. En effectuant des tranformations élémentaires réelles négatives en $\underline{x} = (x_1,\ldots,x_{k_-},z_1,c_\sur(z_1),\ldots,z_l,c_\sur(z_l))$, on obtient un fibré en droites complexes muni d'une structure réelle $(N',c_{N'})$ de degré zéro et dont toutes les composantes réelles sont orientables. D'après le Lemme \ref{lemshev} nous avons un isomorphisme canonique
\[\label{sisi}
\Det = ((t_{-\underline{x}})^*\Det[N'])\otimes \det\left(\bigoplus_{i=1}^{k_-}\RR N_{x_i}\right) \otimes \det\left(\bigoplus_{i=1}^l\Delta_{z_i}\right).\tag{$*$}
\]
Fixons $f$ un des automorphismes considérés dans l'énoncé. D'après l'isomorphisme (\ref{sisi}) le signe de l'action de $f$ sur les orientations de $\Det$ est donné par le signe de son action sur les orientations de $\Det[N']$ si $f$ est positive sur un nombre pair de composantes non-orientables de $\RR N$ et par son opposé sinon. Enfin, en choisissant un isomorphisme entre $(N',c_{N'})$ et $(\TCC,conj)$ on voit que le signe de l'action de $f$ sur les orientations de $\Det[N']$ est le même que sur les orientations de $\Det[\TCC]$.

Vérifions pour finir les différents cas.
\begin{itemize}
\item Si $(\surg,c_\sur)$ est séparante, les automorphismes $f_k,g_k,\ldots,f_{k+m-1},g_{k+m-1}$ sont tous positifs sur $\RR \surg$ et d'après la Proposition \ref{castr} ils préservent les orientations de $\Det[\TCC]$. Ils préservent donc les orientations de $\Det$.
\item Si $(\surg,c_\sur)$ n'est pas séparante, un des automorphismes $f_j$, $j\in\{k,\ldots,g\}$ est positif sur toutes les composantes de $\RR\surg$, et d'après la Proposition \ref{castr} il renverse les orientations de $\Det[\TCC]$. Il renverse donc les orientations de $\Det$.
\item Un des automorphismes $f_j$, $j\in\{0,\ldots,k-1\}$ est positif sur toutes les composantes de $\RR\surg$ sauf sur la $j$-ième, et d'après la Proposition \ref{castr} il renverse les orientations de $\Det[\TCC]$. Il préserve donc les orientations de $\Det$ si et seulement si $(\RR N)_j$ n'est pas orientable.
\end{itemize}
Ce qui conclut la démonstration de la Proposition \ref{enonceaction}.
\end{proof}

\begin{proof}[Démonstration du Théorème \ref{enoncesn}]
Ce Théorème découle maintenant de la Proposition \ref{enonceaction} de l'Exemple \ref{sncalcul} et du Lemme \ref{lemmodsig}.
\end{proof}

\subsection{Fibré déterminant sur le groupe de Picard}\label{picard}

Nous donnons maintenant une application du Théorème \ref{enoncesn}.

Commençons par faire quelques remarques concernant le groupe de Picard réel (voir par exemple \cite{grossharris}). Soit $(\surg,c_\sur)$ une surface de Riemann réelle de genre $g\geq 1$. Nous supposerons que sa partie réelle est non vide.

Notons $\RR\mathcal{L}(\surg)$ l'ensemble des fibrés en droites holomorphes réels sur $(\surg,c_\sur)$. Nous avons sur $\RR\mathcal{L}(\surg)$ un fibré en droites réelles $\ddet$ de fibre $\Lambda^{\max}\left(H^0(\surg,L)_{+1}\right)\otimes\Lambda^{\max}\left(H^1(\surg,L)_{+1}\right)^*$ au-dessus de $(L,c_L)$. 

Considérons $\RR\Pic(\surg)$ le groupe de Picard réel de $(\surg,c_\sur)$. Nous avons une application naturelle $\RR\mathcal{L}(\surg) \rightarrow \RR\Pic(\surg)$ qui consiste à associer à un fibré en droites holomorphe réel sa classe d'isomorphisme. Toutefois, comme le montre la Remarque \ref{parite}, un lacet dans $\RR \Pic(\surg)$ ne correspond pas à une déformation à isomorphisme près de fibrés en droites holomorphes réels.

 Nous pouvons cependant associer à un tel lacet $([L]_t)_{t\in[0,1]}$ un élément de $F^-$, de la façon suivante. Choisissons un chemin d'opérateurs de Cauchy-Riemann réels $(\DB_t)_{t\in[0,1]}$ sur un fibré en droites complexe muni d'une structure réelle $(L,c_L)$ fixé se projetant dans $\RR\Pic(\surg)$ sur le lacet $([L]_t)_{t\in[0,1]}$. En prenant un automorphisme réel $f\in\raut{L}$ tel que $f^*\DB_1 = \DB_0$ nous obtenons, d'après le Lemme \ref{lemmodsig}, un élément $\ind_2(f)$ de $F^-$. D'autre part, si $g\in\raut{L}$ est un autre automorphisme vérifiant $g^*\DB_1 = \DB_0$, alors $g$ est homotope à plus ou moins l'automorphisme $f$ (car l'automorphisme $-1$ fixe les opérateurs de Cauchy-Riemann). Cette construction nous fournit un morphisme de \og monodromie \fg
\[
\mu : H_1(\RR\Pic(\surg),\ZZ/2\ZZ)\rightarrow F^-.
\]

D'autre part, nous introduisons comme Gross et Harris (voir \cite{grossharris}) le morphisme
\[
\begin{array}{c c c c}
\deg\times w_1 : & \RR\Pic(\surg) &\rightarrow & \ZZ\times\hczdeux{\RR\surg} \\
                 & (L,c_L)        &\mapsto     & (\deg(L), w_1(\RR L)).
\end{array}
\]
Pour $d\in\ZZ$ et $w\in \hczdeux{\RR\surg}$ tels que $w([\RR\surg]) = d\mod 2$, nous noterons $\RR\Pic_w^d(\surg) = (\deg\times w_1)^{-1}(d,w)$. Remarquons que 
\[
\RR\Pic(\surg) = \bigsqcup_{
  \begin{subarray}{l}
    (d,w)\in\ZZ\times\hczdeux{\RR\surg}\\w([\RR\surg]) = d\mod 2
  \end{subarray}
}\RR\Pic_w^d(\surg).
\]

Comme nous l'avons noté dans la Remarque \ref{snrem}, lorsque $d$ et $w([\RR\surg])$ sont congrus à $g-1$ modulo $2$, le morphisme $\AA^w$ est bien défini sur $F^-$.

Nous pouvons maintenant énoncer le résultat suivant.

\begin{Theoreme}\label{picp}
  Soit $(\surg,c_\sur)$ une surface de Riemann réelle de partie réelle
  non vide. Soit $d\in\ZZ$ et $w\in\hczdeux{\RR\surg}$ tels que $d = w([\RR\surg]) = g-1 \mod 2$. Le fibré $\ddet$ sur l'ensemble des fibrés en droites holomorphes réels de degré $d$ et de partie réelle de première classe de Stiefel-Whitney $w$ descend sur $\RR\Pic_w^d(\surg)$ en un fibré noté $\ddet_w^d$. De plus, la première classe de Stiefel-Whitney de ce fibré est donnée par:
\[
w_1(\ddet_w^d) = \AA^w\circ\mu.
\]
\end{Theoreme}

Comme nous l'avons déjà vu au \S \ref{enonce}, le fibré $\ddet$ ne descend pas en un fibré sur $\RR\Pic(\surg)$ tout entier. En effet, si $(L,c_L)$ est un fibré en droites holomorphe réel, $-1$ en est un automorphisme holomorphe. Or comme nous l'avons déjà remarqué, $-1$ préserve les orientations de $\Det[L]$ si et seulement si $\deg(L) + 1 - g$ est pair. Ainsi, le fibré $\ddet$ ne descend que sur les composantes de $\RR\Pic(\surg)$ formées des fibrés dont le degré est de même parité que $g-1$, et le Théorème \ref{picp} donne sa première classe de Stiefel-Whitney.
 
Pour étudier l'autre cas, nous fixons un point $p\in \RR\surg$ et considérons le fibré $\ddet_p$ sur $\RR\mathcal{L}(\surg)$ de fibre $\Lambda^{\max}\left(H^0(\surg,L)_{+1}\right)\otimes\Lambda^{\max}\left(H^1(\surg,L)_{+1}\right)^*\otimes \RR L_p$ au-dessus de $(L,c_L)$.

Pour $w\in\hczdeux{\RR\surg}$, notons $\tilde{w}_p\in\hczdeux[g-1]{\RR\surg}$ l'unique classe égale à $w$ partout sauf peut-être sur la composante de $\RR\surg$ contenant $p$. Nous pouvons maintenant énoncer la suite du Théorème \ref{picp}.

\begin{Theoreme}\label{pic}
  Soit $(\surg,c_\sur)$ une surface de Riemann réelle de partie réelle
  non vide et soit $p$ un point réel de $\surg$. Soit $d\in\ZZ$ et $w\in\hczdeux{\RR\surg}$ tels que $d = w([\RR\surg]) = g \mod 2$. Le fibré $\ddet_p$ sur l'ensemble des fibrés en droites holomorphes réels de degré $d$ et de partie réelle de première classe de Stiefel-Whitney $w$ descend sur $\RR\Pic_w^d(\surg)$ en un fibré noté $\ddet_{p,w}^d$. De plus, la première classe de Stiefel-Whitney de ce fibré est donnée par:
\[
w_1(\ddet_{p,w}^d) = \AA^{\tilde{w}_p}\circ\mu.
\]
\end{Theoreme}

Avant d'entamer la démonstration notons d'après la Remarque \ref{remprod} qu'en déplaçant le point $p$ sur une autre composante connexe de $\RR \surg$, on ajoute $c\circ\mu$ à la classe du Théorème \ref{pic}, où $c$ est une courbe simple globalement invariante par $c_\sur$ et reliant les deux composantes connexes de $\RR\surg$ en question, et pour tout $a\in H_1(\RR \Pic(\surg),\ZZ/2\ZZ)$, $c\circ\mu(a) = (\mu(a))(c)$. 

Soulignons d'autre part que si le fibré $\ddet_p$ est défini comme un produit tensoriel de $\ddet$ avec $\RR L_p$ sur $\RR\mathcal{L}(\surg)$, ce n'est plus le cas pour $\ddet_{w,p}^d$ sur $\RR\Pic_w^d(\surg)$. 

\begin{proof}[Démonstration des Théorèmes \ref{picp} et \ref{pic}]
Prenons une base symplectique réelle de $\hhz{\surg}$ avec $p$ comme point base. Choisissons un fibré vectoriel complexe $(N,c_N)$ de rang un sur $(\surg,c_\sur)$. Posons $d = \deg(N)$ et $w=w_1(\RR N)$. Nous savons qu'à chaque opérateur de Cauchy-Riemann sur $N$ correspond une unique structure holomorphe et réciproquement (voir \cite{koba}). Nous avons donc un isomorphisme
\[
\RR\Pic_w^d(\surg) \cong \ropj/\raut{N}
\]
 De plus, le tiré en arrière du fibré $\Det$ par cet isomorphisme est le fibré $\ddet$. D'autre part, le lacet dans $\ropj/\raut{N}$ engendré par un automorphisme $f\in\raut{N}$ associé à une courbe $a\in\hhz{\surg}_{+1}$ est envoyé sur un lacet de monodromie $a^{\pd}\in F^-$. Les Théorèmes \ref{enoncea} et \ref{enoncesn} nous permettent alors de conclure.
\end{proof}

\begin{Exple} 
  Prenons par exemple le cas du genre $g = 1$. Chaque composante $\RR \Pic^d_w(\sur_1)$ est alors un cercle auquel est associée une monodromie $\pm f$, où $f$ est une des deux fonctions de la famille $\mathcal{B}$ construite au \S \ref{modsig}. Comme nous supposons que la partie réelle de $\sur_1$ est non vide, celle-ci a une ou deux composantes.

 Lorsqu'elle en a deux, la courbe est séparante. Si $d$ est pair, alors $w$ a la même valeur sur chaque composante de $\RR\sur_1$. Si nous appliquons alors le Théorème \ref{picp} en utilisant le Théorème \ref{enoncea} et l'Exemple \ref{sncalcul}, nous avons
\[
w_1(\ddet_w^d)([\RR\Pic_w^d(\sur_1)]) = \AA^w(f) = 1-w([\RR\sur_1]_0)
\]
et nous voyons que le fibré déterminant est orientable exactement sur les composantes du groupe de Picard où $w\neq 0$. Si $d$ est impair et $p$ un point sur la composante où $w$ est non nulle, alors $\tilde{w}_p = 0$ et de même que ci-dessus, le fibré $\ddet_{p,w}^d$ n'est pas orientable. Si $p$ est sur la composante où $w$ est nulle, alors $\ddet_{p,w}^d$ est orientable.

Lorsque $\RR\sur_1$ n'a qu'une seule composante, la courbe n'est pas séparante. En raisonnant comme précédemment, lorsque $d$ est pair, et le fibré déterminant n'est orientable sur aucune des composantes du groupe de Picard. Lorsque $d$ est impair, $\tilde{w}_p$ est nulle et $\ddet_{p,w}^d$ n'est pas orientable.
\end{Exple}

\begin{Rem}
  Nous avons en fait calculé la première classe de Stiefel-Whitney du déterminant de la cohomologie de la partie réelle des fibrés universels de Poincaré décris par Biswas et Hurtubise dans \cite{bishurt} lorsque la courbe $(\surg,c_\sur)$ est de partie réelle non vide.
\end{Rem}

\bibliographystyle{plain-fr}%%plain-fr marche pas partout

\begin{thebibliography}{}
\expandafter\ifx\csname fonteauteurs\endcsname\relax
\def\fonteauteurs{\scshape}\fi

\end{thebibliography}


\begin{thebibliography}{10}
\expandafter\ifx\csname fonteauteurs\endcsname\relax
\def\fonteauteurs{\scshape}\fi

\bibitem{atibott}
M.~F. \bgroup\fonteauteurs\bgroup Atiyah\egroup\egroup{},
  R.~\bgroup\fonteauteurs\bgroup Bott\egroup\egroup{} et
  A.~\bgroup\fonteauteurs\bgroup Shapiro\egroup\egroup{} :
\newblock Clifford modules.
\newblock {\em Topology}, 3(suppl. 1)\string:\penalty500\relax 3--38, 1964.

\bibitem{atiyah}
Michael~F. \bgroup\fonteauteurs\bgroup Atiyah\egroup\egroup{} :
\newblock Riemann surfaces and spin structures.
\newblock {\em Ann. Sci. \'Ecole Norm. Sup. (4)}, 4\string:\penalty500\relax
  47--62, 1971.

\bibitem{audin}
Michèle \bgroup\fonteauteurs\bgroup Audin\egroup\egroup{} et Jacques
  \bgroup\fonteauteurs\bgroup Lafontaine\egroup\egroup{}, \'editeurs.
\newblock {\em Holomorphic curves in symplectic geometry}, volume 117 de {\em
  Prog. Math.} Birkha\"user, 1994.

\bibitem{bisw}
Indranil \bgroup\fonteauteurs\bgroup Biswas\egroup\egroup{}, Johannes
  \bgroup\fonteauteurs\bgroup Huisman\egroup\egroup{} et Jacques
  \bgroup\fonteauteurs\bgroup Hurtubise\egroup\egroup{} :
\newblock The moduli space of stable vector bundles over a real algebraic
  curve.
\newblock {\em Math. Ann.}, 347(1)\string:\penalty500\relax 201--233, 2010.

\bibitem{bishurt}
Indranil \bgroup\fonteauteurs\bgroup Biswas\egroup\egroup{} et Jacques
  \bgroup\fonteauteurs\bgroup Hurtubise\egroup\egroup{} :
\newblock Universal bundle over the reals.
\newblock Trans. Amer. Math. Soc., posted on July 25, 2011, PII S
  0002-9947(2011)05345-6 (to appear in print).

\bibitem{article2}
R\'emi \bgroup\fonteauteurs\bgroup Cr\'etois\egroup\egroup{} :
\newblock Orientabilité d'espaces de modules de courbes réelles.
\newblock En pr\'eparation.

\bibitem{degitkha}
Alex \bgroup\fonteauteurs\bgroup Degtyarev\egroup\egroup{}, Ilia
  \bgroup\fonteauteurs\bgroup Itenberg\egroup\egroup{} et Viatcheslav
  \bgroup\fonteauteurs\bgroup Kharlamov\egroup\egroup{} :
\newblock On the number of components of a complete intersection of real
  quadrics.
\newblock Preprint math.AG/0806.4077v2, 2008.

\bibitem{fooo}
Kenji \bgroup\fonteauteurs\bgroup Fukaya\egroup\egroup{}, Yong-Geun
  \bgroup\fonteauteurs\bgroup Oh\egroup\egroup{}, Hiroshi
  \bgroup\fonteauteurs\bgroup Ohta\egroup\egroup{} et Kaoru
  \bgroup\fonteauteurs\bgroup Ono\egroup\egroup{} :
\newblock {\em Lagrangian intersection {F}loer theory: anomaly and obstruction.
  {P}art {II}}, volume~46 de {\em AMS/IP Studies in Advanced Mathematics}.
\newblock American Mathematical Society, Providence, RI, 2009.

\bibitem{grossharris}
Benedict~H. \bgroup\fonteauteurs\bgroup Gross\egroup\egroup{} et Joe
  \bgroup\fonteauteurs\bgroup Harris\egroup\egroup{} :
\newblock Real algebraic curves.
\newblock {\em Ann. Sci. \'Ecole Norm. Sup. (4)},
  14(2)\string:\penalty500\relax 157--182, 1981.

\bibitem{harris}
Joe \bgroup\fonteauteurs\bgroup Harris\egroup\egroup{} :
\newblock Theta-characteristics on algebraic curves.
\newblock {\em Trans. Amer. Math. Soc.}, 271(2)\string:\penalty500\relax
  611--638, 1982.

\bibitem{hoflizsik}
Helmut \bgroup\fonteauteurs\bgroup Hofer\egroup\egroup{}, V{\'e}ronique
  \bgroup\fonteauteurs\bgroup Lizan\egroup\egroup{} et Jean-Claude
  \bgroup\fonteauteurs\bgroup Sikorav\egroup\egroup{} :
\newblock On genericity for holomorphic curves in four-dimensional
  almost-complex manifolds.
\newblock {\em J. Geom. Anal.}, 7(1)\string:\penalty500\relax 149--159, 1997.

\bibitem{ivshev}
S.~\bgroup\fonteauteurs\bgroup Ivashkovich\egroup\egroup{} et
  V.~\bgroup\fonteauteurs\bgroup Shevchishin\egroup\egroup{} :
\newblock Structure of the moduli space in a neighborhood of a cusp-curve and
  meromorphic hulls.
\newblock {\em Invent. Math.}, 136(3)\string:\penalty500\relax 571--602, 1999.

\bibitem{john}
Dennis \bgroup\fonteauteurs\bgroup Johnson\egroup\egroup{} :
\newblock Spin structures and quadratic forms on surfaces.
\newblock {\em J. London Math. Soc. (2)}, 22(2)\string:\penalty500\relax
  365--373, 1980.

\bibitem{kirby}
R.~C. \bgroup\fonteauteurs\bgroup Kirby\egroup\egroup{} et L.~R.
  \bgroup\fonteauteurs\bgroup Taylor\egroup\egroup{} :
\newblock {${\rm Pin}$} structures on low-dimensional manifolds.
\newblock \emph{In} {\em Geometry of low-dimensional manifolds, 2 ({D}urham,
  1989)}, volume 151 de {\em London Math. Soc. Lecture Note Ser.}, pages
  177--242, Cambridge, 1990. Cambridge Univ. Press.

\bibitem{koba}
Shoshichi \bgroup\fonteauteurs\bgroup Kobayashi\egroup\egroup{} :
\newblock {\em Differential geometry of complex vector bundles}, volume~15 de
  {\em Publications of the Mathematical Society of Japan}.
\newblock Princeton University Press, Princeton, NJ, 1987.
\newblock Kan{\^o} Memorial Lectures, 5.

\bibitem{Lawson}
H.~Blaine \bgroup\fonteauteurs\bgroup Lawson\egroup\egroup{}, Jr. et
  Marie-Louise \bgroup\fonteauteurs\bgroup Michelsohn\egroup\egroup{} :
\newblock {\em Spin geometry}, volume~38 de {\em Princeton Mathematical
  Series}.
\newblock Princeton University Press, Princeton, NJ, 1989.

\bibitem{MDS}
Dusa \bgroup\fonteauteurs\bgroup McDuff\egroup\egroup{} et Dietmar
  \bgroup\fonteauteurs\bgroup Salamon\egroup\egroup{} :
\newblock {\em {$J$}-holomorphic curves and symplectic topology}, volume~52 de
  {\em American Mathematical Society Colloquium Publications}.
\newblock American Mathematical Society, Providence, RI, 2004.

\bibitem{nat}
S.~M. \bgroup\fonteauteurs\bgroup Natanzon\egroup\egroup{} :
\newblock Finite groups of homeomorphisms of surfaces, and real forms of
  complex algebraic curves.
\newblock {\em Trudy Moskov. Mat. Obshch.}, 51\string:\penalty500\relax 3--53,
  258, 1988.

\bibitem{Natanzon}
S.~M. \bgroup\fonteauteurs\bgroup Natanzon\egroup\egroup{} :
\newblock {\em Moduli of {R}iemann surfaces, real algebraic curves, and their
  superanalogs}, volume 225 de {\em Translations of Mathematical Monographs}.
\newblock American Mathematical Society, Providence, RI, 2004.
\newblock Translated from the 2003 Russian edition by Sergei Lando.

\bibitem{Shev}
Vsevolod \bgroup\fonteauteurs\bgroup Shevchishin\egroup\egroup{} :
\newblock Pseudoholomorphic curves and the symplectic isotopy problem.
\newblock Preprint math.SG/0010262, 2000.

\bibitem{Solomon}
Jake \bgroup\fonteauteurs\bgroup Solomon\egroup\egroup{} :
\newblock {\em Intersection theory on the moduli space of holomorphic curves
  with lagrangian boundary conditions}.
\newblock Th\`ese de doctorat, MIT, 2008.

\bibitem{viterbo}
Claude \bgroup\fonteauteurs\bgroup Viterbo\egroup\egroup{} :
\newblock Symplectic real algebraic geometry.
\newblock Non publié, 1999.

\bibitem{wel2}
Jean-Yves \bgroup\fonteauteurs\bgroup Welschinger\egroup\egroup{} :
\newblock Real structures on minimal ruled surfaces.
\newblock {\em Comment. Math. Helv.}, 78(2)\string:\penalty500\relax 418--446,
  2003.

\bibitem{wel1}
Jean-Yves \bgroup\fonteauteurs\bgroup Welschinger\egroup\egroup{} :
\newblock Invariants of real symplectic 4-manifolds and lower bounds in real
  enumerative geometry.
\newblock {\em Invent. Math.}, 162(1)\string:\penalty500\relax 195--234, 2005.

\end{thebibliography}
\addcontentsline{toc}{part}{\hspace*{6mm}Références}

Université de Lyon; CNRS; Université Lyon 1; Institut Camille Jordan.

\end{document}